\documentclass[11pt, oneside]{amsart}   	
\usepackage[
backend=bibtex,
style=numeric,
backref=true,
]{biblatex}
\usepackage[margin = 0.75in]{geometry}                		
\usepackage{graphicx}				

\DeclareFontEncoding{FMS}{}{}
\DeclareFontSubstitution{FMS}{futm}{m}{n}
\DeclareFontEncoding{FMX}{}{}
\DeclareFontSubstitution{FMX}{futm}{m}{n}
\DeclareSymbolFont{fouriersymbols}{FMS}{futm}{m}{n}
\DeclareSymbolFont{fourierlargesymbols}{FMX}{futm}{m}{n}
\DeclareMathDelimiter{\VERT}{\mathord}{fouriersymbols}{152}{fourierlargesymbols}{147}

\usepackage[english]{babel}
\usepackage{amsthm}
\usepackage{bbm}	
\usepackage{url}

\usepackage[
colorlinks=true]{hyperref} 
\usepackage{ifthen}
\usepackage{tikz}
\usetikzlibrary{decorations.pathreplacing}
\usepackage{tikz-cd}

\allowdisplaybreaks

\usepackage{blindtext}
\usepackage{scrextend}
\addtokomafont{labelinglabel}{\sffamily}

\usepackage{amssymb}
\usepackage{amsmath}
\newcommand{\R}{\mathbb{R}}

\newcommand{\K}{\mathcal{K}}
\newcommand{\T}{\mathbb{T}}
\newcommand{\C}{\mathbb{C}}
\newcommand{\E}{\mathbb{E}}
\newcommand{\A}{\mathcal{A}}
\newcommand{\1}{\mathbbm{1}}
\newcommand{\Z}{\mathbb{Z}}
\newcommand{\N}{\mathbb{N}}
\newcommand{\real}{\operatorname{Re}}

\newcommand{\norm}[1]{\left\lVert#1\right\rVert}

\newcommand{\floor}[1]{\lfloor#1\rfloor}

\newtheorem{theorem}{Theorem}[section]
\newtheorem{corollary}[theorem]{Corollary}
\newtheorem{lemma}[theorem]{Lemma}
\newtheorem{proposition}[theorem]{Proposition}
\theoremstyle{definition}
\newtheorem{definition}[theorem]{Definition}

\newtheorem{remark}[theorem]{Remark}

\title{Higher order Wiener-Wintner systems: examples and applications}
\author{Idris Assani}
\address{Department of Mathematics, The University of North Carolina at Chapel Hill, 120 E Cameron Avenue, CB 3250
Chapel Hill, NC 27599-3250, USA}
\email{assani@email.unc.edu}
\urladdr{https://idrisassani.web.unc.edu/}

\author{Jacob Folks}
\address{Department of Mathematics, The University of North Carolina at Chapel Hill, 120 E Cameron Avenue, CB 3250
Chapel Hill, NC 27599-3250, USA}
\email{jfolks1@unc.edu}

\author{Ryo Moore}
\address{Department of Mathematics, Southern University of Science and Technology, 1088 Xueyuan Avenue, Shenzhen 518055,
P.R. China (formerly affiliated)}
\email{ryom314@gmail.com}
\urladdr{https://sites.google.com/view/ryomoore/} 

\subjclass[2020]{37A05, 37A30}

\keywords{Wiener-Wintner, multiple recurrence, return times theorem, ergodic averages}

\thanks{RM was partially supported by NSFC 12250710130 from China.}

\date{}							

\begin{document}
\maketitle


\begin{abstract}
    We will construct ``higher dimensional" versions of the Wiener-Wintner dynamical system that was originally studied by I. Assani in 2003. We will show that on these systems, we can provide very simple proofs of the a.e. convergence of the multiple recurrence averages, as well as the multiple recurrence return times averages. We will do so by obtaining a quantitative control of the multiple ergodic averages by extending the estimate for the double recurrence that was attained by J. Bourgain. We will also observe that this class of dynamical systems contains numerous examples that are not bounded by the standard classifications (e.g. entropy, mixing), such as Kolmogorov systems, classical skew products, as well as systems for which the a.e. convergence of multiple recurrence is not currently known. Along our way, we will also provide alternative characteristic of the Host-Kra-Ziegler factors from the point of view of the uniform Wiener-Wintner theorem.
\end{abstract}

\tableofcontents

\section{Introduction}\label{s:intro}

\subsection{Background and contents}\label{ss:BandC}

Let $(X, \mathcal{F}, \mu, T)$ be a measure-preserving system, let $J \in \N$, and let $f_1, \dots, f_J \in L^\infty(\mu)$. Multiple recurrence averages have the form
\begin{equation}\label{Favg}
\frac{1}{N}\sum_{n=1}^N \prod_{j=1}^J f_j \circ T^{jn} \, ,
\end{equation}
and arose from Furstenberg's dynamical proof of the Szemerédi's theorem in 1977 \cite{Fur77}. These averages may sometimes be referred to as ``multiple ergodic averages" or ``nonconventional ergodic averages". Norm convergence of multiple recurrence averages has been established by Host and Kra \cite{hostkra}, and through different methods by Ziegler \cite{ziegler07} (while more general results were later obtained by Tao \cite{Tao2008} and Walsh \cite{Walsh2012} for the case of several transformations). Partial results for pointwise convergence have been established in the cases of some weakly mixing systems by Assani \cite{assani98} and distal systems by Huang, Shao, and Ye \cite{Ye}, as well as by Donoso and Sun for several commuting transformations \cite{DS2018}. Krause, Mirek, and Tao recently obtained a pointwise convergence result for double recurrence with a non-linear polynomial \cite{KMT2022}. For more on the history and open problems regarding multiple recurrence, one may consult a survey paper by Frantzikinakis \cite{Frantz2016}.

Pointwise convergence in the case of double recurrence (i.e. the average (\ref{Favg}) with $J=2$) was established by Bourgain \cite{bourgain}. While his argument is complex, the proof begins by showing that, with relatively simple steps, the $L^2$-norm of double recurrence average can be controlled by the $L^2$-norm of the average
\begin{equation} \label{UWW}
\sup_{t \in \R} \left| \frac{1}{N} \sum_{n=1}^N e^{2\pi int} f \circ T^n \right| \end{equation}
(the same can be said if we replace the $L^2$-norm with $L^1$-norm, as we will discuss the detail in Appendix \ref{a:BB}). This average, known as the \textit{Wiener-Wintner average}, has been studied greatly in prior to Bourgain's work. In the work of N. Wiener and A. Wintner from 1944 \cite{WW}, the following result is announced.
\begin{theorem}[Wiener-Wintner ergodic theorem]\label{t:WW}
Let $(X, \mathcal{F}, \mu, T)$ be a measure-preserving system, and let $f \in L^1(\mu)$. There exists a set $X_f \in \mathcal{F}$ such that $\mu(X_f) = 1$, and for every $x \in X_f$ and for every $t \in \R$, the limit
\[ \lim_{N \to \infty} \frac{1}{N} \sum_{n=1}^N e^{2\pi int} f(T^nx) \]
exists.
\end{theorem}
We note that the set $X_f$ is \textit{independent} of $t$, which makes this result a nontrivial extension of Birkhoff's ergodic theorem.

There was a gap in the original proof of the Wiener-Wintner ergodic theorem. Fortunately, multiple complete proofs of this result have been produced since then. One may consult \cite{AssaniWW} for these proofs as well as subjects related to this result (and/or \cite{AssaniWWB} for more concise and modern survey that includes some updates in the subject).

A stronger variant of the Wiener-Wintner theorem, which is more relevant to study the average in (\ref{UWW}), is called the \textit{uniform Wiener-Wintner theorem}, which is stated as follows: 

\begin{theorem}[Uniform Wiener-Wintner ergodic theorem]\label{t:UWW}
Let $(X, \mathcal{F}, \mu, T)$ be an ergodic measure-preserving system, and let $f \in L^1(\mu)$. The following statements are equivalent.
\begin{enumerate}
    \item The function $f$ belongs to the orthogonal complement of the Kronecker factor.
    \item For $\mu$-a.e. $x \in X$, we have
    \[ \lim_{N \to \infty} \sup_{t \in \R} \left| \frac{1}{N} \sum_{n=1}^N e^{2\pi int}f(T^nx) \right| = 0 \, .  \]
\end{enumerate}
\end{theorem}
See \cite[Theorem 2.4]{AssaniWW} for a proof of the theorem for (1) implying (2); the converse can be obtained by using the spectral characterization of the dynamical theorem as well as Fubini's theorem (see the proof in \cite[Proposition 7.1(1)]{AssaniWW} for a relevant computation). We remark that we \textit{cannot} drop the ergodicity assumption from this theorem (see \cite[pp. 202-203]{AssaniWWB} for a counterexample).

By observing the relevance of the Wiener-Wintner averages, Assani has noticed that the proof of the double recurrence theorem can be simplified greatly \cite[Theorem 9]{assani} with the following additional assumption on the dynamical system: Let $(X, \mathcal{F}, \mu, T)$ be an ergodic measure-preserving system. Assume that there exists a dense (in $L^2$-norm) set of functions in the orthogonal complement of the Kronecker factor, such that there exists $\alpha > 0$ such that for every $f$ in this set, there exists a constant $C_f > 0$ such that for every $N \in \N$, we have
$$\left\Vert \sup_t \left| \frac{1}{N} \sum_{n=1}^N e^{2\pi i n t} f \circ T^n \right| \right\Vert_p \leq \frac{C_f}{N^\alpha} \, .$$
We say such system is a \textit{Wiener-Wintner dynamical system of power type $\alpha$}. It was originally studied in \cite{assani}, and some spectral properties of such system were studied in \cite{assani04}.

Another type of ergodic average that is closely related to the Wiener-Wintner average is the \textit{return times average}. Such average was initially studied by A. Brunel in his PhD thesis \cite{brunel66}. Let $(X, \mathcal{F}, \mu, T)$ be a measure-preserving system, and let $f \in L^1(\mu)$. By the Wiener-Wintner ergodic theorem and an immediate application of the spectral theorem, there exists a set $X_f \subset X$ such that $\mu(X_f) = 1$, and for every $x \in X_f$, and for every other measure-preserving system $(Y, \mathcal{G}, \nu, S)$ and every $g \in L^1(\nu)$, the limit
\[ \lim_{N \to \infty} \frac{1}{N} \sum_{n=1}^N f(T^nx) g \circ S^n \]
exists in $L^2(\nu)$-norm. The pointwise counterpart of this result also exists, but it is not as immediate as the $L^2$-convergence case. Such result, which we now call the \textit{return times theorem}, was obtained by Bourgain \cite{Bourgain1988}, while this proof was later simplified by himself, along with Furstenberg, Katznelson, and Ornstein \cite{BFKO}; we will call this argument ``the BFKO argument" for short.

\begin{theorem}[Return times theorem]\label{t:RTT}
Let $(X, \mathcal{F}, \mu, T)$ be a measure-preserving system, and let $f \in L^\infty(\mu)$. Then there exists a set $X_f \in \mathcal{F}$ such that $\mu(X_f) = 1$, and for every $x \in X_f$, for every other measure-preserving system $(Y, \mathcal{G}, \nu, S)$, for every $g \in L^\infty(\nu)$, and for $\nu$-a.e. $y \in Y$, the limit
\[ \lim_{N \to \infty} \frac{1}{N} \sum_{n=1}^N f(T^nx)g(S^ny) \]
exists.\footnote{An alternative proof of this theorem by Rudolph \cite{Rudolph1994} shows that the result holds if the H\"older duality is preserved, i.e. if $p \in [1, \infty]$, $q \in \R$ for which $1/p + 1/q \leq 1$, and $f \in L^p(\mu)$ and $g \in L^q(\nu)$. The problem regarding the ``break" of duality (i.e. $1/p + 1/q > 1$) was raised by Assani in 1990, and there has been some developments since then; see \cite[\S5]{ap2014} for more detail.}
\end{theorem}
Again, a novelty of the theorem is the existence of the set of full measure $X_f$ that is independent of the other system, which is not guaranteed from a simple application of the Birkhoff theorem on the product space. For more on the history and development of the return times theorem, we refer the readers to \cite{ap2014}.

While the techniques used in the BFKO argument is elementary (in a sense that one can read it with an elementary knowledge of ergodic theory), the proof itself is quite complicated and delicate. However, just as in the case of the double recurrence theorem, one can simplify the proof tremendously if the system $(X, \mathcal{F}, \mu, T)$ is assumed to be a Wiener-Wintner dynamical system; see \cite[\S5]{AssaniWWB}. 

Let $(X, \mathcal{F}, \mu, T)$ be a measure-preserving system, let $J, K \in \N$, and let $f_1, f_2, \ldots, f_J \in L^\infty(\mu)$. We would like to know if we can extend the return times theorem for multiple ergodic averages: Does there exist a set $X' \subset X$ such that for every $x \in X'$, for every other measure-preserving system $(Y, \mathcal{G}, \nu, S)$ and functions $g_1, g_2, \ldots, g_K \in L^\infty(\nu)$, and for $\nu$-a.e. $y \in Y$, the limit
\begin{equation}\label{mrtt}
     \lim_{N \to \infty} \frac{1}{N} \sum_{n=1}^N \prod_{j=1}^J f_j(T^{jn}x) \prod_{k=1}^K g_k(S^{kn}y)
\end{equation}
exists? The case $J=K=1$ is the return times theorem, and for some weakly mixing systems, Assani showed such convergence result for the case $K=1$ for any $J \in \N$ \cite{assani2000}. The case $J=2$ and $K=1$ was obtained by Zorin-Kranich using the convergence criteria obtained in the BFKO argument \cite{ZK2019}.

For $L^2(\nu)$ convergence of (\ref{mrtt}), Host and Kra proved the case for $J=1$ for any $K \in \N$ \cite[Theorem 2.25]{HK2009} (see \cite[Corollary 7.1]{EZ2013} for an extension of this result) using nilsequences. Later, Assani and Moore obtained the norm convergence result for the case $J=2$ for any $K \in \N$ \cite{AM2017} using the double recurrence Wiener-Wintner Theorem \cite{ADM2016}.

It is worth noting that there exists a weakly mixing system that is not a Wiener-Wintner dynamical system of power type $\alpha$; such system was obtained by Assani \cite[Theorem 7]{assani04}), as well as independently by Lesigne.

\subsection{Goals and outline}
The major goal of this paper is to extend the Wiener-Wintner dynamical system to the higher order cases. We will construct classes of dynamical systems for which one can obtain simple proofs for
\begin{enumerate}
    \item the pointwise convergence of multiple recurrence average for more than two functions, and
    \item the extension of the return times theorem for multiple ergodic averages.
\end{enumerate}
We will demonstrate that there are many interesting examples of such systems. In particular, as in a case for the original Wiener-Wintner dynamical systems, Kolmogorov systems and classical skew products are included in this class of systems. We will also show that these higher-order classes contain an example that has not been studied for the pointwise convergence of multiple ergodic averages previously. Furthermore, we will provide an alternative characterization of Host-Kra-Ziegler factors (see \S\ref{ss:ns_hkz} for definitions).

Here is a rough outline of the paper:
\begin{itemize}
    \item  We summarize some concepts and conventions that will be used throughout the paper in \S\ref{s:prelim}.
    \item In \S\ref{s:HWWDS}, we present the analogous decay condition to define a \textit{$k$-th order Wiener-Wintner dynamical system.}
    \item We provide examples of systems that satisfy the condition in \S\ref{s:ex}. In particular, we will show that Kolmogorov systems (\S\ref{ss:Ksys}) and classical skew products (\S\ref{ss:csp}) are higher-order Wiener-Wintner systems.
    \item In \S\ref{s:mea}, we show how this condition allows for a simple proof of pointwise convergence in multiple recurrence averages in higher order Wiener-Wintner dynamical systems (Corollary \ref{AEMultRec}). In the process, we will extend the bound obtained by Bourgain in \cite{bourgain} for multiple recurrence (Theorem \ref{BBgeq3}). We will also prove the uniform Wiener-Wintner theorem for multiple recurrence for Wiener-Wintner dynamical systems in this section as well (Theorem \ref{t:uww_mea}).
    \item  In \S\ref{s:GHKsemnorm}, we show how one can provide alternative characteristics of the Host-Kra-Ziegler factors that characterize norm convergence of multiple recurrence in terms of higher-order Wiener-Wintner averages (Theorem \ref{t:higher.order.UWW}). As a consequence, we may extend the uniform Wiener-Wintner theorem (Theorem \ref{t:UWW}) to higher order (Theorem \ref{t:HUWW}).
    \item In \S\ref{s:K-stability}, we will address the stability property of a product of a Kolmogorov system and a Wiener-Wintner dynamical system (Theorem \ref{t:productWWS}); this provides us another example of Wiener-Wintner dynamical system that is neither weakly mixing nor distal.
    \item  A proof of the return times theorem for multiple ergodic averages for higher order Wiener-Wintner dynamical system is presented in \S\ref{s:rtt_mea} (Corollary \ref{c:retutnTimesCF}). In particular, we obtain a control of the $L^2(\nu)$-norm of the averages in (\ref{mrtt}) in terms of higher-order Wiener-Wintner averages; this control may hold for any measure-preserving systems (Lemma \ref{RT_WWest}).
\end{itemize}

One of the key estimates of the paper can be found in Theorem \ref{BBgeq3}, which may be used to control the $L^1$-norm of the multiple ergodic averages; this estimate, which was established by Bourgain for the case of two functions \cite{bourgain}, holds on any invertible measure-preserving systems. We note that this is a \textit{quantitative} estimate, in a sense that it holds for all $N \in \N$; we see in Remark \ref{WW.average.bound} that it is only when we let $N \to \infty$, we obtain a familiar \textit{qualitative} estimate where the limit of the norm of the averages is controlled by the Gowers-Host-Kra seminorms \cite{hostkra}. We have also seen this relationship between a quantitative estimate with Wiener-Wintner type averages and a qualitative estimate with the seminorms previously in \cite{assani-cubes}.

Unlike the classical Wiener-Wintner dynamical systems (which we will refer to as the first-order Wiener-Wintner dynamical systems), proving that a certain dynamical system is a higher-order Wiener-Wintner dynamical system requires additional work due to the non-linearity of the subset of functions that we are interested. We will handle this issue by establishing the ``multilinearity concerns" (Lemma \ref{l:mlc}). While this approach is theoretically simple and does not change the key ideas behind the proofs, it increases the technicality of some already lengthy computations. To help audience understand these key ideas, we will prove some of our results for simpler cases by restricting to smaller order (e.g. for the second or third order), which use the same ideas. For the reader who is interested, the proofs for the general cases are included in the appendices.

\subsection{Acknowledgment} RM would like to thank the hospitality of SUSTech Mathematics Department, as well as SUSTech International Center for Mathematics while he held his visiting position there until December 2023. 

Furthermore, we would like to thank the anonymous referees for their helpful comments.


\section{Preliminary}\label{s:prelim}
In this section, we summarize some notations, conventions, and concepts that will be used throughout the paper.
\subsection{Notations and conventions}\label{ss:nc}
We denote the quadruple like $(X, \mathcal{F}, \mu, T)$ to be a \textit{probability measure-preserving system} (or simply a \textit{system} for short), i.e. $(X, \mathcal{F}, \mu)$ is a probability space, and $T: X \to X$ is a $\mu$-invariant map (i.e. for every $A \in \mathcal{F}$, we have $T^{-1}A \in \mathcal{F}$ and $\mu(T^{-1}A) = \mu(A)$). Such a system is \textit{ergodic} if for any $A \in \mathcal{F}$ with $T^{-1}A = A$ we have $\mu(A) = 0$ or $\mu(A) = 1$ (i.e. there are no nontrivial $T$-invariant measurable subsets). If the underlying sigma algebra is clear or irrelevant in the discussion, sometimes we may omit the sigma algebra and denote the system as a triple $(X, \mu, T)$.

Given a system $(X, \mathcal{F}, \mu, T)$ and $p \in [1, \infty]$, we often denote $\norm{\cdot}_p$ to be the $L^p$-norm, i.e. if $f \in L^p(\mu)$, then $\norm{f}_p := \norm{f}_{L^p(\mu)}$. Moreover, if $\mathcal{A}$ is a $\sigma$-subalgebra of $\mathcal{F}$, we use $L^p(\mathcal{A})$ or $L^p(\mathcal{A}, \mu)$ to denote the corresponding $L^p$-space over $(X, \mathcal{A}, \mu)$ as a subspace of $L^p(\mu)$.

For $k\in \N$, denote $V_k = \{0, 1\}^k$. If $\eta = (\eta_1, \dots, \eta_k)\in V_k$, then we denote $|\eta| = \sum_{j = 1}^k \eta_j$. 

For $N\in \N$, let $[N] = \{1, \dots, N\}$. Notice that if $h := (h_1, h_2, \ldots, h_k) \in [N]^k$ and $\eta\in V_k$, then we can define the dot product in the usual way, i.e. $h \cdot \eta := \sum_{j=1}^k h_j \eta_j$.

Let $c:\C\to \C$ be complex conjugation, i.e. $cz = \bar z$. Notably, we have $c^m = c$ when $m$ is odd, and $c^m$ is the identity map when $m$ is even.

The floor function will be denoted as $\floor{\cdot}$, i.e. $\floor{\cdot} : \R \to \Z$ such that $\floor{x} := \max\{k \in \Z: k \leq x\}$. 

The term ``Wiener-Wintner" is also frequently abbreviated ``WW".

If $A$ is a finite set, then the cardinality of $A$ will be denoted as $\# A$.

\subsection{Nilsystems and Host-Kra-Ziegler factors}\label{ss:ns_hkz}
Let $G$ be a $k$-step nilpotent Lie group, and $\Gamma$ be a discrete co-compact subgroup. We say $X:=G/\Gamma$ is a \textit{$k$-step nilmanifold}. If $\mu$ is the normalized Haar measure on $X$, and $\tau \in G$, the translation $T:X \to X$ by $Tx = \tau \cdot x$ is invariant under $\mu$. Hence, we call $(X, \mu, T)$ a \textit{$k$-step nilsystem.}
		
A simple example of a nilsystem is a rotation on compact abelian group: Let $G = \R$ and $\Gamma = \Z$. Then $\T := G/\Gamma$ is a $1$-step nilmanifold, and if $R_\alpha: \T \to \T$ is a rotation (i.e. $R_\alpha(x) = x + \alpha)$, then $(X, \mathcal{B}, m, R_\alpha)$ is a $1$-step nilsystem, where $m$ is the Lebesgue measure and $\mathcal{B}$ the Borel $\sigma$-algebra. Given $k \in \N$, it can also be shown that the classical skew products (cf. \S \ref{ss:csp}) on $\T^k$ is a $k$-step nilsystem as well.

It was shown by Leibman \cite{leibman05, leibman05a} that the multiple recurrence averages in (\ref{Favg}) converge a.e. for nilsystems.

To show the mean convergence of the Furstenberg averages, the work of Host and Kra \cite{hostkra} as well as Ziegler \cite{ziegler07} identified characteristic factors: For every $k \in \N$, there exists a factor $(Z_{k-1}, \mathcal{Z}_{k-1}, \mu, T)$ (which we may sometimes denote simply as $\mathcal{Z}_{k-1})$, which we call the \textit{$k-1$-th Host-Kra-Ziegler (HKZ) factor}, of $(X, \mathcal{F}, \mu, T)$ for which
\[ \lim_{N \to \infty}\left\Vert  \frac{1}{N}  \sum_{n=1}^N\prod_{j=1}^k f_j \circ T^{jn} - \frac{1}{N} \sum_{n=1}^N  \prod_{j=1}^k  \E(f_j  | \mathcal{Z}_{k-1}) \circ T^{jn} \right\Vert_2 = 0. \]
While their constructions differ, they were shown to result in the same factor by Leibman \cite{leibman05b}. Moreover, it was shown by Host and Kra that the $k$-th HKZ factor is the inverse limit of $k$-step nilsystems. In particular, the $0$-th HKZ factor is the invariant sigma algebra, and the first HKZ factor is generated by the eigenfunctions of $T$; this would be the Kronecker factor if $T$ is ergodic (which we may sometimes denote as $\mathcal{K}$). It is also known that $\mathcal{Z}_{k-1}$ is the universal characteristic factor for $k$-term multiple ergodic averages, i.e. if there is another characteristic factor for the $k$-term multiple recurrence, then $\mathcal{Z}_{k-1}$ must be a factor of that characteristic factor.

\subsection{Key estimates}
We discuss a few key estimates that will be used throughout the paper.

The first one of such is the Van der Corput estimate (cf. \cite{KN74}), which is stated as follows:
		
		\begin{lemma}[Van der Corput's estimate]\label{vdc-lem}
			Let $N \in \N$. If $\{v_n\}_{n=0}^{N-1}$ is a finite sequence of complex numbers, and if $H$ is an integer between $0$ and $N-1$, then
			\begin{align*}
				& \left| \frac{1}{N} \sum_{n=0}^{N-1} v_n \right|^2
				\leq \frac{N+H}{N^2(H+1)} \sum_{n=0}^{N-1} |v_n|^2 + \frac{2(N+H)}{N^2(H+1)^2} \sum_{h=1}^H(H+1-h) \real \left(\sum_{n=0}^{N-h-1} \overline{v_{n+h}} v_n \right).
			\end{align*}
		\end{lemma}
		
		One useful variant of this estimate is the case $v_n = u_ne^{2\pi int}$ for some sequence of complex numbers $\{u_n\}$. In such case, we can apply Lemma \ref{vdc-lem} to obtain the following: For every $N \in \N$ and $1 \leq H \leq N-1$, we have
		\begin{equation}\label{vdc-1}
			\sup_t \left| \frac{1}{N} \sum_{n=0}^{N-1} u_ne^{2\pi int} \right|^2
			\leq \frac{2}{N(H+1)} \sum_{n=0}^{N-1} |u_n|^2 + \frac{4}{H+1} \sum_{h=1}^H \left| \frac{1}{N}\sum_{n=0}^{N-h-1} \overline{u_{n+h}} u_n \right|.
		\end{equation}
		Sometimes it is more convenient to use the summation variant of (\ref{vdc-1}); it is not too difficult to show that for every $N \in \N$, we have
		\begin{equation}\label{vdc-2}
			\sup_t\left|  \sum_{n=0}^{N-1} u_ne^{2\pi int} \right|^2
				\leq 2\sum_{n=0}^{N-1} |u_n|^2 +4 \sum_{h=1}^{N-1} \left|\sum_{n=0}^{N-h-1} \overline{u_{n+h}} u_n \right|.
		\end{equation}

Another frequently used inequality is a straightforward application of Hölder's inequality for Ces\`aro averages. It is stated and used in the following form:
\begin{lemma}[Hölder's inequality on averages]
Let $\{a_n\}_{n=1}^N$ be a finite sequence of real, nonnegative numbers. The function $A: \R \to \R$ such that
$$A_N(p) := \left(\frac{1}{N}\sum_{n=1}^N a_n^p\right)^{1/p}$$
is increasing in $p$: i.e. for $p \leq q$ we have
$$\left(\frac{1}{N}\sum_{n=1}^N a_n^p\right)^{1/p}\leq \left(\frac{1}{N}\sum_{n=1}^N a_n^q\right)^{1/q}.$$
\end{lemma}

Finally, we will be using the following maximal inequality for a certain approximation arguments (cf. \cite[Theorem 1.8]{AssaniWW} for instance):

\begin{lemma}[Maximal inequality]\label{l:maxIneq}
Let $(X, \mathcal{F}, \mu, T)$ be a measure-preserving system, and $p \in (1, \infty)$. For every real-valued function $f \in L^p(\mu)$, we have
\[ \norm{  \sup_N \frac{1}{N} \sum_{n=1}^N f \circ T^n}_p \leq \frac{p}{p-1} \norm{f}_p \, . \]
\end{lemma}

\section{Higher order Wiener-Wintner functions and dynamical systems}\label{s:HWWDS}

\subsection{Definitions}
Here we provide the definition of higher order Wiener-Wintner functions and dynamical systems.

\begin{definition}\label{def:hof}
Let $(X, \mathcal{F}, \mu, T)$ be a dynamical system, and $p\in [1, \infty]$. We say that $f\in L^\infty(\mu)$ is a \textit{first-order WW function of power type} $\alpha>0$ \textit{in} $L^p$ if there exists a constant $C_f>0$ such that for all $N \in \N$, 
$$\left \Vert \sup_{t} \left| \frac{1}{N} \sum_{n=1}^N e^{2\pi i n t} f\circ T^n \right| \right\Vert_p \leq \frac{C_f}{N^\alpha} \, .$$
For a positive integer $k \geq 2$, we say that $f\in L^\infty(\mu)$ is a $k$-\textit{th order WW function of power type $\alpha >0$ in $L^p$}  if there exists a constant $C_f$ such that for all $N \in \N$, 
$$\frac{1}{\lfloor \sqrt{N} \rfloor^{k-1}} \sum_{h \in [\lfloor \sqrt{N} \rfloor]^{k-1}}\left \Vert \sup_{t} \left| \frac{1}{N} \sum_{n=1}^N e^{2\pi i n t} \left[\prod_{\eta\in V_{k-1}} c^{|\eta|} f \circ T^{h \cdot \eta} \right]\circ T^n \right| \right\Vert_p^{2/3} \leq \frac{C_f}{N^\alpha} \, .$$
\end{definition}
For instance, in the case $k=2$, we see that second order WW functions are determined by polynomial decay on the averages
$$\frac{1}{\lfloor \sqrt{N} \rfloor} \sum_{h =1}^{\lfloor \sqrt{N} \rfloor }\left \Vert \sup_{t} \left| \frac{1}{N} \sum_{n=1}^N e^{2\pi i n t} f\circ T^n \cdot\overline{f \circ T^{n+h}} \right| \right\Vert_p^{2/3} \, .$$
Likewise, third order WW functions are determined by polynomial decay on the averages
$$\frac{1}{\lfloor \sqrt{N} \rfloor^2} \sum_{h_1, h_2 =1}^{\lfloor \sqrt{N} \rfloor }\left \Vert \sup_{t} \left| \frac{1}{N} \sum_{n=1}^N e^{2\pi i n t} f\circ T^n \cdot \overline{ f\circ T^{n + h_1}
 } \cdot \overline{ f\circ T^{n + h_2}} \cdot f \circ T^{n + h_1 + h_2} \right| \right\Vert_p^{2/3} \, ,$$
and so on. As first order WW functions are shown to control double recurrence averages \cite{assani}, we shall see that second order WW functions control triple recurrence averages, and so on.

\begin{remark}
    We note that there does \textit{not} exist a first-order WW function of power type $\alpha > 1/2$ in $L^2(\mu)$. This observation was made in \cite[p. 364]{assani04}. 
\end{remark}

In practice, the $2/3$ exponent is mostly irrelevant. It merely arises as the weakest bound that later arguments will require. If the inequality holds for some $\alpha$ without this exponent, then using Hölder's inequality applied to averages, we can reintroduce the $2/3$ exponent by replacing $\alpha$ with $2 \alpha/3$.

Similarly, the restriction of the $h_i$ averages to $\sqrt{N}$ is somewhat arbitrary. It arises from applications of Van der Corput's inequality, in which $H = \floor{\sqrt{N}}$ is chosen so that a remainder term of the form $1/H + H/N$ decays like a polynomial. Any choice of $H = \floor{N^\beta}$ for $0 < \beta<1$ would also yield such a remainder term, and $\beta = 1/2$ is chosen for convention and to get the best possible decay. If we were to define $WW_\beta(f, N)$ to be the above averages defined over $\floor{N^\beta}$ instead, then polynomial decay on $WW_\beta (f, N)$ for any $0 < \beta < 1$ would yield the same results that we later claim about WW functions. Further, a choice $\beta = 1$ would also suffice, as the proof of Theorem \ref{t:higher.order.UWW} establishes a bound of the form
$$WW_{1/2}(f, N) \leq WW_1(f, \floor{\sqrt{N}}) + \frac{1}{N^{2/3}}.$$
Hence, polynomial decay on the averages without square roots would transfer to polynomial decay of the WW averages from Definition \ref{def:hof}.

In any of computations of specific examples in the following section, $\floor{\sqrt{N}}$ can be replaced with $\floor{N^\beta}$ for $0 < \beta \leq 1$ with minimal changes to the argument.

\begin{definition}
Let $k \in \N$, let $\alpha > 0$, and let $p \in [1, \infty]$. We say that the ergodic dynamical system $(X, \mathcal{F}, \mu, T)$ is a \textit{k-th order WW system of power type $\alpha$ in $L^p(\mu)$} if there exists a set of $k$-th order WW functions of type $\alpha$ that is $L^2$-dense in the orthogonal complement of the $k$-th Host-Kra-Ziegler factor $\mathcal{Z}_k$. 

We say a system is a \textit{WW system} if the system is a WW system of power type $\alpha$ in $L^p$ for some $\alpha > 0$ and $p \in [1, \infty]$.
\end{definition}

In Section \ref{s:GHKsemnorm}, we show that a function $f\in L^\infty(\mu)$ lies in $L^2(\mathcal{Z}_k)^\perp$ if and only if the higher-order WW averages of Definition \ref{def:hof} converge to 0; hence, $k$-th order WW functions must necessarily lie in $L^2(\mathcal{Z}_k)^\perp$. While such averages for $f\in L^2(\mathcal{Z}_k)$ do not converge to zero, their limiting behavior is not known in generality.

\begin{remark}\label{r:rbo}
We note that any $k$-th order WW function will also be a lower order WW function, which follows from applying the Van der Corput inequality and trivially bounding the unweighted averages by the supremum over $e^{2\pi i n t}$-weighted averages. However, it does not immediately follow that $k$-th order WW systems will be lower order WW systems. We recall that the orthogonal complements $L^2(\mathcal{Z}_k)^\perp$ are decreasing as $k$ increases; i.e. $L^2(\mathcal{Z}_k)^\perp \subset L^2(\mathcal{Z}_{k-1})^\perp$ holds for any $k$. Should it happen to be case that for some $k$ this inclusion is strict, it would not immediately follow that a set that is dense in $L^2(\mathcal{Z}_k)^\perp$ is also dense in the larger set $L^2(\mathcal{Z}_{k-1})^\perp$. Such an example, which is covered in detail in \S\ref{ss:csp}, is a classical skew product on $\mathbb{T}^k$. As this transformation is a $k$-step nilsystem, it follows that $L^2(\mathcal{Z}_k)^\perp = \{0\}$, trivially satisfying the conditions for a $k$-th order WW system. However, in such a system $L^2(\mathcal{Z}_{k-1})^\perp \neq \{0\}$, and the dense set of $k$-th order WW functions in $\{0\}$ offer no immediate help for the rates of convergence of the nonzero function in $L^2(\mathcal{Z}_{k-1})^\perp$.

In cases such as weakly mixing transformations where all $\mathcal{Z}_k$ are equal, it does follow that $k$-th order systems will also be systems of lower order.
\end{remark}

\begin{remark}
In \cite{assani04} and \cite{assani}, the notion of first order WW functions and systems was generalized considering different rates of decay on the WW averages, such as a ``log-type" WW system for rate $\log(N)^{-1-\beta}$ or a $g$-system for $g(N)^{-1}$. These definitions carry over to higher orders by establishing these rates of decay to the above higher order WW averages. 
\end{remark}

\subsection{Multilinearity concerns}
For any dynamical system, first order WW functions in $L^p$ of a given power type form a subspace. However, this is not immediately the case for higher order WW functions, and causes great technical difficulties. The following provides a criterion under which we can conclude that the WW property is satisfied for certain linear spans:

\begin{lemma}[Multilinearity concerns] \label{l:mlc}
Let $(X, \mathcal{F}, \mu, T)$ be an ergodic dynamical system and $\mathcal{E}$ be a subset of $L^\infty(\mu)$ and $\alpha > 0$. Suppose it is the case that for every collection $(e_\eta)_{\eta\in V_{k-1}}$ with $e_\eta \in \mathcal{E}$, there exists $C > 0$ such that for all $N \in \N$, we have
\begin{equation}\label{est:mlc}
\frac{1}{\lfloor\sqrt{N}\rfloor^{k-1}}\sum_{h \in [\lfloor \sqrt{N}\rfloor ]^{k-1} } \left\Vert \sup_t \left|\frac{1}{N} \sum_{n=1}^N e^{2\pi i n t} \left[ \prod_{\eta \in V_{k-1}} c^{|\eta|} e_\eta \circ T^{h\cdot \eta} \right] \circ T^n \right|\right\Vert_p^{2/3} \leq \frac{C}{N^\alpha} \, .
\end{equation}
Then the elements of $S := span(\mathcal{E})$ are $k$-th order WW functions of power type $\alpha$ in $L^p$.
\end{lemma}
For instance, to show that the span of functions in $\mathcal{E}$ satisfies the second-order WW property in $L^p$ for $\alpha>0$, it would suffice to show that for every $e_1, e_2 \in \mathcal{E}$, there exists a constant $C > 0$ such that
$$\frac{1}{\lfloor \sqrt{N} \rfloor} \sum_{h =1}^{\lfloor \sqrt{N} \rfloor }\left \Vert \sup_{t} \left| \frac{1}{N} \sum_{n=1}^N e^{2\pi i n t} e_1 \circ T^n \cdot \overline{e_2 \circ T^{n + h}} \right| \right\Vert_p^{2/3} \leq \frac{C}{N^\alpha}$$
holds for all $N\in \N$.

\begin{proof}
Our goal is to show that every function in $S = span(E)$ is a $k$-th order WW function of type $\alpha$. First, suppose that $f \in E$. We see that $f$ is a $k$-th order WW function of type $\alpha$ by setting $e_\eta = f$ in (\ref{est:mlc}) for every $\eta \in V_{k-1}$. Furthermore, we let $\beta \in \C$, it also follows that $\beta f$ is a $k$-th order WW function, with constant multiplied by $|\beta|^{2^{k-1}}$. 

It remains to show that, for every $f, g \in E$, $f + g$ is a $k$-th order WW function of power type $\alpha$. We define a set $\Phi_k := \{ \phi := (\phi_\eta)_{\eta \in V_{k-1}} : \phi_\eta \in \{f, g\} \} \subset E^{2^{k-1}}$. Then we have
\begin{align*}
&\frac{1}{\lfloor\sqrt{N}\rfloor^{k-1}}\sum_{h \in [\lfloor \sqrt{N}\rfloor ]^{k-1} } \left\Vert \sup_t \left|\frac{1}{N} \sum_{n=1}^N e^{2\pi i n t} \left[ \prod_{\eta \in V_{k-1}} c^{|\eta|} (f+g) \circ T^{h\cdot \eta} \right] \circ T^n \right|\right\Vert_p^{2/3}\\ 
&= \frac{1}{\lfloor\sqrt{N}\rfloor^{k-1}}\sum_{h \in [\lfloor \sqrt{N}\rfloor ]^{k-1} } \left\Vert \sup_t \left|\frac{1}{N} \sum_{n=1}^N e^{2\pi i n t} \sum_{\phi \in \Phi_k} \left[ \prod_{\eta \in V_{k-1}} c^{|\eta|} \phi_\eta \circ T^{h\cdot \eta} \right] \circ T^n \right|\right\Vert_p^{2/3} \, .
\end{align*}
Since the supremum, norm, and function $\xi \to \xi^{2/3}$ are all subadditive, we can pull out the sum over such terms so that the last average is bounded above by
\begin{align*}
&\sum_{\phi \in \Phi_k} \frac{1}{\lfloor\sqrt{N}\rfloor^{k-1}}\sum_{h \in [\lfloor \sqrt{N}\rfloor ]^{k-1} } \left\Vert \sup_t \left|\frac{1}{N} \sum_{n=1}^N e^{2\pi i n t} \left[ \prod_{\eta \in V_{k-1}} c^{|\eta|} \phi_\eta \circ T^{h\cdot \eta} \right] \circ T^n \right|\right\Vert_p^{2/3}  \, .
\end{align*}
By the hypothesis of the lemma, for every $\phi \in \Phi_k$, there exists a positive constant $C(\phi) > 0$ so that the last average (inside the summand) is bounded above by $C(\phi) N^{-\alpha}$. Therefore, we have
\begin{align*}
&\frac{1}{\lfloor\sqrt{N}\rfloor^{k-1}}\sum_{h \in [\lfloor \sqrt{N}\rfloor ]^{k-1} } \left\Vert \sup_t \left|\frac{1}{N} \sum_{n=1}^N e^{2\pi i n t} \left[ \prod_{\eta \in V_{k-1}} c^{|\eta|} (f+g) \circ T^{h\cdot \eta} \right] \circ T^n \right|\right\Vert_p^{2/3} &\leq \frac{\sum_{\phi \in \Phi_k} C(\phi)}{N^\alpha} \, ,
\end{align*}
which proves that $f+g$ is indeed a $k$-th order WW function of power type $\alpha$.
\end{proof}

\subsection{Conjugacy class of a WW system}\label{ss:conjugacy}
Two dynamical systems $(X, \mathcal{F}, \mu, T)$ and $(Y, \mathcal{G}, \nu, S)$ are said to be \textit{conjugate} if there exists a measure-preserving bijection $\sigma:X \to Y$ such that $S = \sigma T \sigma^{-1}$. Such a map $\sigma$ may be called the \textit{conjugacy map}, and is only required to be defined between sets $X'$ and $Y'$ of full measure.

Here, we will briefly show that a conjugate of a WW system is also a WW system.
\begin{proposition}\label{p:conjugacy}
    Let $k \in \N$, let $\alpha > 0$, let $p \in [1, \infty]$, let $(X, \mathcal{F}, \mu, T)$ be a $k$-th order WW system of power type $\alpha$ in $L^p(\mu)$, and let $(Y, \mathcal{G}, \nu, S)$ be an ergodic system that is a conjugate of $(X, \mathcal{F}, \mu, T)$. Then $(Y, \mathcal{G}, \nu, S)$ is also a $k$-th order WW system of power type $\alpha$ in $L^p(\mu)$. 
\end{proposition}

\begin{proof}
    We will prove this for the case $k \geq 2$ since the case $k=1$ is similar (and simpler).

    Let $\sigma: (X, \mathcal{F}, \mu, T) \to (Y, \mathcal{G}, \nu, S)$ be the conjugacy map. This means that $\sigma$ is invertible, $S = \sigma T \sigma^{-1}$, the pushforward operator $\sigma^*: L^p(\mu) \to L^p(\nu)$, where $\sigma^*f = f\circ \sigma^{-1}$, is a bijective linear map, and $\nu = \sigma^*\mu$. This implies that for every $f \in L^p(\mu)$, we have 
    \begin{align*}
    &\norm{ \sup_{t} \left| \frac{1}{N} \sum_{n=1}^N e^{2\pi i n t} \left[\prod_{\eta\in V_{k-1}} c^{|\eta|} \sigma^*f \circ S^{h \cdot \eta} \right]\circ S^n \right|}_{L^p(\nu)}\\
    &= \norm{ \sup_{t} \left| \frac{1}{N} \sum_{n=1}^N e^{2\pi i n t} \left[\prod_{\eta\in V_{k-1}} c^{|\eta|} f \circ T^{h \cdot \eta} \sigma \right]\circ T^n\sigma \right|}_{L^p(\sigma^*\mu)} \\
    &= \norm{ \sup_{t} \left| \frac{1}{N} \sum_{n=1}^N e^{2\pi i n t} \left[\prod_{\eta\in V_{k-1}} c^{|\eta|} f \circ T^{h \cdot \eta} \right]\circ T^n \right|}_{L^p(\mu)} \, .
    \end{align*}
    Therefore, if $f$ is a WW function of power type $\alpha > 0$ in $L^p(\mu)$, then $\sigma^*f$ is a WW function of power type $\alpha > 0$ in $L^p(\nu)$. Because the image of a dense set of the map $\sigma^*:L^p(\mu) \to L^p(\nu)$ is a dense set in its range and $\sigma^*L^p(\mathcal{Z}_{k}(T), \mu) = L^p(\mathcal{Z}_{k}(S),  \nu)$, we have shown that $(Y, \mathcal{G}, \nu, S)$ is a WW system of power type $\alpha$ in $L^p(\nu)$.
\end{proof}

\begin{remark}
    Let $(X, \mathcal{F}, \mu)$ be a probability measure space, and let $G$ be the group of all of the invertible $\mu$-invariant transformations of $X$. If there exists $T \in G$ for which $(X, \mathcal{F}, \mu, T)$ is a $k$-th order WW system of power type $\alpha$ in $L^p(\mu)$, then its conjugacy class $\{ S^{-1}TS: S \in G \}$, where every system $(X, \mathcal{F}, \mu, S^{-1}TS)$ is a WW system of power type $\alpha$ in $L^p(\mu)$ by Proposition \ref{p:conjugacy}, is dense in $G$ with respect to the weak topology (in a sense of Halmos, cf. sections on ``weak topology" and ``weak approximation" in \cite{HalmosBook}).
    \end{remark}

    

\section{Examples}\label{s:ex}

In this section we will provide some examples of higher order WW systems. These include a K system (strong mixing, positive entropy) as well as a classical skew product (non-mixing, zero entropy). In \cite{assani04}, a system is constructed which is not a first order WW system of power type for any $\alpha \in (0, 1)$; as this system is weakly mixing, it follows by the comments in Remark \ref{r:rbo} that this system is also not a $k$-th order WW system of power type for any $k \geq 2$.

\textit{Trivial example:} Let $k \in \N$. A $k$-step nilsystem, as well as a system that is isomorphic to a $k$-th HKZ factor, is trivially a $k$-th order WW system, since $L^2(\mathcal{Z}_k)^\perp = \{0\}$.

\subsection{Pinsker algebras and Kolmogorov automorphisms}\label{ss:Ksys}

Recall properties of the Pinsker $\sigma$-subalgebra $\mathcal{P}$: for any ergodic dynamical system, $(X, \mathcal{F}, \mu, T)$, there exists a $\sigma$-subalgebra $\mathcal{A} \subset \mathcal{F}$ such that
\begin{equation}\label{Pinsker}
    T^{-1} \mathcal{A} \subset \mathcal{A} \, , \quad \quad \bigcap_{n\in \mathbb{N}} T^{-n}\mathcal{A} = \mathcal{P} \, , \quad \quad \bigvee_{n\in \mathbb{N}} T^n \mathcal{A} \text{ is dense in }\mathcal{F} \, .
\end{equation}
Moreover, the span of the set of functions 
\begin{equation}\label{setE}
    \mathcal{E} :=  \{ f_k^A := \1_A - \E(\1_A | T^{-k}\A) : k \in \Z \text{ and } A \in T^{-l}\A \text{ for some } l \in \Z\} 
\end{equation}
 is dense in the set $L^2(\mathcal{P})^\perp = \{f\in L^2(\mu): \E (f|\mathcal{P}) = 0\}$ by the martingale convergence theorem. The functions in the span of $\mathcal{E}$ are first-order WW functions \cite[Theorem 4]{assani}. We claim the following:
\begin{theorem}\label{t:Ksys}
Every function in the span of $\mathcal{E}$ from (\ref{setE}) is a $J$th-order WW functions of power type $1/6$ in $L^2$ for all $J\geq 2$.
\end{theorem}

\begin{proof}
For clarity, we present the case $J=2$, corresponding to triple recurrence. A similar argument shows the general case, and is included in Appendix \ref{a:PA}.

Let $j \in \{1, 2\}$, let $k_j \in \Z$, $A_j \in T^{-l_j}\A$, and $l_j \in \Z$. We can assume that $l_j < k_j$, or else $f_{k_j}^{A_j} = 0$. To satisfy the hypotheses of the multilinearity concerns lemma (Lemma \ref{l:mlc}), we are interested in showing polynomial decay on the terms
$$\frac{1}{\lfloor \sqrt{N} \rfloor} \sum_{h=1}^{\lfloor \sqrt{N} \rfloor} \left \Vert \sup_{t} \left | \frac{1}{N} \sum_{n=1}^N e^{2\pi i n t} f_{k_1}^{A_1} \circ T^n \cdot f_{k_2}^{A_2}\circ T^{n+h} \right| \right\Vert_p^{2/3}$$
(since the functions are real-valued, the complex conjugate is not necessary). To this end, fix a large natural number $h$ which satisfies $h > \max\{0, l_1 - l_2, k_1 - l_2, k_1 - k_2\}$. We wish to show the desired decay for the above summand of $h$ alone. 

Define $F = f_{k_1}^{A_1} \cdot f_{k_2}^{A_2} \circ T^h$. As each $\norm{f_{k_j}^{A_j}}_{\infty}\leq 2$, we have $\norm{F}_\infty \leq 4$. We compute pointwise by the summation variant of the Van der Corput lemma (\ref{vdc-2}) that
\begin{align*}
\sup_{t} \left|\sum_{n=1}^N e^{2\pi i n t} F(T^{n} x) \right|^2 &\leq 32N + 4\sum_{m = 1}^{N-1}\left| \sum_{n = 1}^{N-m} F(T^{n} x)F(T^{n + m} x)\right| \\
&\leq 32N + 64(k_1 - l_1)N + 4\sum_{m = k_1 - l_1+1}^{N-1}\left| \sum_{n = 1}^{N-m} F(T^{n} x)F(T^{n + m} x)\right| \, .
\end{align*}
By integrating both sides and bounding the $L^1(\mu)$ norm by the $L^2(\mu)$ norm, we have
\begin{align*}
&\int \sup_{t} \left| \sum_{n=1}^N e^{2\pi i n t} F(T^{n} x) \right|^2 \, d\mu(x) \\ &\leq 32N + 64(k_1 - l_1)N + 4\sum_{m = k_1 - l_1+1}^{N-1} \left(\int\left| \sum_{n = 1}^{N-m} F(T^{n} x)F(T^{n + m} x)\right|^2 \, d\mu(x)\right)^{1/2} \, .
\end{align*}
Consider the inner $n$ sum, which is squared. If we factor out this product, the diagonal terms can be bounded above by $256N$, and we are left with off-diagonal terms, which are
\begin{align*}
&2\sum_{n < j}^{N-m} \int  F(T^{n} x)F(T^{n + m} x)F(T^{j} x)F(T^{j + m} x) \, d\mu(x) \\
&= 2\sum_{n < j}^{N-m} \int f_{k_1}^{A_1} \circ T^n\cdot  f_{k_2}^{A_2} \circ T^{n+h} \cdot f_{k_1}^{A_1} \circ T^{n+m} \cdot f_{k_2}^{A_2} \circ T^{n+ m + h} \\
&\cdot f_{k_1}^{A_1} \circ T^j \cdot f_{k_2}^{A_2} \circ T^{j+h} \cdot f_{k_1}^{A_1} \circ T^{j+m} \cdot f_{k_2}^{A_2} \circ T^{j+ m + h} \, d\mu \, .
\end{align*}
Recall that $m > k_1 - l_1$. We claim that if $j > n + k_1 - l_1$, then 
\begin{equation}\label{0int} \quad \int f_{k_1}^{A_1} \circ T^n \cdot f_{k_2}^{A_2} \circ T^{n+h} \cdot f_{k_1}^{A_1} \circ T^{n+m} \cdot f_{k_2}^{A_2} \circ T^{n+ m + h} \cdot f_{k_1}^{A_1} \circ T^j \cdot f_{k_2}^{A_2} \circ T^{j+h} \cdot f_{k_1}^{A_1} \circ T^{j+m} \cdot f_{k_2}^{A_2} \circ T^{j+ m + h} \, d\mu = 0 \, .
\end{equation}
To this end, we consider the case $j > n + m$. Notice that the above eight functions are measurable to the degrees $T^{-q}\A$ for the following values of $q$:
$$n + l_1 \quad n + h + l_2 \quad n + m + l_1 \quad n + m + h + l_2 \quad j + l_1 \quad j + h + l_2 \quad j + m + l_1 \quad j + m + h + l_2 $$
Hence, by the condition $j > n+m$, and the initially chosen condition that $h > l_1 -l_2$ (or $h + l_2 > l_1$), it follows that
$$f_{k_1}^{A_1} \circ T^{n+m} \cdot f_{k_2}^{A_2} \circ T^{n+ m + h} \cdot f_{k_1}^{A_1} \circ T^j \cdot f_{k_2}^{A_2} \circ T^{j+h} \cdot f_{k_1}^{A_1} \circ T^{j+m} \cdot f_{k_2}^{A_2} \circ T^{j+ m + h} \quad \text{is } T^{-(n + m + l_1)}\A \text{ measurable.}$$
So for the integral $(\ref{0int})$, we would get the same value if we conditioned the integrand on $T^{-(n + m + l_1)}\A$, and the above product of six functions can factor out of this conditional expectation. So the claim that $(\ref{0int})$ is zero reduces to showing 
$$\E( f_{k_1}^{A_1} \circ T^n \cdot f_{k_2}^{A_2} \circ T^{n+h} | T^{-(n + m + l_1)}\A) = 0 \, .$$
By definition, we have 
\begin{align*}
f_{k_1}^{A_1} \circ T^n \cdot f_{k_2}^{A_2}\circ T^{n+h} &= \1_{A_1} \circ T^n \cdot \1_{A_2}\circ T^{n+h} - \E(\1_{A_1} | T^{-k_1}\A)\circ T^n \cdot \1_{A_2}\circ T^{n+h} \\ &\quad - \1_{A_1} \circ T^n \cdot \E(\1_{A_2} | T^{-k_2}\A)\circ T^{n+h} + \E(\1_{A_1} | T^{-k_1}\A)\circ T^n \cdot \E(\1_{A_2} | T^{-k_2}\A)\circ T^{n+h} \, .
\end{align*}
The first and second terms above will cancel each other out under the $T^{-(n + m + l_1)}\A$ conditional: By recalling that we chose $h > k_1 - l_2$ (or $h + l_2 > k_1$), we observe that
$$\E(\1_{A_1} \circ T^n \cdot \1_{A_2}\circ T^{n+h} | T^{-(n + m + l_1)}\A) = \E(\1_{A_1}  \cdot \1_{A_2}\circ T^{h} | T^{-(m + l_1)}\A) \circ T^n \, ,$$
and
\begin{align*}
\E(\E(\1_{A_1}|T^{-k_1}\A) \circ T^n \cdot \1_{A_2}\circ T^{n+h} | T^{-(n + m + l_1)}\A) &= \E(\E(\1_{A_1}|T^{-k_1}\A) \cdot \1_{A_2}\circ T^{h} | T^{-( m + l_1)}\A) \circ T^n \\
&= \E(\E(\1_{A_1} \cdot \1_{A_2}\circ T^{h} |T^{-k_1}\A)  | T^{-( m + l_1)}\A) \circ T^n \\
&= \E(\1_{A_1} \cdot \1_{A_2}\circ T^{h}  | T^{-( m + l_1)}\A) \circ T^n
\end{align*}
as $m +l_1 > k_1$. Likewise, the third and fourth terms cancel: Recalling that we picked $h > k_1 - k_2$, we observe 
$$\E(\1_{A_1} \circ T^n \cdot \E(\1_{A_2}| T^{-k_2}\A) \circ T^{n+h} | T^{-(n + m + l_1)}\A) = \E(\1_{A_1} \cdot \E(\1_{A_2}|T^{-k_2}\A)\circ T^{h} | T^{-(m + l_1)}\A) \circ T^n \, ,$$
and
\begin{align*}
&\E(\E(\1_{A_1}|T^{-k_1}\A) \circ T^n \cdot \E(\1_{A_2}|T^{-k_2}\A)\circ T^{n+h} | T^{-(n + m + l_1)}\A) \\
&= \E(\E(\1_{A_1}|T^{-k_1}\A) \cdot \E(\1_{A_2}|T^{-k_2}\A)\circ T^{h} | T^{-( m + l_1)}\A) \circ T^n \\
&= \E(\E(\1_{A_1} \cdot \E(\1_{A_2}\circ T^{h}|T^{-k_2}\A) |T^{-k_1}\A)  | T^{-( m + l_1)}\A) \circ T^n \\
&= \E(\1_{A_1} \cdot \E(\1_{A_2}|T^{-k_2}\A)\circ T^{h}  | T^{-( m + l_1)}\A) \circ T^n \, .
\end{align*}
Hence, the integral $(\ref{0int})$ does indeed vanish when $j > n + m$. For the case $n + m \geq j > n + k_1 - l_1$, we similarly observe that
$$f_{k_1}^{A_1} \circ T^{n+m} \cdot f_{k_2}^{A_2} \circ T^{n+ m + h} \cdot f_{k_1}^{A_1} \circ T^j \cdot f_{k_2}^{A_2} \circ T^{j+h} \cdot f_{k_1}^{A_1} \circ T^{j+m} \cdot f_{k_2}^{A_2} \circ T^{j+ m + h} \quad \text{is } T^{-(j + l_1)}\A \text{ measurable,}$$
and by the same argument, conditioning $(\ref{0int})$ under $T^{-(j + l_1)}\A$ shows that it is zero.

Returning to our bound, we have observed that the off-diagonal terms vanish when $j$ is larger than $n$ by $k_1 - l_1$. Hence, for each $n = 1$ to $N - m$, at most $k_1 - l_1$ terms are nonzero, and the total number of nonzero terms in the sum over $n$ and $J$ can be bounded above by $N(k_1 -l_1)$. Plugging this into our estimate, we see
\begin{align*}
\int \sup_{t} \left| \sum_{n=1}^N e^{2\pi i n t} F(T^{n} x) \right|^2 \, d\mu(x) &\leq 32N + 64(k_1 - l_1)N + 4\sum_{m = k_1 - l_1+1}^{N-1} \left(256N + 512N(k_1-l_1) \right)^{1/2}\, .
\end{align*}
Hence, we have the bound
\begin{align*}
\int \sup_{t} \left| \frac{1}{N} \sum_{n=1}^N e^{2\pi i n t} f_{k_1}^{A_1}(T^n x) f_{k_2}^{A_2}(T^{n + h}x) \right|^2 \, d\mu(x) &\leq \frac{C}{N^{1/2}} \, ,
\end{align*}
which is uniform in sufficiently large $h > L :=\max\{l_1 - l_2, k_1 - l_2, k_1 - k_2\}$ and the constant $C$ depends only on $A_1, k_1, l_1$. Raising both sides to the $1/3$ power, we have
\begin{align*}
\left\Vert \sup_{t} \left| \frac{1}{N} \sum_{n=1}^N e^{2\pi i n t} f_{k_1}^{A_1} \circ T^n \cdot  f_{k_2}^{A_2} \circ T^{n + h} \right|\right\Vert_2^{2/3} &\leq \frac{C}{N^{1/6}} \, .
\end{align*}
For small $1 \leq h \leq L$, we have the bound
\begin{align*}
\left\Vert \sup_{t} \left| \frac{1}{N} \sum_{n=1}^N e^{2\pi i n t} f_{k_1}^{A_1}\circ T^n \cdot  f_{k_2}^{A_2} \circ T^{n + h} \right|\right\Vert_2^{2/3} &\leq 4^{2/3}
\end{align*}
by the triangle inequality. But since there are only finitely many of such $h$, they are lost in the average for $N > (L + 1)^2$. Therefore,
\begin{align*}
&\quad\frac{1}{\lfloor \sqrt{N} \rfloor} \sum_{h=1}^{\lfloor \sqrt{N} \rfloor} \left\Vert \sup_{t} \left| \frac{1}{N} \sum_{n=1}^N e^{2\pi i n t} f_{k_1}^{A_1} \circ T^n \cdot  f_{k_2}^{A_2} \circ T^{n + h} \right|\right\Vert_2^{2/3} \\
& = \frac{1}{\lfloor \sqrt{N} \rfloor} \sum_{h=1}^{L} \left\Vert \sup_{t} \left| \frac{1}{N} \sum_{n=1}^N e^{2\pi i n t} f_{k_1}^{A_1} \circ T^n \cdot f_{k_2}^{A_2}\circ T^{n + h} \right|\right\Vert_2^{2/3} \\
&\quad + \frac{1}{\lfloor \sqrt{N} \rfloor} \sum_{h=L+1}^{\lfloor \sqrt{N} \rfloor} \left\Vert \sup_{t} \left| \frac{1}{N} \sum_{n=1}^N e^{2\pi i n t} f_{k_1}^{A_1} \circ T^n \cdot  f_{k_2}^{A_2} \circ T^{n + h} \right|\right\Vert_2^{2/3} \\
& \leq \frac{4^{2/3}L}{\lfloor \sqrt{N} \rfloor}  + \frac{\lfloor \sqrt{N} \rfloor - L - 1}{\lfloor \sqrt{N} \rfloor} \frac{C}{N^{1/6}} \leq \frac{C'}{N^{1/6}}
\end{align*}
for a larger $C'$ that still only depends on $A_j, l_j, k_j$. Since the bound $N>(L+1)^2$ only depends on the same constants, we can increase $C'$ further to get the above bound for small $N$, without changing the dependence.
\end{proof}

An immediate corollary of the theorem above via Lemma \ref{l:mlc} is the following:
\begin{corollary}\label{c:PinskerWW}
    Let $J \in \N$. Then the set of $J$-th order WW functions of power type $1/6$ in $L^2$ is dense in $L^2(\mathcal{P})^\perp$.
\end{corollary}

Recall that the set $(X, \mathcal{F}, \mu, T)$ is a \textit{Kolmogorov system} (``\textit{K system}" for short) if $\mathcal{P}$ from (\ref{Pinsker}) is trivial (example: A Bernoulli shift is a K system). This implies that the span of $\mathcal{E}$ is dense in the set of $L^2(\mu)$ functions with zero integral, which is the orthogonal complement of $L^2(\mathcal{Z}_k)$ for any $k \in \N$ since a K system is strongly mixing. Since it was shown in \cite[Theorem 5]{assani} that every K system is a first-order WW system, we conclude that
\begin{corollary}\label{c:Ksys}
Every K system is a $J$-th order WW system for every $J \in \N$.
\end{corollary}

We  note that the same argument can be applied to establish $J$-th order WW functions in product systems for $J > 2$.

\begin{proposition}\label{t:Kprod}
Let $(X, \mathcal{F}, \mu, T)$ be a K system, and $(Y, \mathcal{G}, \nu, S)$ be a measure-preserving system. Let $\mathcal{E} \subset L^\infty(\mu)$ be the set in (\ref{setE}). All of the functions in the span of the set
\[ \{ f \otimes g: f \in \mathcal{E}, g \in L^\infty(\nu) \}\]
are $J$th-order WW functions of power type $1/6$ in $L^2$ for the product system $(X \times Y, \mathcal{F} \otimes \mathcal{G}, \mu \times \nu, T\times S)$ for all $J \geq 2$
\end{proposition}

\begin{proof}
Again, we prove the $J = 2$ case. The $J\geq 3$ case follows by applying the following argument to the general case of Theorem \ref{t:Ksys}. 

As before, let $k_j \in \Z$, $A_j \in T^{-l_j}\A$, $l_j \in \Z$, and $g_i\in L^\infty(\nu)$ hold for $j = 1, 2$. We can assume that $l_j < k_j$, or else $f_{k_j}^{A_j} = 0$. Again, we are interested in showing polynomial decay on the terms
$$\frac{1}{\lfloor \sqrt{N} \rfloor} \sum_{h=1}^{\lfloor \sqrt{N} \rfloor} \left \Vert \sup_{t} \left | \frac{1}{N} \sum_{n=1}^N e^{2\pi i n t} f_{k_1}^{A_1}  \circ T^n \otimes g_1 \circ S^n \cdot f_{k_2}^{A_2} \circ T^{n+h} \otimes\overline{g_2 \circ S^{n+h}} \right| \right\Vert_{L^2(\mu \times \nu)}^{2/3} \, .$$
To this end, we note that the argument from the previous theorem is unhindered if we add in an arbitrary bounded sequence of complex numbers $a_n$. We fix $h$ as before, 
with $F = f_{k_1}^{A_1} \cdot f_{k_2}^{A_2} \circ T^h$, and see that
\begin{align*}
&\int \sup_{t} \left| \sum_{n=1}^N e^{2\pi i n t} a_n f_{k_1}^{A_1}(T^n x) f_{k_2}^{A_2}(T^{n+h} x) \right|^2 \, d\mu(x) \\ 
&= \int \sup_{t} \left| \sum_{n=1}^N e^{2\pi i n t} a_n F(T^{n} x) \right|^2 \, d\mu(x) \\
 &\leq 32\Vert a_n \Vert_{\ell^\infty}^2  N + 64(k_1 - l_1)\Vert a_n \Vert_{\ell^\infty}^2 \ N + 4\sum_{m = k_1 - l_1+1}^{N-1} \left(\int\left| \sum_{n = 1}^{N-m} a_n \overline{a_{n+m}}F(T^{n} x)F(T^{n + m} x)\right|^2 \, d\mu(x) \right)^{1/2} \, .
\end{align*}
If we expand the square around the $n$ sum, the diagonal terms can be bounded by $256 \Vert a_n \Vert_{\ell^\infty}^4  N$. The off-diagonal terms have the form
\begin{align*}
&\int  a_n \overline{a_{n+m}}F(T^{n} x)F(T^{n + m} x) \overline{a_j}a_{j+m}F(T^{j} x)F(T^{j + m} x) \, d\mu(x)\\
&=   a_n \overline{a_{n+m}} \overline{a_j}a_{j+m} \int F(T^{n} x)F(T^{n + m} x) F(T^{j} x)F(T^{j + m} x) \, d\mu(x)
\end{align*}
which vanish for $j > n + k_1 - l_1$, for the exact same reasons as in the proof of Theorem \ref{t:Ksys}. Hence, we are left with a bound
\begin{align*}
\int \sup_{t} \left| \sum_{n=1}^N e^{2\pi i n t} a_n F(T^{n} x) \right|^2 \, d\mu(x)
&\leq 32\Vert a_n \Vert_{\ell^\infty}^2  N + 64\Vert a_n \Vert_{\ell^\infty}^2 (k_1 - l_1)N \\ &+ 4\sum_{m = k_1 - l_1+1}^{N-1} \left(256\Vert a_n \Vert_{\ell^\infty}^4  N + 512\Vert a_n \Vert_{\ell^\infty}^4N(k_1-l_1) \right)^{1/2}\, ,
\end{align*}
which is of the form
\begin{align*}
\int \sup_{t} \left| \frac{1}{N} \sum_{n=1}^N e^{2\pi i n t} a_n f_{k_1}^{A_1}(T^n x)f_{k_2}^{A_2}(T^{n + h}x) \right|^2 d\mu(x)&\leq \frac{C\Vert a_n \Vert_{\ell^\infty}^2 }{N^{1/2}}
\end{align*}
for a constant $C$ that depends only on $k_1, l_1$. So we set $a_n = g_1(S^n y)\overline{g_2(S^{n+h} y)}$, and integrate the above with respect to $\nu$ to see 
\begin{align*}
\left\Vert \sup_{t} \left| \frac{1}{N} \sum_{n=1}^N e^{2\pi i n t} f_{k_1}^{A_1} \circ T^n \otimes g_1 \circ S^n \cdot f_{k_2}^{A_2} \circ T^{n + h} \otimes \overline{g_2\circ S^{n+h}} \right|  \right\Vert_{L^2(\mu \times \nu)} ^2 &\leq \frac{C\Vert g_1 \Vert_{L^\infty(\nu)}^2 \Vert g_2 \Vert_{L^\infty(\nu)}^2 }{N^{1/2}}
\end{align*}
in the product measure. Raising to the third power, we have 
\begin{align*}
\left\Vert \sup_{t} \left| \frac{1}{N} \sum_{n=1}^N e^{2\pi i n t} f_{k_1}^{A_1} \circ T^n \otimes g_1 \circ S^n \cdot f_{k_2}^{A_2} \circ T^{n + h} \otimes \overline{g_2 \circ S^{n+h}} \right|  \right\Vert_{L^2(\mu \times \nu)} ^{2/3} &\leq \frac{C\Vert g_1 \Vert_{L^\infty(\nu)}^{2/3} \Vert g_2 \Vert_{L^\infty(\nu)}^{2/3} }{N^{1/6}}
\end{align*}
for uniformly large $h$. Since we have the same trivial bound on small values of $h$, their contribution vanishes in the average over $h$ as before. Hence, we get the desired bound.
\end{proof}

\begin{remark}\label{r:product}
    An analogous case of Proposition \ref{t:Kprod} for $J=1$ was done in \cite[Theorem 7]{assani}, and that was enough to guarantee the following: If $(X, \mathcal{F}, \mu, T)$ is a K system, and $(Y, \mathcal{G}, \nu, S)$ is any ergodic system, then the product system $(X \times Y, \mathcal{F} \times \mathcal{G}, \mu \otimes \nu, T \times S)$ is a first-order WW system. However, we cannot automatically say the same for the higher-order WW dynamical systems, due to the multilinearity concerns. We will handle this delicate issue in \S \ref{s:K-stability}.
\end{remark}

\subsection{Classical skew products}\label{ss:csp}

For an example with zero entropy, let $\alpha\in (0,1)$ and $k\in \N$ and consider the skew shift $T_\alpha:\T^k \to \T^k$ given by
$$T_\alpha (x_1, \dots, x_k) = (x_1 + \alpha, x_2 + x_1, x_3 + x_2, \dots, x_k + x_{k-1}) \, .$$
Define
$$P_j(n) =\sum_{m_{j-1}=0}^{n-1}\dots \sum_{m_2 = 0}^{m_3 - 1} \sum_{m_1=0}^{m_2 - 1} m_1 \quad  \text{for } j \geq 2.$$
By Faulhaber's formula, each $P_j$ is a polynomial of degree exactly $j$, with leading coefficient $1/(j!)$. Moreover, we compute
$$T_\alpha^n (x_1, \dots, x_k) = \left( x_1 + n\alpha, x_2 + nx_1 + P_2(n) \alpha, \dots, x_{k} + nx_{k-1} +\sum_{j=2}^{k-1} P_{j}(n)x_{k-j} + P_{k}(n) \alpha\right) \, .$$
Let $\mu$ be the normalized Haar measure on $\T^2$. It was demonstrated in \cite[Proposition 8]{assani} that the system $(\T^2, T_\alpha, \mu)$ is a (first order) WW system for Lebesgue a.e. $\alpha \in \R$. Here we will show that the higher dimensional skew product system is indeed a higher order WW system.

\begin{theorem}\label{t:csp}
Let $k\geq 3$, and let $\mu$ be the normalized Haar measure on $\T^k$. For Lebesgue a.e. $\alpha\in \R$, the system $(\T^k, T_{\alpha}, \mu)$ is a $(k-1)$th order WW system of power type $1/24$ in $L^2$.
\end{theorem}
\begin{remark}
Since $\mathcal{Z}_l$ is the entire $\sigma$-algebra for $\T^k$ in the case that $l \geq k$, the system $(\T^k, T_\alpha, \mu)$ is trivially a $l$-th order WW system for $l \geq k$.
\end{remark}

\begin{proof} As before, we present here the case $k=3$. The proof of the general statement follows the same argument and is included in Appendix \ref{a:sp}.

The orthogonal complement of $L^2(\mathcal{Z}_2)$ of $(\T^3, T_\alpha, m)$ is spanned by functions 
$$f_a(x_1, x_2, x_3) = e^{2\pi i a_1 x_1}e^{2\pi i a_2 x_2}e^{2\pi i a_3 x_3} = e^{2\pi i( a \cdot x)} \, ,$$
where $a = (a_1, a_2, a_3)\in \Z^3$ has $a_3 \neq 0$. By the multilinearity concerns (in this case bilinearity), we are interested in 
$$\left \Vert \sup_{t} \left|\frac{1}{N} \sum_{n=1}^N e^{2\pi i n t} f_a \circ T_\alpha^n \cdot \overline{f_b \circ T_\alpha^{n+h}} \right|\right \Vert_2$$
for $a = (a_1, a_2, a_3)$, $b = (b_1, b_2, b_3)$, both $a_3$ and $b_3$ nonzero, and $h$ fixed. Suppose $\alpha \neq 0$. For $x \in \T^3$, we compute
\begin{align*}
f_a(T_\alpha^nx) = \text{Exp}\big[& 1(a_1x_1 + a_2x_2 + a_3x_3) 
\\ + &n(a_1\alpha + a_2 x_1 + a_3 x_2) 
\\ + &P_2(n)(a_2\alpha + a_3x_1) 
\\ + &P_3(n)(a_3\alpha ) \big] \, ,
\end{align*}
from which it follows
\begin{align*}
f_a(T_\alpha^nx)&\overline{f_b(T_\alpha^{n+h} x)} \\
= \text{Exp}\big[& 1(a_1x_1 + a_2x_2 + a_3x_3) - 1(b_1x_1 + b_2x_2 + b_3x_3)
\\ + &n(a_1\alpha + a_2 x_1 + a_3 x_2) - (n+h)(b_1 \alpha + b_2 x_1 + b_3 x_2)
\\ + &P_2(n)(a_2\alpha + a_3x_1) - P_2(n+h)(b_2 \alpha + b_3 x_1)
\\ + &P_3(n)(a_3\alpha ) - P_3(n+h)(b_3 \alpha )\big]
\\ = \text{Exp} & \left[ \begin{pmatrix} n(a_1) - (n+h)(b_1)
\\ + P_2(n)(a_2) - P_2(n+h)(b_2 )
\\ + P_3(n)(a_3 ) - P_3(n+h)(b_3)\end{pmatrix} \alpha + \begin{pmatrix} 1(a_1x_1 + a_2x_2 + a_3x_3) - 1(b_1x_1 + b_2x_2 + b_3x_3) \\ + n(a_2x_1 + a_3 x_2) - (n+h) (b_2 x_1 + b_3 x_2) \\ +P_2(n)(a_3x_1) - P_2(n+h)(b_3x_1) \end{pmatrix} \right] \\
=: \text{Exp}&\left[ Q(n)\alpha + \dot Q_x(n)\right] \, .
\end{align*}
Notice that the polynomial $Q$ depends on $a, b$, and $h$, all of which are fixed. Since all $a_j, b_j$ are integers, and each $P_j$ is integer-valued, it follows that $Q$ is integer valued for $n$. The remainder polynomial $\dot Q$ depends on $a, b, h$ and the point $x$.

As each $P_j$ has degree exactly $j$, it follows that $Q(n)$ has degree at most 3. The $n^2$ term of $Q$, which depends on $h$, is given by $q_2(h) = \frac{Q''(0)}{2}$. Under two derivatives, the $P_2$ terms will become constant and have no $h$ dependence. The $P_3$ terms become linear, and $P_3''(0)(a_3) - P_3''(h)(b_3)$ must have nonzero linear $h$-dependence, because $b_3$ is nonzero. Hence, $q_2(h)$ is linear in $h$, and there is at most one value $h_0$ such that $q_{2}(h_0) = 0$. For all other $h$, it follows that the degree of $Q$ is at least 2.

Now, we apply the summation variant of the Van der Corput Lemma (\ref{vdc-2}) pointwise:
\begin{align*}
\sup_t\left|\sum_{n=1}^N e^{2\pi i n t} f_a(T^n_\alpha x)\overline{f_b(T_\alpha^{n+h}x)} \right|^2 &= \sup_t\left|\sum_{n=1}^N e^{2\pi i n t} e^{2\pi i [Q(n)\alpha + \dot Q(n)]}\right|^2 \\
&\leq 2N + 4\sum_{m=1}^{N-1} \left|\sum_{n=1}^{N-m} e^{2\pi i [(Q(n) - Q(n+m))\alpha + (\dot Q_x(n) - \dot Q_x(n+m))]}\right| \, .
\end{align*}
Define $P(n)= Q(n) - Q(n+m)$ and $\dot P_x(n) = \dot Q_x(n) - \dot Q_x(n+m)$. Notice that this finite difference operation decreases the degrees of $Q$ and $\dot Q$ by exactly one, as $m\neq 0$. Hence, $P(n)$ has degree either 1 or 2. Notably, $P(n)$ has $m$ dependence, but is still integer valued. By the Cauchy-Schwarz inequality, we arrive at
\begin{align*}
\sup_t\left|\frac{1}{N}\sum_{n=1}^N e^{2\pi i n t} f_a(T^n_\alpha x)\overline{f_b(T_\alpha^{n+h}x)} \right|^2 \leq \frac{2}{N} + 4\left( \frac{1}{N}\sum_{m=1}^{N-1} \left|\frac{1}{N}\sum_{n=1}^{N-m} e^{2\pi i [P(n)\alpha + \dot P_x(n)]}\right|^2\right)^{1/2} \, .
\end{align*}
Taking the integral of both sides in $x$, we have
\begin{align*}
\left\Vert \sup_t\left|\frac{1}{N}\sum_{n=1}^N e^{2\pi i n t} f_a \circ T^n_\alpha \cdot \overline{f_b\circ T_\alpha^{n+h}} \right| \right\Vert_2^2 \leq \int \frac{2}{N} + 4\left( \frac{1}{N}\sum_{m=1}^{N-1} \left|\frac{1}{N}\sum_{n=1}^{N-m} e^{2\pi i [P(n)\alpha + \dot P_x(n)]}\right|^2\right)^{1/2} d\mu(x) \, .
\end{align*}
Let us denote $g_N(\alpha, h) := \left\Vert \sup_t\left|\frac{1}{N}\sum_{n=1}^N e^{2\pi i n t} f_a \circ T^n_\alpha\overline{f_b\circ T_\alpha^{n+h}} \right| \right\Vert_2^2$. If we integrate both sides in $\alpha$, notice that since everything is uniformly bounded, we can use Fubini-Tonelli to switch the order of integration:
\begin{align*}
\int g_N(\alpha, h) \, d\alpha \leq \int \frac{2}{N} + 4\left( \frac{1}{N}\sum_{m=1}^{N-1}  \frac{1}{N^2} \int \left|\sum_{n=1}^{N-m} e^{2\pi i [P(n)\alpha + \dot P_x(n)] \alpha}\right|^2 \, d\alpha \right)^{1/2} d\mu(x) \, .
\end{align*}
On the inside integral, we expand as a double sum. Since the value itself is positive, we can retain an absolute value on the outside:
\begin{align*}
\int \left|\sum_{n=1}^{N-m} e^{2\pi i [P(n)\alpha + \dot P_x(n)] \alpha}\right|^2 \, d\alpha &= \left| \int \sum_{n, j}^{N-m} e^{2\pi i [(P(n) - P(j))\alpha + (\dot P_x(n) - \dot P_x(j))]} \, d\alpha\right | \\
&= \left| \sum_{n, j}^{N-m} e^{2\pi i [\dot P_x(n) - \dot P_x(j)]}\int e^{2\pi i (P(n) - P(j))\alpha } \, d\alpha\right | \\
&\leq \sum_{n, j}^{N-m} \left| \int e^{2\pi i (P(n) - P(j))\alpha }  \, d\alpha  \right| \\
&\leq \sum_{j=1}^N \left(\sum_{n=1}^{N} \left| \int e^{2\pi i (P(n) - P(j))\alpha  }\, d\alpha \right |  \right) \, .
\end{align*}
Recall that $P$ depends on $m, h$. Since $h$ is fixed and we are looking inside the sum in $m$, the only nontrivial dependence of $P$ is in the variables $n$ and $j$. Inside the parenthesis, $j$ is fixed, and $P(n) - P(j)$ is an integer-valued polynomial of degree 1 or 2. Hence, there are at most 2 values of $n$ where $P(n) - P(j) = 0$. Since $\int e^{2\pi i n k \alpha} \, d\alpha$ is 0 for $k \neq 0$ and 1 otherwise, it follows that the inner $n$ sum is bounded by 2, and the whole term bounded by $2N$.

Substituting this into the original estimate, we see
\begin{align*}
\int g_N(\alpha, h) \, d\alpha \leq \int \frac{2}{N} + 4\left( \frac{1}{N}\sum_{m=1}^{N-1}  \frac{1}{N^2} (2N) \right)^{1/2} d\mu(x) \leq \frac{2}{N} + \frac{4\sqrt{2}}{N^{1/2}} \leq \frac{C}{N^{1/2}}
\end{align*}
for $C = 2 + 4\sqrt{2}$. Hence, if we take $N = M^4$, we have 
$$\int M^{1/2}g_{M^4}(\alpha, h) \, d\alpha \leq \frac{C}{M^{3/2}} \, .$$
Recall that this estimate holds for all $h$ but the one $h_0$. Bounding $g_{M^4}(\alpha, h_0)$ trivially by 1, we get the estimate $\int M^{1/2}g_{M^4}(\alpha, h_0) \, d\alpha \leq M^{1/2}$. Hence, for $M^2 > h_0$, we have the following:
\begin{align*}
\int \frac{M^{1/2}}{M^2}\sum_{h=1}^{M^2} g_{M^4}(\alpha, h) \, d\alpha &= \frac{1}{M^2}\sum_{h=1}^{M^2} \int M^{1/2}g_{M^4}(\alpha, h) \, d\alpha \\
&= \frac{1}{M^2}\int M^{1/2} g_{M^4}(\alpha, h_0) \, d\alpha + \frac{1}{M^2}\sum_{h=1, h \neq h_0}^{M^2} \int M^{1/2}g_{M^4}(\alpha, h) \, d\alpha \\
&\leq \frac{1}{M^2}M^{1/2} + \frac{1}{M^2}\sum_{h=1, h \neq h_0}^{M^2} \frac{C}{M^{3/2}} \leq \frac{1+C}{M^{3/2}} \, .
\end{align*}
The monotone convergence theorem then tells us that $\int \sum_{M} \frac{M^{1/2}}{M^2}\sum_{h=1}^{M^2} g_{M^4}(\alpha, h) \, d\alpha < \infty$. So for almost all $\alpha$, the term $\frac{M^{1/2}}{M^2}\sum_{h=1}^{M^2} g_{M^4}(\alpha, h)$ goes to zero, and is hence bounded by some constant $C(\alpha)$. So 
$$\frac{1}{M^2}\sum_{h=1}^{M^2} g_{M^4}(\alpha, h) \leq \frac{C(\alpha)}{M^{1/2}}$$
holds for all $M$. Substituting the term $g$ back in, we have 
$$\frac{1}{M^2}\sum_{h=1}^{M^2} \left\Vert \sup_t\left|\frac{1}{M^4}\sum_{n=1}^{M^4} e^{2\pi i n t} f_a \circ T^n_\alpha \cdot \overline{f_b\circ T_\alpha^{n+h}} \right| \right\Vert_2^2\leq \frac{C(\alpha)}{M^{1/2}} \, .$$
The extension from $M^4$ to $N$ must be done on both the $h$ and $n$ sums simultaneously, but is standard (see \cite[Proposition 8]{assani} or Appendix \ref{a:sp}) and yields a term $\frac{C'(\alpha)}{(N^{1/4})^{1/2}} = \frac{C'(\alpha)}{N^{1/8}}$ for an increased $C'(\alpha) = 4C(\alpha) + 788$.

To account for the $2/3$ exponent in the definition, we apply Hölder's inequality on averages:
\begin{align*}
&\frac{1}{\lfloor \sqrt{N} \rfloor}\sum_{h=1}^{\lfloor \sqrt{N} \rfloor} \left\Vert \sup_t\left|\frac{1}{N}\sum_{n=1}^{N} e^{2\pi i n t} f_a \circ T^n_\alpha \cdot \overline{f_b\circ T_\alpha^{n+h}} \right| \right\Vert_2^{2/3} \\
&\leq \left(\frac{1}{\lfloor \sqrt{N} \rfloor}\sum_{h=1}^{\lfloor \sqrt{N} \rfloor} \left\Vert \sup_t\left|\frac{1}{N}\sum_{n=1}^{N} e^{2\pi i n t} f_a \circ T^n_\alpha \cdot \overline{f_b\circ T_\alpha^{n+h}} \right| \right\Vert_2^{2}\right)^{1/3} \leq \frac{(4C(\alpha) + 788)^{1/3}}{N^{1/24}}
\end{align*}
Hence, for each choice of $a$ and $b$, there exists a set of $\alpha$ full measure where the above holds. Since there are countably many choices for $a$ and $b$, we can find a set of $\alpha$ of full measure where the above holds for every $a$ and $b$. Hence, for all $\alpha$ in this set, we have satisfied the multilinearity conditions, and $(\T^3, T_{\alpha}, \mu)$ is a second-order WW system of order $1/24$ in $L^2$.
\end{proof}

\section{Pointwise convergence of multiple ergodic averages}\label{s:mea}

\subsection{Extending Bourgain's bound on double recurrence}

The following bound can be distilled from the argument in the first section of J. Bourgain's article on double recurrence \cite{bourgain}:

\begin{theorem}[Bourgain's bound on double recurrence]\label{t:BBfor2} Let $(X, \mathcal{F}, \mu, T)$ be an invertible dynamical system, and let $a_1, a_2 \in \Z$ be distinct and both nonzero. Then there exists $C > 0$ such that for every $f_1, f_2\in L^\infty(\mu)$ for which $\max_{j=1, 2}\norm{f_j}_\infty \leq 1$, and for all $N \in \N$, we have
$$\left \Vert \frac{1}{N} \sum_{n=1}^N f_1\circ T^{a_1n} \cdot f_2\circ T^{a_2n} \right\Vert_1 \leq C\left( \frac{1}{N} + \left\Vert \sup_t \left| \frac{1}{N} \sum_{n=1}^N e^{2\pi i n t} f_1\circ T^n\right|\right\Vert_1^{2/3} \right)\, .$$
\end{theorem}
\begin{remark}
Without significant changes to the argument, the Bourgain bound also holds if we replace the 1-norms with 2-norms.
\end{remark}
This is proved following the argument by Assani \cite{AssaniWW}, with subtle but important modifications to maintain the bound for all $N$ without any dependence on $f_1$ or $f_2$. A proof of this bound can be found in Appendix \ref{a:BB}.  We can extend this bound to triple recurrence as follows:
\begin{theorem}\label{BB3} Let $(X, \mathcal{F}, \mu, T)$ be an invertible dynamical system, and let $a_1, a_2, a_3 \in \Z$ be distinct and all nonzero. Then there exists $C > 0$ such that for every $f_1, f_2, f_3\in L^\infty(\mu)$ for which $\max_{j=1, 2, 3}\norm{f_j}_\infty \leq 1$, and for every $N \in \N$ such that $N \geq |a_1|^2$, we have
$$\left \Vert \frac{1}{N} \sum_{n=1}^N \prod_{j=1}^3 f_j \circ T^{a_jn}\right\Vert_1 \leq C\left( \frac{1}{\lfloor \sqrt{N}\rfloor^{1/2}} + \left(\frac{1}{\lfloor\sqrt{N}\rfloor} \sum_{h=1}^{\lfloor\sqrt{N}\rfloor}\left\Vert \sup_t \left| \frac{1}{N} \sum_{n=1}^N e^{2\pi i n t} [f_1 \cdot \overline{f_1 \circ T^{h}}] \circ T^n \right|\right\Vert_1^{2/3} \right)^{1/2} \right) \, .$$ 
\end{theorem}
\begin{remark}
We note that the constant $C$ only depends on the integers $a_1, a_2$, and $a_3$.
\end{remark}
\begin{proof}
Pointwise, we apply the Van der Corput inequality (Lemma \ref{vdc-lem}) to see that for all $1 \leq H<N$, 
\begin{align*}
\left| \frac{1}{N}\sum_{n=1}^N \prod_{j=1}^3 f_j(T^{a_j n} x) \right|^2 &\leq \frac{2}{H+1} + \frac{2(H+N)}{N^2(H+1)^2} \sum_{h=1}^H (H+1 - h) \text{Re}\left( \sum_{n=1}^{N-h} \prod_{j=1}^3 f_j(T^{a_jn}x) \overline{ f_j(T^{a_jn + a_jh}x)}\right) \, .
\end{align*}
Integration yields a squared $L^2$ norm on the left-hand side, which bounds the squared $L^1$ norm of the same term. On the right-hand side, the integral can pass into the $n$ sum, where we can apply the $T$-invariance of $\mu$ by $a_3n$. Bounding $\text{Re}(z)$ by $|z|$, we may clean up to see the following:
\begin{align*}
\left\Vert \frac{1}{N}\sum_{n=1}^N \prod_{j=1}^3 f_j \circ T^{a_j n} \right\Vert_1^2 
&\leq \frac{2}{H+1} + \frac{4}{(H+1)} \sum_{h=1}^H \int \left| \frac{1}{N}\sum_{n=1}^{N-h} \prod_{j=1}^3 f_j \circ T^{(a_j-a_3)n} \cdot \overline{ f_j \circ T^{(a_j-a_3)n + a_jh}}\right| d\mu \, .
\end{align*}
Notice that the $j=3$ term in the product does not depend on $n$. So we can factor it out of the sum, and bound its absolute value away by 1. If we add the terms $n=N-h+1, \dots N$ back into the sum, the leftover term is size $4H/N$. So we have 
\begin{align*}
\left\Vert \frac{1}{N}\sum_{n=1}^N \prod_{j=1}^3 f_j \circ T^{a_j n}  \right\Vert_1^2 
&\leq \frac{2}{H+1} + \frac{4H}{N} + \frac{4}{(H+1)} \sum_{h=1}^H  \left\Vert \frac{1}{N}\sum_{n=1}^{N} \prod_{j=1}^2 [f_j \cdot \overline{f_j \circ T^{a_j h} }] \circ T^{(a_j - a_3)n}\right\Vert_1 \,. 
\end{align*}
Set $H = \lfloor \frac{\sqrt{N}}{|a_1|}\rfloor$, in which both $1/(H+1)$ and $H/N$ are bounded above by $|a_1|\lfloor \sqrt{N} \rfloor^{-1}$. Every term in the $h$ sum is a double recurrence for functions $f_j \cdot \overline{f_j \circ T^{a_j h}}$ with exponents $a_j - a_3$. Hence, on every such term, we can apply the Bourgain bound (Theorem \ref{t:BBfor2}); the constant will depend on the exponents $a_i$, but will be uniform on all functions. Hence, we are left with:
\begin{align*}
\left\Vert \frac{1}{N}\sum_{n=1}^N \prod_{j=1}^3 f_j \circ T^{a_j n} \right\Vert_1^2 
&\leq \frac{6|a_1|}{\lfloor \sqrt{N} \rfloor} + \frac{4|a_1|}{\lfloor \sqrt{N} \rfloor} \sum_{h=1}^{\lfloor \frac{\sqrt{N}}{|a_1|} \rfloor} C\left( \frac {1}{N} + \left\Vert \sup_t \left|\frac{1}{N} \sum_{n=1}^N e^{2\pi i n t} [f_1 \cdot  \overline{f_1 \circ T^{a_1 h}}] \circ T^n \right|\right\Vert_1^{2/3}\right) \,.
\end{align*}
The $1/N$ term can pass out of the sum, and can be absorbed into the $1/\lfloor \sqrt{N} \rfloor$ term:
\begin{align*}
\left\Vert \frac{1}{N}\sum_{n=1}^N \prod_{j=1}^3 f_j \circ T^{a_j n}  \right\Vert_1^2 
&\leq \frac{6|a_1|+4|a_1|C}{\lfloor \sqrt{N} \rfloor} + \frac{4C|a_1|}{\lfloor \sqrt{N} \rfloor} \sum_{h=1}^{\lfloor \frac{\sqrt{N}}{|a_1|} \rfloor}  \left\Vert \sup_t \left|\frac{1}{N} \sum_{n=1}^N e^{2\pi i n t} [f_1 \cdot \overline{f_1\circ T^{a_1 h}}] \circ T^n \right|\right\Vert_1^{2/3} \, .
\end{align*}
By adding in some nonnegative terms, we can remove the $a_1$ on the function $f_1 \cdot \overline{f_1 \circ T^{a_1 h}}$, at the cost of expanding the sum in $h$ to $|a_1| \lfloor \frac{\sqrt{N}}{|a_1|} \rfloor$. Since this term is itself less that $\lfloor \sqrt{N} \rfloor$, we have
\begin{align*}
\left\Vert \frac{1}{N}\sum_{n=1}^N \prod_{j=1}^3 f_j \circ T^{a_j n} \right\Vert_1^2 
&\leq \frac{6|a_1|+4|a_1|C}{\lfloor \sqrt{N} \rfloor} + \frac{4C|a_1|}{\lfloor \sqrt{N} \rfloor} \sum_{h=1}^{\lfloor \sqrt{N} \rfloor}  \left\Vert \sup_t \left|\frac{1}{N} \sum_{n=1}^N e^{2\pi i n t} [f_1 \cdot \overline{f_1\circ T^{ h}}] \circ T^n  \right|\right\Vert_1^{2/3} \, .
\end{align*}
Pulling out the constants, taking the square root of both sides, and using subadditivity of the square root, we get the desired statement.
\end{proof}

\begin{remark}
Similarly to the Bourgain bound (Theorem \ref{t:BBfor2}), the above holds true if we replace the 1-norms with 2-norms, without significant changes to the argument.
\end{remark}


For multiple recurrence averages of four terms, the same application of the Van der Corput inequality bounds the $L^1$ norm in terms of triple recurrence averages. Continuing inductively, we reach the following estimate for higher recurrence.
\begin{theorem}\label{BBgeq3} 
Let $J \in \N$ such that $J \geq 3$. Let $(X, \mathcal{F}, \mu, T)$ be an invertible dynamical system, and for every $j \in [J]$, let $a_j \in \Z$ be distinct and all nonzero. Then there exist constants $C_J > 0$ and $N_J \in \N$ such that for every $f_1, f_2, \ldots, f_J \in L^\infty(\mu)$ such that $\max_{j \in [J]}\norm{f_j}_\infty \leq 1$, and for every $N \in \N$ such that $N \geq N_J$, we have   
\begin{align*}
&\left \Vert \frac{1}{N} \sum_{n=1}^N \prod_{j=1}^J f_j \circ T^{a_jn} \right\Vert_1 \leq \\
 &C_J\left( \frac{1}{\lfloor \sqrt{N}\rfloor^{1/2^{J-2}}} + \left(\frac{1}{\lfloor\sqrt{N}\rfloor^{J-2}} \sum_{h \in [\lfloor\sqrt{N}\rfloor]^{J-2}}\left\Vert \sup_t \left| \frac{1}{N} \sum_{n=1}^N e^{2\pi i n t} \left[ \prod_{\eta \in V_{J-2}}  c^{|\eta|} f_1 \circ T^{h \cdot \eta}\right]\circ T^n\right|\right\Vert_1^{2/3} \right)^{1/2^{J-2}} \right) \, .
\end{align*}
\end{theorem}

\begin{remark}
The constants $C_J$ and $N_J$ are independent of the functions $f_1, f_2, \ldots, f_J$ and system $(X, \mathcal{F}, \mu, T)$, and indeed only depend on $J$ and the exponents $a_1, \dots, a_j$. In Appendix \ref{a:BB}, we compute a value of $C_2$.
\end{remark}

\begin{remark}
By reordering the functions, it follows that the multiple recurrence averages can be bounded by any $f_j$.
\end{remark}

\begin{proof}[Proof of Theorem \ref{BBgeq3}:]
The previous theorem is the base case. To induct, the same argument gets us to the point 
\begin{align*}
\left\Vert \frac{1}{N}\sum_{n=1}^N \prod_{j=1}^J f_j \circ T^{a_j n} \right\Vert_1^2 
&\leq \frac{6|a_1|}{\lfloor \sqrt{N} \rfloor} + \frac{4}{\lfloor \frac{\sqrt{N}}{|a_1|} \rfloor} \sum_{q=1}^{\lfloor \frac{\sqrt{N}}{|a_1|}\rfloor} \left\Vert \frac{1}{N}\sum_{n=1}^{N} \prod_{j=1}^{J-1} [f_j \cdot \overline{f_j \circ T^{a_j q} }] \circ T^{(a_j - a_J)n}\right\Vert_1
\end{align*}
for $N \geq |a_1|^2$. Every summand in $q$ is a $J-1$ recurrence bound, all with the same nonzero exponents. Now apply the inductive step to see for $N \geq N_J :=\max\{|a_1|^2, N_{J-1}\}$ to see
\begin{align*}
\left\Vert \frac{1}{N}\sum_{n=1}^N \prod_{j=1}^J f_j\circ T^{a_j n} \right\Vert_1^2 \leq \frac{6|a_1|}{\lfloor \sqrt{N} \rfloor} + \frac{4}{\lfloor \frac{\sqrt{N}}{|a_1|} \rfloor} \sum_{q=1}^{\lfloor \frac{\sqrt{N}}{|a_1|}\rfloor} C_{J-1}\Bigg( \frac{1}{\lfloor \sqrt{N}\rfloor^{1/2^{J-3}}} + \Bigg(\frac{1}{\lfloor\sqrt{N}\rfloor^{J-3}} \sum_{h \in [\lfloor\sqrt{N}\rfloor]^{J-3}} \\ \left\Vert \sup_t \left| \frac{1}{N} \sum_{n=1}^N e^{2\pi i n t} \left[ \prod_{\eta \in V_{J-3}}  c^{|\eta|} (f_1 \cdot \overline{f_1\circ T^{a_1 q}} )\circ T^{h \cdot \eta}\right](T^n x)\right|\right\Vert_1^{2/3} \Bigg)^{1/2^{J-3}}\Bigg) \, .
\end{align*}
Similarly to before we can pull out the $\lfloor \sqrt{N} \rfloor^{-1/2^{J-3}}$ term, and the $\lfloor \sqrt{N} \rfloor^{-1}$ can be absorbed into it. By Hölder's inequality on averages, we can pull the $q$ sum into the $1/2^{J-2}$ power, which leaves us with
\begin{align*}
&\left\Vert \frac{1}{N}\sum_{n=1}^N \prod_{j=1}^J f_j \circ T^{a_j n} \right\Vert_1^2 \leq \frac{6|a_1| + 4|a_1|C_{J-1}}{\lfloor \sqrt{N}\rfloor^{1/2^{J-3}}}  \\& + C_{J-1}\Bigg( \frac{4}{\lfloor \frac{\sqrt{N}}{|a_1|} \rfloor} \sum_{q=1}^{\lfloor \frac{\sqrt{N}}{|a_1|}\rfloor} \frac{1}{\lfloor\sqrt{N}\rfloor^{J-3}} \sum_{h \in [\lfloor\sqrt{N}\rfloor]^{J-3}} \left\Vert \sup_t \left| \frac{1}{N} \sum_{n=1}^N e^{2\pi i n t} \left[ \prod_{\eta \in V_{J-3}}  c^{|\eta|} (f_1 \cdot \overline{f_1\circ T^{a_1 q}} )\circ T^{h \cdot \eta}\right](T^n x)\right|\right\Vert_1^{2/3} \Bigg)^{1/2^{J-3}} \, .
\end{align*}
But the $f_1 \cdot f_1 \circ T^{a_1 q}$ inside the product along the cube adds one more layer, but with the coefficient of $q$ scaled by $a_1$. By the same argument as before, we can get rid of this scaling by extending the bound on $q$ to $\lfloor \sqrt{N} \rfloor$. Since getting rid of the $\lfloor \frac{\sqrt{N}}{|a_1|} \rfloor$ in the denominator only picks up a constant $|a_1|$, we have exactly constructed a cube on $(h, q) \in [\lfloor \sqrt{N} \rfloor]^{J-3}\times [\lfloor \sqrt{N} \rfloor] = [\lfloor \sqrt{N} \rfloor]^{J-2}$. Specifically, we have
\begin{align*}
&\left \Vert \frac{1}{N} \sum_{n=1}^N \prod_{j=1}^J f_j \circ T^{a_jn} \right\Vert_1^2 \leq \\
 & \frac{6|a_1|+4|a_1|C_{J-1}}{\lfloor \sqrt{N}\rfloor^{1/2^{J-3}}} + \left(\frac{4C_{J-1}|a_1|}{\lfloor\sqrt{N}\rfloor^{J-2}} \sum_{h \in [\lfloor\sqrt{N}\rfloor]^{J-2}}\left\Vert \sup_t \left| \frac{1}{N} \sum_{n=1}^N e^{2\pi i n t} \left[ \prod_{\eta \in V_{J-2}}  c^{|\eta|} f_1 \circ T^{h \cdot \eta}\right]\circ T^n \right|\right\Vert_1^{2/3} \right)^{1/2^{J-3}} \, . 
\end{align*}
Pulling out constants, taking the square root of both sides, and using subadditivity again gives us the desired statement.
\end{proof}

\subsection{Application to pointwise characteristic factors}

For a first order WW dynamical system, pointwise convergence double recurrence averages follows from just by the monotone convergence theorem, as seen in \cite{assani98}. With the higher order recurrence bound we have just established, we get the following:

\begin{lemma}\label{l:mea_cf}
Let $(X, \mathcal{F}, \mu, T)$ be an invertible dynamical system. Suppose that $(f_m)$ is a sequence of $k$-th order WW functions of power type $\alpha$ in $L^p$ that converges in $L^2$ to $f\in L^\infty$. Then for all $g_1, \dots, g_{k}\in L^\infty(\mu)$ and $a_1, \dots, a_{k+1} \in \N$, all distinct and nonzero, we have
\begin{align}\label{pt.conv.to.zero}
\lim_{N} \frac{1}{N} \sum_{n=1}^N \prod_{j=1}^k g_j(T^{a_j n}x)f(T^{a_{k+1}n}x) = 0
\end{align}
for $\mu$-a.e. $x$.
\end{lemma}

\begin{proof}
Without loss of generality, take $\Vert f \Vert_\infty \leq 1$ and all $\Vert g_i \Vert_\infty \leq 1$. In the case that $f$ itself is a $k$-th order WW function of power type $\alpha$ in $L^p$, the higher-order Bourgain bound (Theorem \ref{BBgeq3}) yields
$$\left \Vert \frac{1}{N} \sum_{n=1}^N \prod_{j = 1}^k g_j\circ T^{a_j n} \cdot f\circ T^{a_{k+1} n} \right\Vert_1 \leq C_J\left(\frac{1}{N^{1/2^{k-2}}} + \left(\frac{C_{f}}{N^\alpha} \right)^{1/2^{k-1}} \right)\leq \frac{C_{f}'}{N^\beta} $$
where $N$ is sufficiently large and $\beta = \min\{1/2^{k-2}, \alpha/2^{k-1}\}$. Choose $\gamma \in \R$ such that $\beta \gamma > 1$. By taking a subsequence of the form $N = \lfloor M^\gamma \rfloor$, and if we sum over $M$, we have
\begin{equation*}
   \sum_{M=1}^\infty \left \Vert \frac{1}{\floor{M^\gamma}} \sum_{n=1}^{\floor{M^\gamma}} \prod_{j = 1}^k g_j\circ T^{a_j n} \cdot f\circ T^{a_{k+1} n} \right\Vert_1 \leq \sum_{M=1}^\infty \frac{C'_f}{\floor{M^{\gamma}}^\beta} < \infty.
\end{equation*}
The monotone convergence theorem tells us that for $\mu$-a.e. $x \in X$, we have
\[ \sum_{M=1}^\infty \left| \frac{1}{\floor{M^\gamma}} \sum_{n=1}^{\floor{M^\gamma}} \prod_{j = 1}^k g_j\circ T^{a_j n}(x) f\circ T^{a_{k+1} n}(x)\right| < \infty \, , \]
which implies that for $\mu$-a.e. $x \in X$, we have
\[ \lim_{M \to \infty} \left| \frac{1}{\floor{M^\gamma}} \sum_{n=1}^{\floor{M^\gamma}} \prod_{j = 1}^k g_j\circ T^{a_j n}(x) f\circ T^{a_{k+1} n}(x)\right| = 0 \, .  \]
The remaining argument of showing $(\ref{pt.conv.to.zero})$ when $f$ is a WW function (i.e. along $N$ instead of the subsequence $\floor{M^\gamma}$) is standard.

For $f_m \to f$ in $L^2$, we assume $\Vert f_m - f \Vert_2 \leq m^{-2}$. By the maximal inequality on averaging operators (Lemma \ref{l:maxIneq}), it follows for each $m\in \N$ that
\begin{align*}
&\left\Vert \sup_N \left| \frac{1}{N} \sum_{n=1}^N \prod_{j=1}^k g_j\circ T^{a_j n} \cdot f\circ T^{a_{k+1}n} - \frac{1}{N} \sum_{n=1}^N \prod_{j=1}^k g_j\circ T^{a_j n} \cdot f_m\circ T^{a_{k+1}n}\right|\right\Vert_1 \\
&\leq \left\Vert \sup_N \frac{1}{N} \sum_{n=1}^N \left|f - f_m\right|\circ T^{a_{k+1}n}\right\Vert_2 \leq 2\Vert f - f_m \Vert_2 \leq \frac{2}{m^2}
\end{align*}
Hence, these terms are summable in $m$. Pushing the sum inside the integral by the monotone convergence theorem, we get pointwise almost everywhere convergence on integrand. Since the sum in $m$ converges to zero, its tail converges to zero, or
$$\lim_m \sup_N \left| \frac{1}{N} \sum_{n=1}^N \prod_{j=1}^k g_j(T^{a_j n}x)f(T^{a_{k+1}n}x) - \frac{1}{N} \sum_{n=1}^N \prod_{j=1}^k g_j(T^{a_j n}x)f_m(T^{a_{k+1}n}x)\right| = 0$$
for almost all $x$. We note for any $m$ that
\begin{align*}
&\limsup_N \left| \frac{1}{N} \sum_{n=1}^N \prod_{j=1}^k g_j(T^{a_j n}x)f(T^{a_{k+1}n}x)\right| \\
&\leq \limsup_N \left| \frac{1}{N} \sum_{n=1}^N \prod_{j=1}^k g_j(T^{a_j n}x)f(T^{a_{k+1}n}x) - \frac{1}{N} \sum_{n=1}^N \prod_{j=1}^k g_j(T^{a_j n}x)f_m(T^{a_{k+1}n}x)\right|  + 0\, .
\end{align*}
Bounding the $\limsup$ by the $\sup$ and taking the limit in $m$, we get the desired statement.
\end{proof}
\begin{remark} 
For decay rates $\log(N+1)^{- \beta}$ with $\beta > 1$, the same standard extension argument works by taking subsequences of the form $\floor {r^M}$ and letting $r \to 1$. Hence, a similar argument can be used to show convergence for WW functions of logarithmic decay. However, since the exponent must be greater than 1 and the Bourgain bound decreases the exponent in the transfer, the WW averages of the function $f$ would have to decay a rate $\log(N+1)^{- \beta}$ for $\beta > 2^{k-1}$ to obtain the same result.

\end{remark}

\begin{remark}
This proof does not require that the limiting function $f$ is itself a $k$-th order WW function. As the constants $C_{f_m}$ for the polynomial rate of decay could be unbounded as $m\to \infty$, the rate of decay for WW averages of $f_m$ does not immediately transfer to $f$. Since generically, sequences $\{a_n\}_{n\in \mathbb{N}}$ of good decay can converge in $\ell^\infty(\mathbb{N})$ to a sequence of worse decay, it is not necessarily expected that higher-order WW functions should form a closed subset of $L^2(\mu)$.

Furthermore, in the case of weakly mixing Lebesgue systems (such as K systems), it cannot be the case that every function in $L^2(\mathcal{Z}_{k-1})^\perp$ is a $k$-th order WW function of some power type. As the $k$-th order WW averages bound $k$ multiple recurrence, they also bound the ``single" recurrences, which are just the classical Birkhoff averages. Since it was shown by Krengel \cite{Krengel} that in every Lebesgue system the Birkhoff averages have no uniform rate of decay, either pointwise or in norm, there must be some function whose Birkhoff average converges at at least a logarithmic rate. In a weakly mixing system, this function must lie in $L^2(\mathcal{Z}_{k-1})^\perp$ for all $k$, and by its rate of decay it is not a $k$-th order WW function for any $k$.

\end{remark}

As the previous result holds pointwise, it is sufficient to create pointwise characteristic factors for multiple recurrence: 

\begin{theorem}\label{t:wwchar}
Let $(X, \mathcal{F}, \mu, T)$ be an invertible dynamical system, and let $\mathcal{B}$ be a $\sigma$-subalgebra of $\mathcal{F}$. Suppose that there exists in $L^2(\mathcal{B})^\perp$ an $L^2$-dense subset of $k$-th order WW functions of power type $\alpha$ for $L^p$. Then $(X, \mathcal{B}, \mu, T)$ is a pointwise characteristic factor for $k+1$-multiple recurrence: for all $f_1, \dots, f_{k+1}\in L^\infty(\mu)$, we have
\begin{align*}
\lim_N \left| \frac{1}{N} \sum_{n=1}^N \prod_{j=1}^{k+1} f_j(T^{jn} x) - \frac{1}{N} \sum_{n=1}^N \prod_{j=1}^{k+1} \E(f_j | \mathcal{B}) (T^{jn} x)\right| = 0
\end{align*}
for $\mu$-a.e. $x\in X$.
\end{theorem}

\begin{proof}
For every $j \in [k+1]$, we denote $f_j^\mathcal{B} := \E(f_j|\mathcal{B})$ and $f^\perp_j := f_j - \E(f_j|\mathcal{B})$. If we decompose each $f_j = f^\mathcal{B}_j + f^\perp_j$, we note that both pieces are bounded as conditional expectation preserves the $L^\infty$ norm. We write
$$ \left|\frac{1}{N} \sum_{n=1}^N \prod_{j=1}^{k+1} f_j(T^{jn} x) - \frac{1}{N} \sum_{n=1}^N \prod_{j=1}^{k+1} f^\mathcal{B}_j(T^{jn} x) \right|= \left| \frac{1}{N} \sum_{n=1}^N \prod_{j=1}^{k+1} \left[f_j^\mathcal{B}(T^{jn} x) + f_j^\perp (T^{jn}x)\right] - \frac{1}{N} \sum_{n=1}^N \prod_{j=1}^{k+1} f^\mathcal{B}_j(T^{jn} x)\right| \, .$$
If we expand the product in $j$, we get $2^{k+1}$ terms. The multiple recurrence average over the $f^{\mathcal{B}}_j$'s will cancel, and all that remains are $2^{k+1}-1$ terms of the form (\ref{pt.conv.to.zero}), and vanish under the limit.
\end{proof}
Applying this to Theorem \ref{t:Ksys}, we can immediately recover the fact that for any invertible dynamical system $(X, \mathcal{F}, \mu, T)$, the Pinkser subalgebra $(X, \mathcal{P}, \mu, T)$ (where $\mathcal{P}$ is from (\ref{Pinsker})) is a pointwise characteristic factor for multiple recurrence: i.e., for pointwise convergence in multiple recurrence averages, it suffices to consider systems of zero entropy \cite[Proposition 4]{assani98}. 


If $(X, \mathcal{F}, \mu, T)$ is a K system and $(Y, \mathcal{G}, \nu, S)$ is any other dynamical system, we showed in Proposition \ref{t:Kprod} that there exists a dense set of WW functions in $L^2(X \otimes \mathcal{G})^\perp$. Hence, it follows that $(Y, \mathcal{G}, \nu, S)$ is a characteristic factor for multiple recurrence in the product system $(X \times Y, \mathcal{F} \otimes \mathcal{G}, \mu \times \nu, T \times S)$: i.e. if $X$ is a K system and $Y$ satisfies pointwise convergence for multiple recurrence, then so does $X \times Y$.


For $J$-order WW dynamical systems, it follows that $\mathcal{Z}_{J-1}$ is a pointwise characteristic factor for $J$-multiple recurrence. Since $T$ inside $\mathcal{Z}_{J-1}$ acts like translation on a $J-1$ pro-nilmanifold, as shown by Host-Kra  \cite{hostkra}, and pointwise convergence of these transformations has been shown by Leibman \cite{leibman05a}, it follows that we have pointwise convergence for all $f_j \in L^\infty$:

\begin{corollary}\label{AEMultRec} Let $(X, \mathcal{F}, \mu, T)$ be an invertible $J-1$-th order WW system
of power type $\alpha$ in $L^p$ for some $p \in [1, \infty]$, and $f_1, \dots, f_J \in L^\infty$. Then 
$$\lim_{N} \frac{1}{N} \sum_{n=1}^N \prod_{j = 1}^J f_j(T^{jn}x) $$
converges for almost all $x$. Specifically, the characteristic factor of pointwise convergence is $\mathcal{Z}_{J-1}$.
\end{corollary}

\begin{remark}
    We note that the pointwise convergence of multiple recurrence for the K systems is proven by Derrien and Lesigne \cite{DL1996}, and for the classical skew products, one may apply Leibman's result since they are nilsystems \cite{leibman05}. We will see, in Theorem \ref{t:productWWS}, another example of a higher order WW system that is neither weakly mixing nor distal.
\end{remark}

\subsection{Uniform Wiener-Wintner theorem for multiple ergodic averages}
In this subsection, we will show that one can prove a uniform Wiener-Wintner theorem for multiple ergodic averages on a higher order WW system.
\begin{theorem}\label{t:uww_mea}
Let $(X, \mathcal{F}, \mu, T)$ be a $k$-th order WW system of power type $\alpha>0$ in $L^p$ for $k \geq 2$. Suppose $f_1, \dots f_k\in L^\infty(\mu)$ and $f_1\in L^2(\mathcal{Z}_{k-1})^\perp$. Then for all $a_1, \dots, a_{k}\in \Z$ distinct and nonzero, we have
\begin{align*}
\lim_N \sup_{t} \left| \frac{1}{N} \sum_{n=1}^N e^{2\pi i n t} \prod_{j=1}^k f_j(T^{a_j n} x)\right| = 0
\end{align*}
for $\mu$-a.e. $x\in X$.
\end{theorem}

\begin{remark}
   For general ergodic systems, the uniform WW theorem is proven in \cite{ADM2016} for $k = 2$. Later, Zorin-Kranich announced a uniform WW result for any $k \in \N$ with nilsequences \cite{ZK2015}. In those results, the pointwise convergence of $k$-recurrence averages was assumed (partly to be able to apply the dominated convergence theorem to switch the limit and the integral). As we have shown that $k$-th order WW systems satisfy pointwise convergence for $k$-recurrence, the full conclusion of \cite{ZK2015} follows for these systems. However, we remark that we can obtain this simpler version without appealing to Corollary \ref{AEMultRec}.
\end{remark}

\begin{proof}
The idea of the proof is similar to that of Lemma \ref{l:mea_cf}, with additional application of the Van der Corput lemma.

Consider the case in which $f_1$ is a $k$-th order WW function of power type $\alpha$ in $L^p$. Applying the Van der Corput inequality (\ref{vdc-2}) pointwise for $N^2 \geq |a_1|$ and $H = \floor{\frac{\sqrt{N}}{|a_1|}}$, we see that 
\begin{align*}
\sup_{t}\left|\frac{1}{N}\sum_{n=1}^N e^{2\pi i n t} \prod_{j=1}^k f(T^{a_jn}x) \right|^2 &\leq \frac{2|a_1|}{\floor{\sqrt{N}}} + \frac{4}{\floor{\frac{\sqrt{N}}{|a_1|}}} \sum_{q=1}^{\floor{\frac{\sqrt{N}}{|a_1|}}} \left| \frac{1}{N} \sum_{n=1}^{N-q} \prod_{j=1}^k f_j(T^{a_j n}x) \overline{f_j(T^{(a_j + q)n}x)}\right| \, .
\end{align*}
As in previous results, we can easily extend the last average from $N-q$ to $N$ (at the cost of adding an $O(N^{-1/2})$-term). Integrating both sides, we get a $k$-recurrence average on the right which we may bound with the $k$-th order Bourgain bound for sufficiently large $N$: By H\"older's inequality on averages, we get
\begin{align*}
&\left\Vert \sup_{t}\left|\frac{1}{N}\sum_{n=1}^N e^{2\pi i n t} \prod_{j=1}^k f(T^{a_jn}x) \right|\right\Vert_{2}^2 \leq \frac{C}{\floor{\sqrt{N}}} + \frac{C}{\floor{\frac{\sqrt{N}}{|a_1|}}} \sum_{q=1}^{\floor{\frac{\sqrt{N}}{|a_1|}}} \Bigg( \frac{1}{\lfloor \sqrt{N}\rfloor^{1/2^{k-2}}} \\ & + \left(\frac{1}{\lfloor\sqrt{N}\rfloor^{k-2}} \sum_{h \in [\lfloor\sqrt{N}\rfloor]^{k-2}}\left\Vert \sup_t \left| \frac{1}{N} \sum_{n=1}^N e^{2\pi i n t} \left[ \prod_{\eta \in V_{k-2}}  c^{|\eta|} [f_1 \cdot \overline{f_1\circ T^{a_1 q}}] \circ T^{h \cdot \eta}\right]\circ T^n\right|\right\Vert_1^{2/3} \right)^{1/2^{k-2}} \Bigg) \\
&\leq \frac{C}{\lfloor \sqrt{N}\rfloor^{1/2^{k-2}}} + \\ &\left( \frac{C}{\floor{\frac{\sqrt{N}}{|a_1|}}} \sum_{q=1}^{\floor{\frac{\sqrt{N}}{|a_1|}}} \frac{C}{\lfloor\sqrt{N}\rfloor^{k-2}} \sum_{h \in [\lfloor\sqrt{N}\rfloor]^{k-2}}\left\Vert \sup_t \left| \frac{1}{N} \sum_{n=1}^N e^{2\pi i n t} \left[ \prod_{\eta \in V_{k-2}}  c^{|\eta|} [f_1 \cdot \overline{f_1\circ T^{a_1 q}}] \circ T^{h \cdot \eta}\right]\circ T^n\right|\right\Vert_1^{2/3} \right)^{1/2^{k-2}} \, .
\end{align*}
When we remove the $a_1$ scaling on $q$ by extending its range to $\floor{\sqrt{N}}$, we note that we have constructed the $k$-th WW average on the function $f_1$. Since this is a $k$-th order WW function, we have
\begin{align*}
\left\Vert \sup_{t}\left|\frac{1}{N}\sum_{n=1}^N e^{2\pi i n t} \prod_{j=1}^k f_j(T^{a_jn}x) \right|\right\Vert_{1}^2 &\leq  \frac{C}{\lfloor \sqrt{N}\rfloor^{1/2^{k-2}}} + \left(\frac{C_{f_1}}{N^{\alpha}}  \right)^{1/2^{k-2}} \leq \frac{C}{N^\beta}
\end{align*}
for $\beta = \min \{1/2^{k-3}, \alpha/2^{k-2}\}$ and sufficiently large $N$. Hence, by taking a subsequence of the form $\floor{M^\gamma}$, where $\gamma \in \R$ such that $\gamma \beta > 1$, we get pointwise a.e. convergence to zero.

If $g_m \to f_1$ in $L^2$ are a sequence of $k$-th order WW functions, we take $\Vert f - g_m \Vert_2 \leq m^{-2}$, and we apply the maximal inequality (Theorem \ref{l:maxIneq}) to see that
\begin{align*}
&\left\Vert \sup_N \sup_t\left| \frac{1}{N} \sum_{n=1}^N e^{2\pi i n t}\prod_{j=1}^k f_j\circ T^{a_j n} - \frac{1}{N} \sum_{n=1}^N e^{2\pi i n t} g_m\circ T^{a_j n} \prod_{j=2}^k f_j\circ T^{a_j n} \right|\right\Vert_2 \\
&\leq \left\Vert \sup_N \frac{1}{N} \sum_{n=1}^N \left|f_1 - g_m\right|\circ T^{a_{1}n}\right\Vert_2 \leq 2\Vert f_1 - g_m \Vert_2 \leq \frac{2}{m^2} \, .
\end{align*}
Hence, these terms are summable in $m$. Pushing the sum inside the integral by the monotone convergence theorem, we get pointwise almost everywhere convergence on the integrand. Since the sum in $m$ converges to zero, its tail converges to zero, or
$$\lim_m \sup_N \sup_{t}\left| \frac{1}{N} \sum_{n=1}^N e^{2\pi i n t}\prod_{j=1}^k f_j(T^{a_j n}x) - \frac{1}{N} \sum_{n=1}^N e^{2\pi i n t}g_m(T^{a_1 n}x)\prod_{j=2}^k f_j(T^{a_j n}x)\right| = 0$$
for almost all $x$. We note for any $m$ that
\begin{align*}
&\limsup_N \sup_t \left|  \frac{1}{N} \sum_{n=1}^N e^{2\pi i n t}\prod_{j=1}^k f_j(T^{a_j n}x)\right| \\
&\leq \limsup_N \sup_t \left| \frac{1}{N} \sum_{n=1}^N e^{2\pi i n t}\prod_{j=1}^k f_j(T^{a_j n}x) - \frac{1}{N} \sum_{n=1}^N e^{2\pi i n t}g_m(T^{a_1 n}x)\prod_{j=2}^k f_j(T^{a_j n}x)\right|  + 0\, .
\end{align*}
Bounding the $\limsup$ by the $\sup$ and taking the limit in $m$, we get the desired statement.
\end{proof}

\begin{remark}
Since the same reasoning applies for general $\sigma$-subalgebras $\mathcal{B}$ rather than $\mathcal{Z}_{k-1}$, we can show an analogous statement to Theorem \ref{t:wwchar} for uniform Wiener-Wintner multiple recurrence averages, i.e., if there exists a dense set of WW functions in $L^2(\mathcal{B})^\perp$, then for all $f_1, \dots, f_k\in L^\infty(\mu)$ with $f_1 \in L^2(\mathcal{B})^\perp$ and $a_1, \dots, a_k \in \Z$ distinct and nonzero we have
$$\lim_N \sup_t \left| \frac{1}{N} \sum_{n=1}^N e^{2\pi i n t} \prod_{j=1}^k f_j(T^{a_j n}x) \right| = 0$$
for $\mu$-a.e. $x\in X$. By orthogonally decomposing the $f_j$'s on $L^2(\mathcal{B})$ and using subadditivity, we have for any $f_1 \dots, f_k \in L^2(\mathcal{B})^\perp$ that
$$\limsup_N \sup_t \left| \frac{1}{N} \sum_{n=1}^N e^{2\pi i n t} \prod_{j=1}^k f_j(T^{a_j n}x) \right| \leq \limsup_N \sup_t \left| \frac{1}{N} \sum_{n=1}^N e^{2\pi i n t} \prod_{j=1}^k \mathbb{E}(f_j|\mathcal{B})(T^{a_j n}x)\right| \, .$$
Hence, we can recover similar conclusions about these averages as the ones following Theorem \ref{t:wwchar}: for pointwise convergence to zero of uniform WW multiple ergodic averages in a system $X$, we may project those averages to the Pinkser algebra, and if $X$ is a K system, for pointwise convergence of uniform WW multiple ergodic averages in $X \times Y$ for any $Y$, we may project those averages to $Y$.
\end{remark}

\section{Relationship to Gowers-Host-Kra seminorms}\label{s:GHKsemnorm}

Consider an ergodic dynamical system, and recall that the factor $\mathcal{Z}_{J-1}$ is a \textit{universal} characteristic factor for the norm convergence of $J$ multiple recurrence, which for a bounded function $f$ means that $f\in L^2(\mathcal{Z}_{J-1})^\perp$ if and only if
\begin{align}\label{univCF} \forall f_1, \dots, f_J\in L^\infty(\mu), \text{ with } f_j = f \text{ for some }j, \quad \lim_{N} \left\Vert \frac{1}{N} \sum_{n=1}^N \prod_{j=1}^{J} f_j\circ T^{jn} \right\Vert_2 = 0 \, .
\end{align}
Moreover, we have $f\in L^2(\mathcal{Z}_{J-1})^\perp$ if and only if $\VERT f \VERT_J = 0$, where the seminorms $\VERT \cdot \VERT_J$ are constructed inductively:
\begin{align*} 
 \VERT f \VERT_2^4 &= \lim_H \frac{1}{H} \sum_{h=1}^H \left| \int f \cdot \overline{f\circ T^h  } \, d\mu \right|^2 \\
\VERT f \VERT_3^8 &= \lim_H \frac{1}{H} \sum_{h=1}^H \left\VERT f \cdot \overline{f \circ T^h }\right\VERT_2^4 \\
\VERT f \VERT_{k}^{2^k} &= \lim_H \frac{1}{H} \sum_{h=1}^H \left\VERT f\cdot \overline{f\circ T^h}\right\VERT_{k-1}^{2^{k-1}} 
\end{align*}

In the case $J=1$ of double recurrence, much more is known. Noting that $\VERT f\VERT_2^4$ is the average of the Fourier coefficients of the spectral measure of $f$, it follows from Wiener's theorem that $ \VERT f\VERT_2 = 0$ exactly when the spectral measure of $f$ is continuous, which is true exactly when $f\in \K^\perp$ by the spectral theorem. Hence, the Host-Kra-Ziegler factor $\mathcal{Z}_2$ is equal to the Kronecker factor $\K$. 

The Kronecker factor itself can be characterized pointwise through Bourgain's uniform Wiener-Wintner theorem (Theorem \ref{t:UWW}), which states that $f\in \K^\perp$ if and only if the averages $\sup_t \left| \frac{1}{N}\sum_{n=1}^N e^{2\pi i n t} f(T^nx)\right|$ converge to zero for almost every $x$. Since almost everywhere convergence implies norm convergence by the dominated convergence theorem, and norm convergence implies that double recurrence holds by Bourgain's bound, it follows that all of these statements are equivalent. All of these observations are well-known and can be summarized in the following theorem:
\begin{theorem}[Characterizations of norm double recurrence]
Let $(X, \mathcal{F}, \mu, T)$ be an invertible ergodic dynamical system and $f\in L^\infty(\mu)$. Then the following are equivalent:
\begin{enumerate}
    \item (Norm convergence to zero for double recurrence) $f\in L^2(\mathcal{Z}_1)^\perp$
    \item (Pointwise characterization) $$\lim_N \sup_t \left| \frac{1}{N}\sum_{n=1}^N e^{2\pi i n t} f(T^nx)\right| = 0$$ for almost all $x$.
    \item (Norm characterization) $$\lim_N \left\Vert \sup_t \left| \frac{1}{N}\sum_{n=1}^N e^{2\pi i n t} f \circ T^n\right| \right\Vert_2 = 0$$
\end{enumerate}

\end{theorem}

\begin{remark}
For the pointwise characterization, convergence on a set of positive measure implies convergence almost everywhere (due to Poincar\'e recurrence), and hence is another equivalent characterization.


\end{remark}
As stated before, the seminorm characterization generalizes to multiple recurrence. The extension of Bourgain's bound (Theorem \ref{BBgeq3}) gives some direction to generalize these characterizations to higher recurrence:

\begin{theorem}\label{t:higher.order.UWW}
Let $(X, \mathcal{F}, \mu, T)$ be an invertible ergodic dynamical system and $f\in L^\infty(\mu)$. Let $k \geq 2$. Then the following are equivalent:

\begin{enumerate}
    \item (Norm convergence for $k+1$-multiple recurrence) $f \in L^2(\mathcal{Z}_{k})^\perp$
    \item[(2a)] (Pointwise characterization, multiple functions) For all collections $(g_\eta)_{\eta\in V_{k-1}}$ with $g_\eta\in L^\infty(\mu)$ and some $g_\eta = f$, we have $$\lim_N \frac{1}{N^{k-1}} \sum_{h \in [N]^{k-1}} \sup_{t} \left| \frac{1}{N} \sum_{n=1}^N e^{2\pi i n t} \left[\prod_{\eta\in V_{k-1}} c^{|\eta|} g_\eta \circ T^{h \cdot \eta} \right](T^{n} x) \right|^2 = 0$$ for almost all $x$.
    \item[(2b)] (Pointwise characterization) $$\lim_N \frac{1}{N^{k-1}} \sum_{h \in [N]^{k-1}} \sup_{t} \left| \frac{1}{N} \sum_{n=1}^N e^{2\pi i n t} \left[\prod_{\eta\in V_{k-1}} c^{|\eta|} f \circ T^{h \cdot \eta} \right](T^{n} x) \right|^2 = 0$$ for almost all $x$.
    \item[(3a)] (Norm characterization, multiple functions) For all collections $(g_\eta)_{\eta\in V_{k-1}}$ with $g_\eta\in L^\infty(\mu)$ and some $g_\eta = f$, we have $$\lim_N \frac{1}{N^{k-1}} \sum_{h \in [N]^{k-1}}\left \Vert \sup_{t} \left| \frac{1}{N} \sum_{n=1}^N e^{2\pi i n t} \left[\prod_{\eta\in V_{k-1}} c^{|\eta|} g_\eta \circ T^{h \cdot \eta} \right]\circ T^{n} \right| \right\Vert_2^2 = 0$$
    \item[(3b)] (Norm characterization) $$\lim_N \frac{1}{N^{k-1}} \sum_{h \in [N]^{k-1}}\left \Vert \sup_{t} \left| \frac{1}{N} \sum_{n=1}^N e^{2\pi i n t} \left[\prod_{\eta\in V_{k-1}} c^{|\eta|} f \circ T^{h \cdot \eta} \right]\circ T^{n} \right| \right\Vert_2^2 = 0$$
    \item[(4a)] (Wiener-Wintner characterization, multiple functions) For all collections $(g_\eta)_{\eta\in V_{k-1}}$ with $g_\eta\in L^\infty(\mu)$ and some $g_\eta = f$, we have $$\lim_N \frac{1}{\lfloor \sqrt{N} \rfloor^{k-1}} \sum_{h \in [\lfloor \sqrt{N} \rfloor]^{k-1}}\left \Vert \sup_{t} \left| \frac{1}{N} \sum_{n=1}^N e^{2\pi i n t} \left[\prod_{\eta\in V_{k-1}} c^{|\eta|} g_\eta \circ T^{h \cdot \eta} \right] \circ T^{n} \right| \right\Vert_2^{2/3} = 0 $$
    \item[(4b)] (Wiener-Wintner characterization) $$\lim_N \frac{1}{\lfloor \sqrt{N} \rfloor^{k-1}} \sum_{h \in [\lfloor \sqrt{N} \rfloor]^{k-1}}\left \Vert \sup_{t} \left| \frac{1}{N} \sum_{n=1}^N e^{2\pi i n t} \left[\prod_{\eta\in V_{k-1}} c^{|\eta|} f \circ T^{h \cdot \eta} \right] \circ T^{n} \right| \right\Vert_2^{2/3} = 0 \, .$$
\end{enumerate}
\end{theorem}
We pay special attention to the equivalence between (1) and (2b). For the undefined case $k=1$, suitably interpreting $V_0$ as the empty set we recover the uniform Wiener-Wintner theorem. Hence, this equivalence may be considered as an extension of the uniform Wiener-Wintner theorem for bounded functions. 
\begin{theorem}[Uniform Wiener-Wintner theorem for higher orders]\label{t:HUWW}
Let $(X, \mathcal{F}, \mu, T)$ be an invertible ergodic dynamical system and $f\in L^\infty(\mu)$. Then $f\in L^2(\mathcal{Z}_k)^\perp$ if and only if $$\lim_N \frac{1}{N^{k-1}} \sum_{h \in [N]^{k-1}} \sup_{t} \left| \frac{1}{N} \sum_{n=1}^N e^{2\pi i n t} \left[\prod_{\eta\in V_{k-1}} c^{|\eta|} f \circ T^{h \cdot \eta} \right](T^{n} x) \right|^2 = 0$$ for $\mu$-a.e. $x\in X$.
\end{theorem}

\begin{remark}[Seminorm control of higher order WW average]\label{WW.average.bound}
We can witness from the proof of Theorem \ref{t:higher.order.UWW} that higher-order WW averages can be controlled by the seminorms: Given $k \in \N$, there exists $C > 0$ such that for every $f \in L^\infty(\mu)$, we have
$$\limsup_N \frac{1}{\lfloor \sqrt{N} \rfloor^{k-1}} \sum_{h \in [\lfloor \sqrt{N} \rfloor]^{k-1}}\left \Vert \sup_{t} \left| \frac{1}{N} \sum_{n=1}^N e^{2\pi i n t} \left[\prod_{\eta\in V_{k-1}} c^{|\eta|} f \circ T^{h \cdot \eta} \right] \circ T^{n} \right| \right\Vert_2^{2/3} \leq C \VERT f \VERT_{k+1}^{2/3} \, .$$ 
We note that the averages in the left-hand side of the estimate above appears in the right-hand side of the estimate in Theorem \ref{BBgeq3}; from this, we obtain a qualitative estimate (i.e. the estimate where we let $N \to \infty$) of the norm of the multiple ergodic averages in terms of the seminorms.
\end{remark}

\begin{proof}[Proof of Theorem \ref{t:higher.order.UWW}]
Without loss of generality, we may assume that $\Vert f \Vert_\infty \leq 1$ throughout, and the same for any functions $g_\eta$.

($1 \implies 2a$) We claim that for each $k\geq 2$, there exists a constant $C$ such that
$$\limsup_{N}\frac{1}{N^{k-1}}\sum_{h_1, \dots, h_{k-1}=1}^N \sup_{t} \left|\frac{1}{N} \sum_{n=1}^N e^{2\pi i n t} \left[\ \prod_{\eta\in V_{k-1}} c^{|\eta|} g_\eta \circ T^{h \cdot \eta}\right](T^n x) \right|^2 \leq C\min_{\eta\in V_{k-1}} \left\{\VERT g_\eta \VERT_{k+1}^2 \right\}$$
for $\Vert g_\eta \Vert_\infty \leq 1$ and almost all $x$. Although not explicitly stated, this claim is shown by Assani \cite[Lemma 3]{assani-cubes} , where the $k=2$ case is proved and the argument extends by induction. He also proves the claim that 
\begin{align}\label{two.functions}
\frac{1}{N} \sum_{n=1}^N \left|\frac{1}{N} \sum_{m=1}^N f_1(T^mx) \overline{f_2(T^{n+m} x)}\right|^2 \leq C \min_{j = 1, 2} \left\{ \sup_t \left| \frac{1}{jN} \sum_{n=1}^{jN} e^{2\pi i n t} f_j(T^n x) \right|^2 \right\}
\end{align}
for normalized functions and an absolute constant $C$, which will be needed in the induction proof (specifically, he remarks that such a bound exists for each $f_1$ and $f_2$ separately, so taking the larger constant gives a bound on the minimum) \cite[p. 248]{assani-cubes}. The above bound is also generalized in \cite[p.2, Lemma 1]{assani-cubes-2}.

Toward this end, we prove the $k+1$ case following \cite{assani-cubes}. As in that argument, the constant $C$ may change from line to line, but only ever picks up dependence on $k$. For $(H+1)^2 < N$, we notice by the Van der Corput estimate (Lemma \ref{vdc-lem}) and the Cauchy-Schwarz estimate that
\begin{align*}
&\frac{1}{N^{k}}\sum_{h_1, \dots, h_k=1}^N \sup_{t} \left|\frac{1}{N} \sum_{n=1}^N e^{2\pi i n t} \left[\ \prod_{\eta\in V_{k}} c^{|\eta|} g_\eta \circ T^{h \cdot \eta}\right](T^n x) \right|^2 \\
&\leq \frac{1}{N^{k}}\sum_{h_1, \dots, h_k=1}^N \left( \frac{C}{H} + \frac{C}{H} \sum_{m=1}^H \left|\frac{1}{N} \sum_{n=1}^N  \left[\ \prod_{\eta\in V_{k}} c^{|\eta|} [g_\eta \cdot \overline{g_\eta \circ T^m}] \circ T^{h \cdot \eta}\right](T^n x) \right| \right) \\
&\leq   \frac{C}{H} + \left(\frac{C}{H} \sum_{m=1}^H \frac{1}{N^{k}}\sum_{h_1, \dots, h_k=1}^N \left|\frac{1}{N} \sum_{n=1}^N  \left[\ \prod_{\eta\in V_{k}} c^{|\eta|}  [g_\eta \cdot \overline{g_\eta \circ T^m}] \circ T^{h \cdot \eta}\right](T^n x) \right|^2 \right)^{1/2} \, .
\end{align*}
For tuples $h, \eta$, let $h'$ and $\eta'$ be the tuples with the last component removed. If we break up the product along the $\eta_k$ components, we can apply (\ref{two.functions}) to the sum in $h_k$:
\begin{align*}
&\frac{C}{H} + \left(\frac{C}{H} \sum_{m=1}^H \frac{1}{N^{k}}\sum_{h_1, \dots, h_k=1}^N \left|\frac{1}{N} \sum_{n=1}^N  \left[\ \prod_{\eta\in V_{k}} c^{|\eta|}  [g_\eta \cdot \overline{g_\eta \circ T^m}] \circ T^{h \cdot \eta}\right](T^n x) \right|^2 \right)^{1/2}\\
&= \frac{C}{H} + \Bigg(\frac{C}{H} \sum_{m=1}^H \frac{1}{N^{k-1}}\sum_{h_1, \dots, h_{k-1=1}}^N \\ & \frac{1}{N} \sum_{h_k = 1}^N \Bigg|\frac{1}{N} \sum_{n=1}^N  \left[\ \prod_{\overset{\eta\in V_{k}} {\eta_k = 0}} c^{|\eta'|}  [g_\eta \cdot \overline{g_\eta \circ T^m}] \circ T^{h' \cdot \eta'}\right](T^n x) \overline{\left[\ \prod_{\overset{\eta\in V_{k}} {\eta_k = 1}} c^{|\eta'|}  [g_\eta \cdot \overline{g_\eta \circ T^m}] \circ T^{h' \cdot \eta'}\right]}(T^{n+h_k} x) \Bigg|^2 \Bigg)^{1/2} \\
&\leq \frac{C}{H} + \Bigg(\frac{C}{H} \sum_{m=1}^H \frac{1}{N^{k-1}}\sum_{h_1, \dots, h_{k-1=1}}^N   \\ &\min_{j=0, 1} \Bigg\{\sup_t\Bigg|\frac{1}{(j+1)N} \sum_{n=1}^{(j+1)N} e^{2\pi i n t} \left[\ \prod_{\overset{\eta\in V_{k}} {\eta_k = j}} c^{|\eta'|}  [g_\eta \cdot \overline{g_\eta \circ T^m}] \circ T^{h' \cdot \eta' }\right](T^n x) \Bigg|^2\Bigg\} \Bigg)^{1/2} \\
&\leq \min_{j=0, 1} \Bigg\{\frac{C}{H} + \Bigg(\frac{C}{H} \sum_{m=1}^H \frac{1}{((j+1)N)^{k-1}}\sum_{h_1, \dots, h_{k-1=1}}^{(j+1)N} \\ & \sup_t\Bigg|\frac{1}{(j+1)N} \sum_{n=1}^{(j+1)N}  e^{2\pi i n t}\left[\ \prod_{\overset{\eta\in V_{k}} {\eta_k = j}} c^{|\eta'|}  [g_\eta \cdot \overline{g_\eta \circ T^m}] \circ T^{h' \cdot \eta'}\right](T^n x) \Bigg|^2 \Bigg)^{1/2} \Bigg\} 
\end{align*}
where the sum in the last line is extended by adding positive terms, and the extra $j+1$ terms are absorbed in the constant. If we take the $\limsup$ in $N$ of both sides, the $\limsup$ on the right passes through the minimum and the $m$ sum. Moreover, the $j+1$ scaling is lost, as we can bound the $\limsup$ along the subsequence $(j+1)N$ by the $\limsup$ along the entire sequence:
\begin{align*}
& \limsup_N\frac{1}{N^{k}}\sum_{h_1, \dots, h_k=1}^N \sup_{t} \left|\frac{1}{N} \sum_{n=1}^N e^{2\pi i n t} \left[\ \prod_{\eta\in V_{k}} c^{|\eta|} g_\eta \circ T^{h \cdot \eta}\right](T^n x) \right|^2 \\
&\leq \min_{j=0, 1} \left\{\frac{C}{H} + \left(\frac{C}{H} \sum_{m=1}^H \limsup_N \frac{1}{N^{k-1}}\sum_{h_1, \dots, h_{k-1=1}}^N \sup_t\left|\frac{1}{N} \sum_{n=1}^N e^{2\pi i n t} \left[\ \prod_{\overset{\eta\in V_{k}} {\eta_k = j}} c^{|\eta'|}  [g_\eta \cdot \overline{g_\eta \circ T^m}] \circ T^{h' \cdot \eta'}\right](T^n x) \right|^2 \right)^{1/2} \right\} \, .
\end{align*}
Since we have lost dependence on $\eta_k$ in the product over $V_k$, we are really taking a product over $V_{k-1}$. By the inductive hypothesis, for every $m\in \N$ and $j = 0, 1$, there exists a set of full measure where on which 
\begin{align*}
&\limsup_N\frac{1}{N^{k-1}}\sum_{h_1, \dots, h_{k-1}=1}^N  \sup_t\left|\frac{1}{N} \sum_{n=1}^N e^{2\pi i n t} \left[\ \prod_{\overset{\eta\in V_{k}}{\eta_k = j}} c^{|\eta'|} (g_\eta \cdot \overline{g_\eta \circ T^m})\circ T^{h'\cdot \eta'}\right](T^n x)\right|^2 \\ 
&\leq C\min_{\overset{\eta\in V_k} {\eta_k = j}} \{\VERT g_\eta \cdot \overline{g_\eta\circ T^m } \VERT_{k+1}^2 \}
\end{align*}
holds. Since there are countably many pairs $\{(m, j) : m \in \N \text{ and } j = 0, 1 \}$, we can find a set of full measure where the above holds for all simultaneously. Hence, for any $x$ in that set we have
\begin{align*}
&\limsup_N\frac{1}{N^{k}}\sum_{h_1, \dots, h_k=1}^N \sup_{t} \left|\frac{1}{N} \sum_{n=1}^N e^{2\pi i n t} \left[\ \prod_{\eta\in V_{k}} c^{|\eta|} g_\eta \circ T^{h \cdot \eta}\right](T^n x) \right|^2 \\
&\leq \min_{j=0, 1} \left\{\frac{C}{H} + \left(\frac{C}{H} \sum_{m=1}^H \min_{\{\eta\in V_k: \eta_k = j\}} \{\VERT g_\eta \cdot \overline{g_\eta\circ T^m } \VERT_{k+1}^2 \} \right)^{1/2} \right\} \\
&\leq \min_{\eta \in V_k} \left\{\frac{C}{H} + \left(\frac{C}{H} \sum_{m=1}^H  \VERT g_\eta \cdot \overline{g_\eta\circ T^m } \VERT_{k+1}^2  \right)^{1/2} \right\} \\
&\leq \min_{\eta \in V_k} \left\{\frac{C}{H} + \left(\frac{C}{H} \sum_{m=1}^H  \VERT g_\eta \cdot \overline{g_\eta\circ T^m } \VERT_{k+1}^{2^{k+1}}  \right)^{1/2^{k+1}} \right\} 
\end{align*}

Taking the limit in $H$, it passes in the minimum on the right-hand side we get $\min_{\eta \in V_k} \{C\VERT g_\eta \VERT_{k+2}^{2^{k+2}})^{1/2^{k+1}} \} = C\min_{\eta\in V_k} \{\VERT g_\eta \VERT_{k+2}^2\}$, which establishes the claim.

($2a \implies 2b$, $3a \implies 3b$, $4a \implies 4b$) Let $g_\eta = f$ for all $\eta$.

($2a \implies 3a$, $2b \implies 3b$) Since everything is uniformly bounded, this follows from the dominated convergence theorem.

($3a \implies 4a$, $3b \implies 4b$) Notice for an arbitrary function $g$ (taken to have $\Vert g \Vert_\infty \leq 1$) we have
\begin{align*}
\left\Vert \sup_t \left| \frac{1}{N} \sum_{n=1}^N e^{2\pi i n t} g\circ T^n  \right| \right\Vert_2 &\leq \left\Vert \sup_t \left| \frac{1}{N} \sum_{n=1}^{\lfloor \sqrt{N} \rfloor^2} e^{2\pi i n t} g \circ T^n  \right| \right\Vert_2 + \left\Vert \sup_t \left| \frac{1}{N} \sum_{n=\lfloor \sqrt{N} \rfloor^2+1}^{N} e^{2\pi i n t} g \circ T^n  \right| \right\Vert_2 \\
&\leq \left\Vert \sup_t \left| \frac{1}{\lfloor \sqrt{N} \rfloor^2} \sum_{n=1}^{\lfloor \sqrt{N} \rfloor^2} e^{2\pi i n t} g \circ T^n  \right| \right\Vert_2 + \frac{2}{\sqrt{N}} \\
&= \left\Vert \sup_t \left| \frac{1}{\lfloor \sqrt{N} \rfloor^2} \sum_{j=0}^{\lfloor \sqrt{N} \rfloor-1}\sum_{n=1}^{\lfloor \sqrt{N} \rfloor} e^{2\pi i (n + j\lfloor \sqrt{N} \rfloor) t} g \circ T^{n + j\lfloor \sqrt{N} \rfloor}  \right| \right\Vert_2 + \frac{2}{\sqrt{N}} \\
&\leq \frac{1}{\lfloor \sqrt{N} \rfloor}\sum_{j=0}^{\lfloor \sqrt{N} \rfloor-1}\left\Vert \sup_t \left| \frac{1}{\lfloor \sqrt{N} \rfloor} \sum_{n=1}^{\lfloor \sqrt{N} \rfloor} e^{2\pi i (n + j\lfloor \sqrt{N} \rfloor) t} g \circ T^{n + j\lfloor \sqrt{N} \rfloor}  \right| \right\Vert_2 + \frac{2}{\sqrt{N}} \\
&= \frac{1}{\lfloor \sqrt{N} \rfloor}\sum_{j=0}^{\lfloor \sqrt{N} \rfloor-1}\left\Vert \sup_t \left| \frac{1}{\lfloor \sqrt{N} \rfloor} \sum_{n=1}^{\lfloor \sqrt{N} \rfloor} e^{2\pi i n t} g \circ T^{n}  \right| \right\Vert_2 + \frac{2}{\sqrt{N}} \\
&= \left\Vert \sup_t \left| \frac{1}{\lfloor \sqrt{N} \rfloor} \sum_{n=1}^{\lfloor \sqrt{N} \rfloor} e^{2\pi i n t} g \circ T^{n}  \right| \right\Vert_2 + \frac{2}{\sqrt{N}} \, ,
\end{align*}
where the $j$ dependence is lost in the absolute value and measure-preserving norm. Hence,
\begin{align*}
&\frac{1}{\lfloor \sqrt{N} \rfloor^{k-1}} \sum_{h \in [\lfloor \sqrt{N} \rfloor]^{k-1}}\left \Vert \sup_{t} \left| \frac{1}{N} \sum_{n=1}^N e^{2\pi i n t} \left[\prod_{\eta\in V_{k-1}} c^{|\eta|} g_\eta \circ T^{h \cdot \eta} \right] \circ T^{n} \right| \right\Vert_2^{2/3} \\
&\leq \frac{1}{\lfloor \sqrt{N} \rfloor^{k-1}} \sum_{h \in [\lfloor \sqrt{N} \rfloor]^{k-1}} \left( \left \Vert \sup_{t} \left| \frac{1}{\lfloor \sqrt{N} \rfloor} \sum_{n=1}^{\lfloor \sqrt{N} \rfloor} e^{2\pi i n t} \left[\prod_{\eta\in V_{k-1}} c^{|\eta|} g_\eta \circ T^{h \cdot \eta} \right] \circ T^{n} \right| \right\Vert_2 + \frac{2}{\sqrt{N}}\right)^{2/3} \\
&\leq \frac{1}{\lfloor \sqrt{N} \rfloor^{k-1}} \sum_{h \in [\lfloor \sqrt{N} \rfloor]^{k-1}} \left \Vert \sup_{t} \left| \frac{1}{\lfloor \sqrt{N} \rfloor} \sum_{n=1}^{\lfloor \sqrt{N} \rfloor} e^{2\pi i n t} \left[\prod_{\eta\in V_{k-1}} c^{|\eta|} g_\eta \circ T^{h \cdot \eta} \right] \circ T^{n} \right| \right\Vert_2 ^{2/3} + \left(\frac{2}{\sqrt{N}}\right)^{2/3} \\
&\leq \left(\frac{1}{\lfloor \sqrt{N} \rfloor^{k-1}} \sum_{h \in [\lfloor \sqrt{N} \rfloor]^{k-1}} \left \Vert \sup_{t} \left| \frac{1}{\lfloor \sqrt{N} \rfloor} \sum_{n=1}^{\lfloor \sqrt{N} \rfloor} e^{2\pi i n t} \left[\prod_{\eta\in V_{k-1}} c^{|\eta|} g_\eta \circ T^{h \cdot \eta} \right] \circ T^{n} \right| \right\Vert_2 ^{2}\right)^{1/3} + \left(\frac{2}{\sqrt{N}}\right)^{2/3} \, ,
\end{align*}
which convergence to zero by assumption. 

($4b \implies 1$) Let $f_1 ,\dots, f_{k+1}$ be uniformly bounded by 1, and have some $f_j = f$. By the Bourgain bound on multiple recurrence (Theorem \ref{BBgeq3}, with 2-norm, as previously remarked), we have
\begin{align*}
&\left \Vert \frac{1}{N} \sum_{n=1}^N \prod_{j=1}^{k+1} f_j \circ T^{jn} \right\Vert_2 \leq \\
 &C_{k+1}\left( \frac{1}{\lfloor \sqrt{N}\rfloor^{1/2^{k-1}}} + \left(\frac{1}{\lfloor\sqrt{N}\rfloor^{k-1}} \sum_{h \in [\lfloor\sqrt{N}\rfloor]^{k-1}}\left\Vert \sup_t \left| \frac{1}{N} \sum_{n=1}^N e^{2\pi i n t} \left[ \prod_{\eta \in V_{k-1}}  c^{|\eta|} f_1 \circ T^{h \cdot \eta}\right] \circ T^nm\right|\right\Vert_2^{2/3} \right)^{1/2^{k-1}} \right)
\end{align*}
for sufficiently large $N$. By assumption, the limit on the right-hand side is zero. Hence, 
$$\lim_N \left \Vert \frac{1}{N} \sum_{n=1}^N \prod_{j=1}^{k+1} f_j\circ T^{jn} \right\Vert_2 = 0$$
for all $f_1, \dots, f_{k+1}$ with some $f_j = f$. So $f\in L^2(\mathcal{Z}_k)^\perp$.
\end{proof}

\section{WW Stability under products of K systems}\label{s:K-stability}

In this section, we will show that a product of a K system (which we know is a $j$-th order WW system for any $j \in \N$) and a $k$-th order WW system is also a $k$-th order WW system. More specifically, we will prove the following:
\begin{theorem}\label{t:productWWS}
Let $(X, \mathcal{F}, \mu, T)$ be a K system and $(Y, \mathcal{G}, \nu, S)$ be a $k$-th order WW system of power type $\alpha$ in $L^p$ for $p\in [1, \infty]$. Then $(X \times Y, \mathcal{F} \otimes \mathcal{G}, \mu \times \nu, T \times S)$ is a $k$-th order WW system of power type $\min\{1/6, \alpha\}$ in $L^{\min\{2,p\}}(\mu \times \nu)$.
\end{theorem}

\begin{remark}
    For instance, if $(Y, \mathcal{G}, \nu, S)$ from the theorem is a $k$-dimensional classical skew product, which we have shown that is a $k-1$-th order WW system from Theorem \ref{t:csp}, then the product system in the theorem is a $k-1$-th order WW system that is neither weakly mixing (since the classical skew product is not weakly mixing) nor distal (since the product system has a positive entropy). Hence, this is an example of a WW system where the pointwise convergence of multiple recurrence has not been studied previously.
\end{remark}

In order to prove Theorem \ref{t:productWWS}, we first need the following lemma:
\begin{lemma}\label{l:an}
Let $(X, \mathcal{F}, \mu, T)$ is a measure-preserving system, and let $\mathcal{A}$ be a subsigma-algebra of $\mathcal{F}$ that satisfies (\ref{Pinsker}). Let $J\in \N$ and for $j = 1, \dots, J$ let $k_j, l_j \in \Z$ and $A_j\in T^{-l_j}\A$, and let $f_{k_j}^{A_j} = \1_{A_j} - \E(\1_{A_j}|T^{-k_j}\mathcal{A})$. Then there exists $C>0$ and $L\in \N$ such that for all bounded complex sequences $(a_n)_n$ and for all $p_1, \dots, p_J \in \N$ with $|p_\alpha - p_\beta| > L$ for all $\alpha \neq \beta$, the following bound holds for all $N \in \N$:
\begin{align}
\left\Vert \sup_t \left| \frac{1}{N}\sum_{n=1}^N e^{2\pi i n t} a_n \left[ \prod_{j=1}^J f_{k_j}^{A_j} \circ T^{p_j} \right] \circ T^n \right| \right\Vert_2^2 \leq \frac{C\Vert a_n \Vert_{\ell^\infty}^2}{N^{1/2}} \, .
\end{align}
\end{lemma}

\begin{remark}
The condition that the $p_j$'s be sufficiently far apart can be weakened to $|p_\alpha - \min_j p_j| > L$ for all $\alpha$ that do not minimize $p$.
\end{remark}

\begin{proof}

Let $(a_n)_n$ and $p_1, \dots, p_J$ be as in the statement. Assume that $\Vert a_n \Vert_{\ell^\infty} = 1$ and recall that we may assume $l_j < k_j$, or else the entire term is zero. Pick
$$L = \max_{a, b \in [J], a \neq b} \{|l_a - l_b|, |k_a - l_b|, |k_a - k_b|\}\, .$$
Denote $\alpha$ as the index such that $\min_j p_j = p_\alpha $. For each $j$, define $q_j = p_j - p_\alpha$. Notice that for all $j\neq \alpha$, we have $q_j > L$. By the $T$-invariance of $\mu$, we note that
$$\left\Vert \sup_t \left| \frac{1}{N}\sum_{n=1}^N e^{2\pi i n t} a_n \left[ \prod_{j=1}^J f_{k_j}^{A_j} \circ T^{p_j} \right] \circ T^n \right| \right\Vert_2^2 = \left\Vert \sup_t \left| \frac{1}{N}\sum_{n=1}^N e^{2\pi i n t} a_n \left[ \prod_{j=1}^J f_{k_j}^{A_j} \circ T^{q_j} \right] \circ T^n \right| \right\Vert_2^2 \, ,$$
so it suffices to consider the latter term, which has $q_{\alpha} = 0$.

Define $F = \prod_{j=1}^J f_{k_j}^{A_j} \circ T^{q_j}$, and note that $\Vert F \Vert_\infty \leq 2^{J}$. We compute pointwise by the Van der Corput inequality (\ref{vdc-2}) that for almost all $x$, we have
\begin{align*}
\sup_{t} \left|\sum_{n=1}^N e^{2\pi i n t} a_n F(T^{n} x) \right|^2 &\leq 2^{2J+1}N + 4\sum_{m = 1}^{N-1}\left| \sum_{n = 1}^{N-m} a_n F(T^{n} x) \overline{a_{n+m}}F(T^{n + m} x)\right| \\
&\leq 2^{2J+1}N + 2^{2J+2}(k_\alpha - l_\alpha)N + 4\sum_{m = k_\alpha - l_\alpha+1}^{N-1}\left| \sum_{n = 1}^{N-m} a_n \overline{a_{n+m}}F(T^{n} x)F(T^{n + m} x)\right| \, .
\end{align*}
By integrating both sides and bounding the $L^1(\mu)$ norm by the $L^2(\mu)$ norm, we have
.
\begin{equation} \label{lemma-k-est}
\begin{split}
&\int \sup_{t} \left| \sum_{n=1}^N e^{2\pi i n t} a_n F(T^{n} x) \right|^2 \, d\mu(x) \\ &\leq 2^{2J+1}N + 2^{2J+2}(k_\alpha - l_\alpha)N + 4\sum_{m = k_\alpha - l_\alpha+1}^{N-1} \left(\int\left| \sum_{n = 1}^{N-m} a_n \overline{a_{n+m}}F(T^{n} x)F(T^{n + m} x)\right|^2 \, d\mu(x)\right)^{1/2} \, .
\end{split}
\end{equation}

Consider the inner $n$ sum, which is squared. If we factor out this product, the diagonal terms can be bounded away by $2^{4J}N$, and we are left with off-diagonal terms
\begin{align*}
&\sum_{n \neq n'}^{N-m} \int  a_{n} \overline{a_{n+m}} \overline{a_{n'}} a_{n' + m}F(T^{n} x)F(T^{n + m} x)F(T^{n'} x)F(T^{n' + m} x) \, d\mu(x) \\
&= 2\sum_{n < n'}^{N-m} \real(a_n \overline{a_{n+m}} \overline{a_{n'}} a_{n'+m})\int \prod_{j=1}^J f^{A_j}_{k_j} \circ T^{n + q_{j}} \cdot f^{A_j}_{k_j} \circ T^{n + m + q_j} \cdot f^{A_j}_{k_j} \circ T^{n' + q_j} \cdot f^{A_j}_{k_j} \circ T^{n' +m + q_j}\, d\mu \, .
\end{align*}
Recall that $m > k_\alpha - l_\alpha$. We claim that if $n' > n + k_\alpha - l_\alpha$, then 
\begin{align} \label{lemma-int}
\int \prod_{j=1}^J f^{A_j}_{k_j} \circ T^{n + q_{j}} \cdot f^{A_j}_{k_j} \circ T^{n + m + q_j} \cdot f^{A_j}_{k_j} \circ T^{n' + q_j} \cdot f^{A_j}_{k_j} \circ T^{n' +m + q_j}\, d\mu = 0 \, .
\end{align}

To this end, we consider the case $n' > n + m$. Notice that in the product, the above four functions in the product are measurable to the degrees $T^{-b}\A$ for the following values of $b$:
$$n + q_j + l_j \quad n + m + q_j + l_j \quad n' + q_j + l_j \quad n' + m + q_j + l_j \, .$$
Hence, by the condition $n' > n+m$, and the fact that each $q_j \geq l_\alpha - l_j$, we see that
$$\prod_{j=1}^J f^{A_j}_{k_j} \circ T^{n + m + q_j} \cdot f^{A_j}_{k_j} \circ T^{n' + q_j} \cdot f^{A_j}_{k_j} \circ T^{n'+m + q_j} \quad \text{is } T^{-(n + m + l_\alpha)}\A \text{ measurable.}$$
So for the integral (\ref{lemma-int}), we would get the same value if we conditioned the integrand on $T^{-(n + m + l_\alpha)}\A$, and the above product of functions can factor out of this conditional expectation. So the claim that (\ref{lemma-int}) is zero reduces to showing 
$$\E\left( \prod_{j=1}^J f^{A_j}_{k_j} \circ T^{n+ q_j} \Bigg| T^{-(n + m + l_\alpha)}\A\right) = 0 \, .$$

Recall that $f_{k_j}^{A_j} = \1_{A_j} - \E(\1_{A_j} |T^{-k_j})$. As we take the product of $J$ such functions, we wish to expand all such $f$'s except for the $j = \alpha$ term. For $\theta\in V_{J}$, let
$$G_\theta = \prod_{\overset{j\in [J]-\alpha}{ \theta_j = 0}} \1_{A_{j}}\circ T^{q_j} \prod_{\overset{j\in [J]-\alpha} {\theta_j = 1}} (-1)\E(\1_{A_j} | T^{-k_j}\A) \circ T^{q_j} \, . $$
Notice that each of the functions in the left product are $T^{-(q_j+ l_j)}\A$ measurable and each of the functions in the right product are $T^{-(q_j+ k_j)}\A$ measurable.

Expanding the product we wish to condition, we see
$$\prod_{j=1}^J f^{A_j}_{k_j} \circ T^{n+ q_j} = \sum_{\overset{\theta\in V_{J}}{\theta_\alpha = 0}} \1_{A_{\alpha}} \circ T^n \cdot G_\theta \circ T^n - \E(\1_{A_{\alpha}} | T^{-k_\alpha}) \circ T^n \cdot G_\theta \circ T^n$$
where restricting to $\theta_\alpha = 0$ prevents double-counting. We claim for each $\theta$ that the above two summands cancel under the $T^{-(n + m + l_\alpha)}\A$ conditional. Looking at the left summand, we have
$$\E(\1_{A_\alpha} \circ T^n \cdot G_\theta \circ T^n | T^{-(n + m + l_\alpha)}\A) = \E(\1_{A_\alpha} \cdot G_\theta | T^{-(m + l_\alpha)}\A) \circ T^n \, ,$$
On the right, the term $G_\theta$ may pass into the $T^{-k_\alpha}$ conditional, as we have $q_j > k_\alpha - l_j$ and $q_j > k_\alpha - k_j$ for all $j \neq \alpha$:
\begin{align*}
\E(\E(\1_{A_\alpha}|T^{-k_\alpha}\A) \circ T^n \cdot G_\theta \circ T^n | T^{-(n + m + l_\alpha)}\A) &= \E(\E(\1_{A_\alpha}|T^{-k_\alpha}\A) \cdot G_\theta | T^{-( m + l_\alpha)}\A) \circ T^n \\
&= \E(\E(\1_{A_\alpha} \cdot G_\theta |T^{-k_\alpha}\A)  | T^{-( m + l_\alpha)}\A) \circ T^n \\
&= \E(\1_{A_\alpha} \cdot G_\theta  | T^{-( m + l_\alpha)}\A) \circ T^n
\end{align*}
as $m +l_\alpha > k_\alpha$. So these terms cancel, and the integral (\ref{lemma-int}) does indeed vanish under the conditioning.

For the case $n + m \geq n' > n + k_\alpha - l_\alpha$, we similarly observe that
$$\prod_{j=1}^J f^{A_j}_{k_j} \circ T^{n + m + q_j} \cdot f^{A_j}_{k_j} \circ T^{n' + q_j} \cdot f^{A_j}_{k_j} \circ T^{n'+m + q_j} \quad \text{is } T^{-(n' + l_\alpha)}\A \text{ measurable.}$$
and by the same argument, conditioning (\ref{lemma-int}) under $T^{-(n' + l_\alpha)}\A$ shows that it is zero.

Returning to our bound (\ref{lemma-k-est}), we have observed that the off-diagonal terms vanish when $n'$ is larger than $n$ by $k_\alpha - l_\alpha$. Hence, for each $n = 1$ to $N - m$, at most $k_\alpha - l_\alpha$ terms are nonzero, and the total number of nonzero terms in the double sum over $n$ and $n'$ can be bounded by $N(k_\alpha -l_\alpha)$. Plugging this into our estimate (\ref{lemma-k-est}), we see
\begin{align*}
&\int \sup_{t} \left| \sum_{n=1}^N e^{2\pi i n t} a_n F(T^{n} x) \right|^2 \, d\mu(x) \\
&\leq 2^{2J+1}N + 2^{2J+2}(k_\alpha - l_\alpha)N + 4\sum_{m = k_\alpha - l_\alpha+1}^{N-1} \left(2^{4J} N + 2^{4J+1}N(k_\alpha-l_\alpha) \right)^{1/2}\, .
\end{align*}
Hence, we have the bound
\begin{align*}
\int \sup_{t} \left| \frac{1}{N} \sum_{n=1}^N e^{2\pi i n t} a_n \left[ \prod_{j=1}^J f^{A_j}_{k_\eta} \circ T^{q_j}\right](T^nx) \right|^2 \, d\mu(x) &\leq \frac{C}{N^{1/2}} \, .
\end{align*}
Accounting for $\Vert a_n \Vert_{\ell^\infty}$, we get the desired bound.
\end{proof}

\begin{theorem} \label{th:prodfunc}
Let $(X, \mathcal{F}, \mu, T)$ be a K system and $(Y, \mathcal{G}, \nu, S)$ be a $k$-th order WW dynamical system of power type $\alpha$ in $L^p$ for $p\in [1, \infty]$. Let $F\in L^\infty(\mu \times \nu)$ be of the form $$F(x, y) = g(y) + \sum_{i=1}^I \alpha_i f_{k_i}^{A_i}(x) g_i(y) \, ,$$ where $g\in L^\infty(\nu)$ is a $k$-th order WW function in $Y$ of power type $\alpha$ in $L^p$, $g_i(y)\in L^\infty(\nu)$, and $f_{k_j}^{A_j} \in L^\infty(\mu)$ are functions in $\mathcal{E}$ from (\ref{setE}). Then $F$ is a $k$-th order WW function of power type $\min\{\alpha, 1/6\}$ in $L^{\min\{p, 2\}}$.
\end{theorem}

\begin{proof}
By the multilinearity concerns, we need to establish polynomial decay on WW averages for all possible diagonal and off-diagonal terms. So let $(e_\eta)_{\eta \in V_{k-1}}$ be a sequence of functions in $L^\infty(\mu \otimes \nu)$ for which $$e_\eta \in \{g\} \cup \{f_{k}^{A} \otimes \gamma: k, l\in \Z, A\in T^{-l}\mathcal{A}, \gamma \in L^\infty(\nu)\}$$
which absorbs the constants $\alpha$ into $g$ (recall that the subsigma-algebra $\mathcal{A} \subset \mathcal{F}$ is from (\ref{Pinsker}); the Pinsker algebra $\mathcal{P}$ is trivial in our case). We are interested in the terms
$$W_N^{p} ((e_\eta)_\eta) :=  \frac{1}{\lfloor\sqrt{N}\rfloor^{k-1}}\sum_{h \in [\lfloor \sqrt{N}\rfloor ]^{k-1} } \left\Vert \sup_t \left|\frac{1}{N} \sum_{n=1}^N e^{2\pi i n t} \left[ \prod_{\eta \in V_{k-1}} c^{|\eta|} e_\eta \circ (T\times S)^{h\cdot \eta} \right] \circ (T\times S)^n \right|\right\Vert_{L^p(\mu\times \nu)}^{2/3}$$
In the case that all $e_\eta = g$, for all $N \in \N$, there exists $C_g > 0$ such that 
$$W_N^{\min\{p, 2\}}((e_\eta)_\eta) \leq W_N^{p}((e_\eta)_\eta) \leq \frac{C_g}{N^{\alpha}} \leq \frac{C_g}{N^{\min\{\alpha, 1/6\}}}\, ,$$
since $g$ is a WW function and the $L^p(\mu\times \nu)$ and $L^p(\nu)$ norms coincide because of the lack of $x$ dependence. 

For $(e_\eta)_\eta$ where not all $e_\eta = g$, we partition $V_{k-1}$ as follows: Define $U\subset V_{k-1}$ as the collection of $\eta$ such that $e_\eta \neq g$. Hence, $U$ and $U^c$ is a disjoint partition of $V_{k-1}$ where $U$ is nonempty, and for $\eta \in U$ we can write $e_\eta = f_{k_\eta}^{A_\eta}(x)g_\eta(y)$ as expected. Let $J = \# U$. Applying Lemma \ref{l:an} for indexing set $[J] = U$, $k_\eta, l_\eta, A_\eta$, and $p_\eta = \eta\cdot h$ for $h\in \N^{k-1}$, we get an $L\in \N$ and $C > 0$ such that
\begin{equation}\label{hNotInH}
    \left\Vert \sup_t \left| \frac{1}{N}\sum_{n=1}^N e^{2\pi i n t} a_n \left[ \prod_{\eta\in U} f_{k_\eta}^{A_\eta} \circ T^{\eta \cdot h} \right] \circ T^n \right| \right\Vert_2^2 \leq \frac{C\Vert a_n \Vert_{\ell^\infty}^2}{N^{1/2}}
\end{equation}
holds for all bounded sequences so long as $|(\eta_1 - \eta_2)\cdot h| > L$ for all $\eta_1\neq \eta_2 \in U$. Letting $H$ be the set of $h$ that fail to satisfy this condition, we note for any $N$ that
\begin{equation}\label{cardBound}
    \#(H\cap [N]^{k-1}) \leq {\# U \choose 2} (2L+1) N^{k-2}
\end{equation}
as for each pair of distinct $\eta_1, \eta_2 \in U$, the term $(\eta_1 - \eta_2)\cdot h$ depends nontrivially on some component $h_i$, and fixing all other components there are at most $2L+1$ values of $h_i$ that would satisfy $|(\eta_1 - \eta_2)\cdot h| \leq L$.

Let $h \in H^c$. Then (\ref{hNotInH}) holds for any bounded sequence $a_n$ and $N \in \N$. If $y\in Y$, we may set
$$a_n = \prod_{\eta \in U^c} c^{|\eta|}g(S^{n + \eta \cdot h} y) \prod_{\eta\in U} c^{|\eta|} g_\eta(S^{n + \eta \cdot h}y) \, $$
which is uniformly bounded for almost all $y$ by $\prod_{\eta \in U^c } \Vert g \Vert_\infty^2 \prod_{\eta \in U} \Vert g_\eta \Vert_\infty^2$ to see
$$\left\Vert \sup_t \left| \frac{1}{N}\sum_{n=1}^N e^{2\pi i n t} \prod_{\eta \in U^c} c^{|\eta|}g(S^{n+ \eta\cdot h}y) \prod_{\eta\in U} f_{k_\eta}^{A_\eta} \circ T^{n + \eta\cdot h} \cdot c^{|\eta|}g_\eta(S^{n + \eta \cdot h}y) \right| \right\Vert_{L^2(\mu)}^2 \leq \frac{C}{N^{1/2}}$$
for an increased $C$. Integrating in $\nu$, we recover the $L^2(\mu \times \nu)$ norm on the left-hand side. But by the partition in $U$, we have recovered the product over $\eta$ of $e_\eta$'s that we were looking for:
\begin{equation}\label{hNotInH2}
    \left\Vert \sup_t \left| \frac{1}{N}\sum_{n=1}^N e^{2\pi i n t} \left[\prod_{\eta \in V_{k-1}} c^{|\eta|}e_\eta \circ (T \times S)^{\eta\cdot h}) \right] \circ (T\times S)^n   \right| \right\Vert_{L^2(\mu \times \nu)}^2 \leq \frac{C}{N^{1/2}} \, .
\end{equation}
Recall that (\ref{hNotInH2}) holds for all $N$ and $h\notin H$. For $h \in H$, we bound the average above trivially by $C' = \prod_{\eta\in V_{k-1}} \Vert e_\eta \Vert_\infty^2$, and their contribution may be lost in the average over $h$: By applying the triangle inequality as well as the estimates (\ref{cardBound}) and (\ref{hNotInH2}), we get
\begin{align*}
&\frac{1}{\lfloor\sqrt{N}\rfloor^{k-1}}\sum_{h \in [\lfloor \sqrt{N}\rfloor ]^{k-1} } \left\Vert \sup_t \left|\frac{1}{N} \sum_{n=1}^N e^{2\pi i n t} \left[ \prod_{\eta \in V_{k-1}} c^{|\eta|} e_\eta \circ (T\times S)^{h\cdot \eta} \right] \circ (T\times S)^n \right|\right\Vert_{L^2(\mu\times \nu)}^{2} \\
&\leq \frac{\#(H \cap [\lfloor \sqrt{N}\rfloor]^{k-1})C'}{\lfloor \sqrt{N}\rfloor^{k-1}} + \frac{C}{N^{1/2}} \leq \frac{C''}{N^{1/2}}
\end{align*}
for an increased $C''$. By Hölder's inequality on averages, we may reintroduce the 2/3 power for WW inequalities by changing $C''N^{-1/2}$ to $C''^{1/3}N^{{-1/6}}$.
Adjusting to the $L^{\min\{p, 2\}}(\mu \times \nu)$ norm and power type $\min\{\alpha, 1/6\}$ exactly as in the previous case, we see that we have satisfied the multilinearity conditions to show that $g(y) + \sum_{i=1}^I \alpha_i f_{k_i}^{A_i}(x) g_i(y)$ is a WW function. 
\end{proof}

\begin{proof}[Proof of Theorem \ref{t:productWWS}]
In this proof, we denote the trivial $\sigma$-subalgebra $\{\emptyset, X\}$ as just $X$. 

For the $\sigma$-subalgebra $X\otimes \mathcal{G} \subset \mathcal{F} \otimes \mathcal{G}$, the subspace $L^2(X\otimes \mathcal{G}, \mu\times \nu) \subset L^2(\mu \times \nu)$ is naturally isomorphic to $L^2(\nu)$, as these functions have no $x$ dependence. Treating this isomorphism as an equivalence, we write for $F\in L^2(\mu\times \nu)$ that
$$\E(F|X \otimes \mathcal{G})(y) = \int F(s, y) \, d\mu(s) \in L^2(\nu)$$
and
$$L^2(X\otimes\mathcal{G})^\perp = \left\{F: \int F(s, y)\, d\mu(s) = 0 \text{ for } \nu\text{-a.e. }y\right\}\, .$$
By the same analysis on the $\sigma$-subalgebra $X \otimes \mathcal{Z}_k^Y$, where $\mathcal{Z}_k^Y \subset \mathcal{G}$ is the $k$-th Host-Kra-Ziegler factor of $Y$, the subspace $L^2(X \otimes \mathcal{Z}_k^Y, \mu \times \nu) \subset L^2(\mu \times \nu)$ is naturally equivalent to $L^2(\mathcal{Z}_k^Y, \nu)$. For $F$ as before, we write
$$\E(F|X \otimes \mathcal{Z}_k^Y)(y) = \E\left( \int f(s, y) \, d\mu(s) \bigg| \mathcal{Z}_k^Y \right)\in L^2(\nu)$$
and
$$L^2(X\otimes \mathcal{Z}_k^Y, \mu \times \nu)^\perp = \left\{ F: \int F(s, y)\, d\mu(s) \in L^2(\mathcal{Z}_k^Y )^\perp \right\}\, .$$
Since for any $F\in L^2(X\otimes \mathcal{Z}_k^Y, \mu \times \nu)^\perp$ we have the orthogonal decomposition
$$F(x, y) = \left( F(x, y) - \int F(s, y)\, d\mu(s)\right) + \int F(s, y)\, d\mu(s)\, ,$$
where the first term of the right-hand side lies in $\{f\in L^2(\mu): \int f \, d\mu = 0\} \otimes L^2(\nu)$, while the other term lies in $L^2(\mathcal{Z}_k^Y)$, it follows by the previous theorem that we have found an $L^2$-dense set of $k$-th order WW functions in $L^2(X\otimes \mathcal{Z}_k^Y, \mu \times \nu)^\perp$. 
Since the K system $X$ is mixing, it follows that $X \times Y$ is ergodic, in which its $k$-th order Host-Kra-Ziegler factor $\mathcal{Z}_k^{X\times Y}$ is well-defined. Based on the dense set of WW functions we have found, it would suffice to show that $\mathcal{Z}_k^{X\times Y} = X \otimes\mathcal{Z}_k^{Y}$.

To this end, we create a chain of equivalent statements. For $F\in L^\infty(\mu \times \nu)$, (\ref{univCF}) tells us that we have $F\in L^2(\mathcal{Z}_k^{X \times Y}, \mu \times \nu)^\perp$ if and only if the following statement is true:
$$\text{For all }F_1, \dots, F_{k+1} \in L^\infty(\mu \times \nu) \text{ with some }F_i = F\text{, we have }\lim_N \left\Vert \frac{1}{N} \sum_{n=1}^N \prod_{j=1}^{k+1} F_j \circ (T \times S)^{jn} \right\Vert_{L^2(\mu \times \nu)} = 0\, .$$
The WW criterion for characteristic factors (Theorem \ref{t:wwchar}) applied to the set of WW functions found in Theorem \ref{th:prodfunc} shows that $X \otimes \mathcal{Z}_k^Y$ is a pointwise characteristic factor for $k+1$ multiple recurrence in $X \times Y$. By the dominated convergence theorem, it follows that it is also a characteristic factor in norm convergence. Hence, the above statement is equivalent to the following statement:
\begin{align*}
&\text{For all }F_1, \dots, F_{k+1} \in L^\infty(\mu \times \nu) \text{ with some }F_i = F\text{, } \\ &\text{we have }\lim_N \left\Vert \frac{1}{N} \sum_{n=1}^N \prod_{j=1}^{k+1} \E(F_j| X \otimes \mathcal{Z}_k^{Y}) \circ (T \times S)^{jn} \right\Vert_{L^2(\mu \times \nu)} = 0\, .
\end{align*}
Since we have already noted that this conditional expectation maps into $L^2(\mathcal{Z}_k^Y, \nu)$ as a subset of $L^2(\mu \times \nu)$, the above statement has no dependence in $x$. Reinterpreting it as a statement purely in $L^2(\nu)$, it is equivalent to the following statement:
$$\text{For all }g_1, \dots, g_{k+1} \in L^\infty(\mathcal{Z}_k^Y, \nu) \text{ with some }g_i = \E(F|X \otimes \mathcal{Z}_k^Y) \text{, we have }\lim_N \left\Vert \frac{1}{N} \sum_{n=1}^N \prod_{j=1}^{k+1} g_j \circ S^{jn} \right\Vert_{L^2(\nu)} = 0\, .$$
Notice that we may extend the choice of the $g_j$'s to all of $L^\infty(\nu)$ at no cost, as decomposing each $g_j$ as $(g_j - \E(g_j|\mathcal{Z}_k^Y) ) + \E(g_j|\mathcal{Z}_k^Y)$, we expand out the multiple recurrence product and all of the pieces containing some $g_j - \E(g_j|\mathcal{Z}_k^Y)$ converge to zero anyway. Hence, the above is equivalent to 
$$\text{For all }g_1, \dots, g_{k+1} \in L^\infty(\nu) \text{ with some }g_i = \E(F|X \otimes \mathcal{Z}_k^Y) \text{, we have }\lim_N \left\Vert \frac{1}{N} \sum_{n=1}^N \prod_{j=1}^{k+1} g_j \circ S^{jn} \right\Vert_{L^2(\nu)} = 0\, ,$$
or $\E(F|X \otimes \mathcal{Z}_k^Y) \in L^2(\mathcal{Z}_k^Y)^\perp$. Since $\E(F|X \otimes \mathcal{Z}_k^Y) \in L^2(\mathcal{Z}_k^Y)$, this is only possible if $\E(F|X \otimes \mathcal{Z}_k^Y)=0$, or $F \in L^2(X \otimes \mathcal{Z}_k^Y)^\perp$.

Hence, a bounded function $F$ lies in $L^2(\mathcal{Z}_k^{X \times Y})^\perp$ if and only if it lies in $L^2(X \otimes \mathcal{Z}_k^Y)^\perp$. By density and closure, it follows that $L^2(\mathcal{Z}_k^{X \times Y})^\perp = L^2(X \otimes \mathcal{Z}_k^Y)^\perp$, in which $\mathcal{Z}_k^{X \times Y} = X \otimes \mathcal{Z}_k^Y$, proving the claim.
\end{proof}

\section{Return times theorem for multiple ergodic averages}\label{s:rtt_mea}
We recall that the non-trivial part in the proof of the return times theorem (Theorem \ref{t:RTT}) was to show the following: 

Let $(X, \mathcal{F}, \mu, T)$ be an ergodic theorem, and let $f \in L^\infty(\mu)$. If $f$ is orthogonal to $L^2(\mathcal{Z}_1)$ (i.e. the Kronecker factor), then there exists a set of full measure $X_f$ such that for every $x \in X_f$, and for any other measure-preserving system $(Y, \mathcal{G}, \nu, S)$ and a bounded function $g \in L^\infty(\nu)$, we have
\[ \lim_{N \to \infty} \frac{1}{N} \sum_{n=1}^N f(T^nx) g(S^ny) = 0 \]
for $\nu$-a.e. $y \in Y$. 

Our goal in this section is to show the analogous result for multiple recurrence return times theorem. Let $(X, \mathcal{F}, \mu, T)$ and $(Y, \mathcal{G}, \nu, S)$ be two different ergodic systems. Suppose $J, K \in \N$, and $f_1, \ldots, f_J \in L^\infty(\mu)$ and $g_1, \ldots, g_K \in L^\infty(\nu)$. We consider the averages
\[ \frac{1}{N} \sum_{n=1}^N \prod_{j=1}^J f_j (T^{jn}x) \prod_{k=1}^K g_k (S^{kn}y) \, . \]
In particular, we will show that if one of the functions, $f_j$ for some $1 \leq j \leq J$, is a $J+K-1$-th order Wiener-Wintner function of power type $\alpha > 0$, then there exists a set of full measure in $X$ that is independent of the other system $(Y, \mathcal{G}, \nu, S)$ and functions $g_1, \ldots, g_K$ such that the averages above converge to $0$, provided that the pointwise convergence of multiple ergodic averages hold on the system $(Y, \mathcal{G}, \nu, S)$.

\begin{theorem}[Multiple recurrence return times theorem for Wiener-Wintner functions]\label{t:returnTimesMain}
		Let $(X, \mathcal{F}, \mu, T)$ be an ergodic measure-preserving system, let $J, K \in \N$, and $K \geq 2$ and let $f_1, f_2, \ldots f_J \in L^\infty(\mu)$. Suppose that, for some $\alpha>0$, $f_1$ is a $J+K-1$-th Wiener-Wintner function of power type $\alpha$. Then there exists $X' \subset X$ such that $\mu(X') = 1$, and for every $x \in X'$, for any other measure-preserving system $(Y, \mathcal{G}, \nu, S)$ and $g_1, g_2, \ldots, g_K \in L^\infty(\nu)$, and for $\nu$-a.e. $y \in Y$ we have
		\[\lim_{N \to \infty} \frac{1}{N} \sum_{n=1}^N \prod_{j=1}^J f_j(T^{jn}x) \prod_{k=1}^K g_k(S^{kn}y) = 0 \, . \]
	\end{theorem}
Of course, one can rearrange the functions so that the statement is valid if $f_j$ is a $J+K-1$-th Wiener-Wintner function for any $j \in [J]$. 

In order to prove Theorem \ref{t:returnTimesMain}, we need the following lemma.

\begin{lemma}\label{RT_WWest}
Let $(X, \mathcal{F}, \mu, T)$ be a measure-preserving system, let $J, K \in \N$ for which $K \geq 2$, and let $f_1, f_2, \ldots f_J \in L^\infty(\mu)$. Then for every $x \in X$, there exists a constant $C > 0$ such that for every $N \in \N$, for every measure-preserving system $(Y, \mathcal{G}, \nu, S)$, and for every $g_1, g_2, \ldots, g_K \in L^\infty(\nu)$ such that $\max_{k \in [K]} \norm{g_k}_{L^\infty(\nu)} \leq 1$, we have
	\begin{align*}
	&\norm{\frac{1}{N} \sum_{n=1}^N \prod_{j=1}^J f_j(T^{jn}x) \prod_{k=1}^K g_k \circ S^{kn} }_{L^2(\nu)}  \\
	&\leq C \left(\frac{1}{\floor{\sqrt{N}}} + \frac{1}{\floor{\sqrt{N}}^{K-1}} \sum_{h \in [\floor{\sqrt{N}}]^{K-1}} \sup_t \left| \frac{1}{N} \sum_{n=1}^N  e^{2\pi int} \prod_{j=1}^J \left[ \prod_{\eta \in V_{K-1}} (c^{|\eta|} f_j \circ T^{jh\cdot\eta} )\right]  (T^{jn}x) \right| \right)^{2^{-(K-1)}} \, .
	\end{align*}
\end{lemma}
	\begin{proof}
		Without loss of generality, we will assume that $\norm{f_j}_{L^\infty(\mu)} \leq 1$ for every $j = 1, 2, \ldots, J$.
		
We denote $F_{j, h} := \overline{f_j} \cdot f_j \circ T^{jh}$ and $G_{k, h} := \overline{g_k} \cdot g_k \circ S^{k h}$ for $j = 1, 2, \ldots J$ and $k = 1, 2, \ldots, K$. 

We will proceed by induction on $K$. We first consider the base case $K=2$. Let $y \in Y$. Van der Corput's lemma (Lemma \ref{vdc-lem}) tells us that for each $N \in \N$ and $1 \leq H < N$, we have
	\begin{align*}
	&\left| \frac{1}{N} \sum_{n=1}^N \prod_{j=1}^J f_j(T^{jn}x) \prod_{k=1}^2 g_k(S^{kn}y) \right|^2 \\
	&\leq 	\frac{2}{N(H+1)} \sum_{n=0}^{N-1} \left|\prod_{j=1}^J f_j(T^{jn}x) \prod_{k=1}^2 g_k(S^{kn}y) \right|^2 \\
	&+ \frac{2(N+H)}{N^2(H+1)^2} \sum_{h=1}^H (H+1-h) \real\left( \sum_{n=0}^{N-h-1} \prod_{j=1}^J F_{j, h}(T^{jn}x) \prod_{k=1}^2 G_{k, h}(S^{kn}y) \right).
	\end{align*}
We integrate both sides of the inequality with respect to $\nu$ to get
\begin{align*}
	&\norm{ \frac{1}{N} \sum_{n=1}^N \prod_{j=1}^J f_j(T^{jn}x) \prod_{k=1}^2 g_k \circ S^{kn}}_{L^2(\nu)} ^2 \\
	&\leq 	\frac{2}{N(H+1)} \sum_{n=0}^{N-1} \norm{ \prod_{j=1}^J f_j(T^{jn}x) \prod_{k=1}^2 g_k \circ S^{kn}}_{L^2(\nu)}^2 \\
	&+ \frac{2(N+H)}{N^2(H+1)^2} \sum_{h=1}^H (H+1-h) \real\left( \int  \sum_{n=0}^{N-h-1} \prod_{j=1}^J F_{j, h}(T^{jn}x) \prod_{k=1}^2 G_{k, h} \circ S^{kn}  \, d\nu \right)\\
	&\leq \frac{2}{H+1} + \frac{4}{H+1} \sum_{h=1}^H \left| \int \frac{1}{N} \sum_{n=0}^{N-h-1} \prod_{j=1}^J F_{j, h}(T^{jn}x) G_{1, h} \cdot G_{2, h} \circ S^{n} \, d\nu \right|.
\end{align*}
Since this estimate holds for any $1 \leq H \leq N-1$, we set $H = \floor{\sqrt{N}}$. Furthermore, by the triangle inequality and H\"older's inequality, we get
\begin{align*}
& \left| \int \frac{1}{N} \sum_{n=0}^{N-h-1} \prod_{j=1}^J F_{j, h}(T^{a_jn}x) G_{1, h} \cdot G_{2, h} \circ S^{n} \, d\nu \right| \\
&\leq \frac{1}{\floor{\sqrt{N}}} + \int \left| G_{1, h}\frac{1}{N} \sum_{n=0}^{N} \prod_{j=1}^J F_{j, h}(T^{a_jn}x) G_{2, h} \circ S^{n} \right| d\nu\\
&\leq \frac{1}{\floor{\sqrt{N}}} + \norm{\frac{1}{N} \sum_{n=0}^{N} \prod_{j=1}^J  F_{j, h}(T^nx) G_{2, h} \circ S^n }_{L^2(\nu)} \, .
\end{align*}
By the spectral theorem, we have the spectral measure of $G_{2, h}$, denoted $\sigma_{G_{2, h}}$, on $\T$ such that
\begin{align*} 
\norm{\frac{1}{N} \sum_{n=0}^{N} \prod_{j=1}^J F_{j, h}(T^nx) G_{2, h} \circ S^n }_{L^2(\nu)} 
&= \left( \int \left| \frac{1}{N} \sum_{n=1}^N \prod_{j=1}^J F_{j, h}(T^nx) e^{2\pi int} \right|^2 \, d\sigma_{G_{2, h}}(t) \right)^{1/2}\\
&\leq \sup_t \left| \frac{1}{N} \sum_{n=1}^N \prod_{j=1}^J F_{j, h}(T^nx) e^{2\pi int} \right|.  
\end{align*}
This completes the base case.

Now suppose the statement is true for $K = q$. To show the statement holds for the case $K=q+1$, we again start with the Van der Corput lemma and H\"older's estimate to see that
\begin{align*}
&\norm{ \frac{1}{N}  \sum_{n=1}^N \prod_{j=1}^J f_j(T^{jn}x) \prod_{k=1}^{q+1} g_k \circ S^{k n} }_{L^2(\nu)}^2 \\
&\leq  \frac{2}{\floor{\sqrt{N}}} + \frac{4}{\floor{\sqrt{N}}} \sum_{h=1}^{\floor{\sqrt{N}}} \norm{\frac{1}{N} \sum_{n=1}^N \prod_{j=1}^J F_{j, h}(T^{jn}x) \prod_{j=2}^{q+1} G_{j, h} \circ S^{j n} }_{L^2(\nu)} \, .
\end{align*}
By applying the inductive hypothesis, there exists a constant $C > 0$ such that for every $h \in \N$ such that $1 \leq h \leq \floor{N}$, for every measure-preserving system $(Y, \mathcal{G}, \nu, S)$ and $K$ many functions $G_{2, h}, \ldots G_{K+1, h}$, such that
\begin{align*}
&\norm{ \frac{1}{N} \sum_{n=1}^N \prod_{j=1}^J f(T^{jn}x) \prod_{k=1}^{q+1} g_k \circ S^{k n} }_{L^2(\nu)}^2 \\
&\leq  \frac{2}{\floor{\sqrt{N}}} + \frac{4C}{\floor{\sqrt{N}}} \sum_{h=1}^{\floor{\sqrt{N}}}  \\
& \left(\frac{1}{\floor{\sqrt{N}}} + \frac{1}{\floor{\sqrt{N}}^{q-1}} \sum_{h \in [\floor{\sqrt{N}}]^{q-1}} \sup_t \left| \frac{1}{N} \sum_{n=1}^N  e^{2\pi int} \prod_{j=1}^J \left[ \prod_{\eta \in V_{q-1}} (c^{|\eta|} f_j \circ T^{jh\cdot\eta} )  (T^{jn}x)\right] \right| \right)^{2^{-(q-1)}} \, .
\end{align*}
Once we apply the Cauchy-Schwarz inequality ($q-1$ times), and noting the sub-additivity of the function $\xi \mapsto \xi^{2^{-(q-1)}}$, we may conclude the proof.
\end{proof}

\begin{proof}[Proof of Theorem \ref{t:returnTimesMain}]
	We first consider the case $J=1$. Because $f_1$ is a $K$-th order WW function of type $\alpha$, and since we may and will assume that $\norm{f_1}_{L^\infty(\mu)} \leq 1$ without loss of generality, there exists a constant $C>0$ such that for every $N \in \N$, we have
	\[ \frac{1}{\floor{\sqrt{N}}^{K-1}} \sum_{h \in [\floor{\sqrt{N}}]^{K-1}} \norm{ \sup_t \left| \frac{1}{N} \sum_{n=1}^N  e^{2\pi int}  \left[ \prod_{\eta \in V_{K-1}} (c^{|\eta|} f_1 \circ T^{h\cdot\eta} ) \circ T^n \right] \right| }_{L^1(\mu)} <  \frac{C}{N^{\alpha}} \, . \]
	
	Choose $\gamma \in \N$ such that $\alpha \gamma > 1$. Then
	\begin{align*} 
	&\sum_{M=2}^\infty \frac{1}{\floor{M^{\gamma/2}}^{K-1}} \sum_{h \in [\floor{\sqrt{M^{\gamma}}}]^{K-1}} \int \sup_t \left| \frac{1}{\floor{\sqrt{M^\gamma}}} \sum_{n=1}^{\floor{\sqrt{M^\gamma}}} e^{2\pi int} \prod_{\eta \in V_{K-1}} c^{|\eta|} (f_1 \circ T^{h \cdot \eta}) \circ T^n \right| \, d\mu  \\
	&< \sum_{M=2}^\infty \frac{C}{M^{\alpha \gamma}} < \infty.  
	\end{align*}
	By the monotone convergence theorem, we may conclude that there exists a set of full measure $X'$ such that for every $x \in X'$,
	\[ \sum_{M=2}^\infty \frac{1}{\floor{M^{\gamma/2}}^{K-1}} \sum_{h \in [\floor{M^{\gamma/2}}]^{K-1}} \sup_t \left| \frac{1}{M^\gamma} \sum_{n=1}^{M^\gamma} e^{2\pi int} \prod_{\eta \in V_{m-1}} c^{|\eta|} (f_1 \circ T^{h \cdot \eta})  (T^nx) \right| < \infty. \]
    Without loss of generality, we will assume that $\max_{k \in [K]} \norm{g_k}_{L^\infty(\nu)} \leq 1$. Lemma \ref{RT_WWest} tells us that for every $x \in X'$, we have
	\[\sum_{M=2}^\infty \norm{ \frac{1}{M^\gamma} \sum_{n=1}^{M^\gamma} f_1(T^nx) \prod_{k=1}^K g_k \circ S^{k n}}_{L^1(\nu)} < \infty. \]
	Again by the monotone convergence theorem, for every $x \in X'$ and $\nu$-a.e. $y \in Y$, we have
	\[ \sum_{M=2}^\infty  \left|\frac{1}{M^\gamma} \sum_{n=1}^{M^\gamma} f_1(T^nx) \prod_{k=1}^K g_k (S^{k n}y) \right| < \infty \, . \]
	 Therefore, for every $x \in X'$ and for $\nu$-a.e. $y \in Y$, we have
	\[ \lim_{M \to \infty} \frac{1}{M^\gamma} \sum_{n=1}^{M^\gamma} f_1(T^nx) \prod_{k=1}^K g_k (S^{k n}y) = 0\, .\]
	The remainder of the argument (i.e. convergence of the averages for the case outside of the subsequence $(M^\gamma)$) is standard, with details in the proof of Theorem \ref{t:csp}.
	
	One may use similar argument to show that the theorem holds for the case $J > 1$, but requires additional intermediate steps. Without loss of generality, we will assume that $\max_{j \in [J]} \norm{f_j}_{L^\infty(\mu)} \leq 1$. In order to apply Lemma \ref{RT_WWest}, we must have control over the average
	\[ \sup_t \left| \frac{1}{N} \sum_{n=1}^N  e^{2\pi int} \prod_{j=1}^J \left[ \prod_{\eta \in V_{K-1}} (c^{|\eta|} f_j \circ T^{jh\cdot\eta} )  (T^{jn}x)\right] \right|^2 \, . \]
	By applying the Van der Corput lemma (\ref{vdc-1}), we have
    \begin{align}\label{est1}
        &\sup_t \left| \frac{1}{N} \sum_{n=1}^N  e^{2\pi int} \prod_{j=1}^J \left[ \prod_{\eta \in V_{K-1}} (c^{|\eta|} f_j \circ T^{jh\cdot\eta} )  (T^{jn}x)\right] \right|^2 \\
        &\leq \frac{2}{\floor{\sqrt{N}}} + \frac{4}{\floor{\sqrt{N}}} \sum_{\ell=1}^{\floor{\sqrt{N}}} \left| \frac{1}{N} \sum_{n=1}^N \prod_{j=1}^J \left( \prod_{\eta \in V_{K}} c^{|\eta|} f_j \circ T^{j (h, \ell)\cdot \eta}  \right) (T^{jn}x) \right| \nonumber
    \end{align} 
    for every $x \in X$ (where $(h, l)$ denotes the vector of length $K$ obtained by adjoining $l$ to $h$). With this estimate in mind, we consider the integral of the right-hand side of the inequality in Lemma \ref{RT_WWest} with respect to $\mu$. After applying H\"older's estimate several times (as well as the Cauchy-Schwarz inequality to square the $\sup_t$ term) and applying the estimate (\ref{est1}), there exists a positive constant $C_1>0$ for which the integral is bounded above by
	\begin{align*}
	&C_1 \left[ \frac{1}{\floor{\sqrt{N}}} + \left( \frac{1}{\floor{\sqrt{N}}^{K-1}} \sum_{h \in [\floor{\sqrt{N}}]^{K-1}} \Bigg( \frac{1}{\floor{\sqrt{N}}} \right. \right. \\
	&\left. \left. \left. + \frac{1}{\floor{\sqrt{N}}} \sum_{\ell=1}^{\floor{\sqrt{N}}} \norm{\frac{1}{N} \sum_{n=1}^N \prod_{j=1}^J \left( \prod_{\eta \in V_K} c^{|\eta|} f_j \circ T^{j(h, \ell) \cdot \eta} \right) \circ T^{jn} }_{L^1(\mu)}  \right) \right)^{1/2} \right]^{2^{-(K-1)}} \, . \end{align*}
	Let us focus on the average
	\[\frac{1}{\floor{\sqrt{N}}^{K-1}} \sum_{h \in \floor{\sqrt{N}}^{K-1}} \frac{1}{\floor{\sqrt{N}}} \sum_{\ell=1}^{\floor{\sqrt{N}}} \norm{\frac{1}{N} \sum_{n=1}^N \prod_{j=1}^J \left( \prod_{\eta \in V_K} c^{|\eta|} f_j \circ T^{j(h, \ell) \cdot \eta} \right) \circ T^{jn} }_{L^1(\mu)} \, , \]
	which equals to
	\[\frac{1}{\floor{\sqrt{N}}^{K}} \sum_{h \in \floor{\sqrt{N}}^{K}}  \norm{\frac{1}{N} \sum_{n=1}^N \prod_{j=1}^J \left( \prod_{\eta \in V_K} c^{|\eta|} f_j \circ T^{jh \cdot \eta} \right) \circ T^{jn} }_{L^1(\mu)} \, . \]
	By one of our Bourgain bounds (Theorems \ref{t:BBfor2}, or \ref{BBgeq3}), there exists a positive constant $C_2>0$ (that does \textit{not} depend on the values of $h \in [\floor{\sqrt{N}}]^{K-1}$) such that for every $N \in \N$, last average is bounded above by
	\begin{align*}
		&\frac{C_2}{\floor{\sqrt{N}}^K} \sum_{h \in [\floor{\sqrt{N}}]^K}\left( \frac{1}{\floor{\sqrt{N}}^{2^{-(J-2)}}} + \frac{1}{\floor{\sqrt{N}}^{J-2}} \sum_{\ell \in {\floor{\sqrt{N}}}^{J-2}} \right. \\
		&\left.  \norm{\sup_t \left| \frac{1}{N} \sum_{n=1}^N e^{2\pi int} \left[\prod_{\rho \in V_{J-2}} c^{|\rho|} \left( \prod_{\eta \in V_k} c^{|\eta|} f_1 \circ T^{h \cdot \eta} \right) \circ T^{\ell \cdot \rho} \right] \circ T^n \right| }_{L^1(\mu)}^{2/3} \right)^{2^{-(J-2)}} \, .
	\end{align*}
	By applying the Cauchy-Schwarz inequality, the last display is bounded above by
	\begin{align*}
		&C_2\left( \frac{1}{\floor{\sqrt{N}}^{2^{-(J-2)}}} + \frac{1}{\floor{\sqrt{N}}^{J+K-2}} \sum_{h \in {[\floor{\sqrt{N}}}]^{J+K-2}} \right. \\
		&\left.  \norm{\sup_t \left| \frac{1}{N} \sum_{n=1}^N e^{2\pi int} \left[\prod_{\eta \in V_{J+K-2}} c^{|\eta|} \left( f_1 \circ T^{h \cdot \eta} \right)\right] \circ T^n \right| }_{L^1(\mu)}^{2/3} \right)^{2^{-(J-2)}} \, .
	\end{align*}
	By recalling that $f_1$ is a $J+K-1$-th WW function, one may apply the argument involving the monotone convergence theorem as we have witnessed for the case $J=1$.
\end{proof}

\begin{corollary}\label{c:retutnTimesCF}
     Let $J, K \in \N$, and $K \geq 2$. Let $(X, \mathcal{F}, \mu, T)$ be a $J+K-1$-th Wiener-Wintner system, and let $f_1, f_2, \ldots f_J \in L^\infty(\mu)$. Suppose that for some $j \in [J]$, $f_j \in L^2(\mathcal{Z}_{J+K-1})^\perp$. Then there exists a set $X' \subset X$ such that $\mu(X') = 1$, and for every $x \in X'$ and any other measure-preserving system $(Y, \mathcal{G}, \nu, S)$ and $g_1, g_2, \ldots, g_K \in L^\infty(\nu)$, we have
		\[\lim_{N \to \infty} \frac{1}{N} \sum_{n=1}^N \prod_{j=1}^J f_j(T^{jn}x) \prod_{k=1}^K g_k(S^{kn}y) = 0 \]
		for $\nu$-a.e. $y \in Y$.
 \end{corollary}

 \begin{proof}
Without loss of generality, we will assume that $f_1 \in L^2(\mathcal{Z}_{J+K-1})^\perp$. We will further assume that for every $j \in [J]$, $\norm{f_j}_{L^\infty(\mu)} \leq 1$. Let $(Y, \mathcal{G}, \nu, S)$ be any measure-preserving system, and for every $k \in [K]$, we will also assume that $g_k \in L^\infty(\nu)$ for which $\norm{g_k}_{L^\infty(\nu)} \leq 1$.

Because $(X, \mathcal{F}, \mu, T)$ is a $J+K-1$-th WW system, there exists a sequence $(\phi_m)_{m \in \N}$ in $L^2(\mathcal{Z}_{J+K-1})^\perp$ such that $\phi_m \to f$ in $L^2(\mu)$-norm. We will assume (by possibly passing to a subsequence) that we have $\norm{f_1 - \phi_m} < m^{-2}.$ By the maximal inequality (Lemma \ref{l:maxIneq}), for $\nu$-a.e. $y \in Y$ and for every $m \in \N$, one has
\begin{align*}
& \norm{ \sup_N \left| \frac{1}{N} \sum_{n=1}^N \prod_{j=1}^J f_j \circ T^{jn} \prod_{k=1}^K g_k(S^{kn}y) - \frac{1}{N} \sum_{n=1}^N \phi_m \circ T^n \prod_{j=2}^J f_j \circ T^{jn}x \prod_{k=1}^K g_k(S^{kn}y) \right|  }_{L^1(\mu)} \\
&\leq \norm{ \sup_N \frac{1}{N} \sum_{n=1}^N | f_1 - \phi_m | \circ T^n }_{L^2(\mu)} \leq 2 \norm{f_1 - \phi_m}_{L^2(\mu)}  \leq 2m^{-2} \, .
\end{align*}
By summability of the sequence $m^{-2}$, we apply the monotone convergence theorem to show that there exists a set of full measure $\tilde{X} \subset X$ such that for every $x \in \tilde{X}$, we have
\begin{equation}\label{m_to_infty}
\lim_{m \to \infty} \sup_N \left| \frac{1}{N} \sum_{n=1}^N \prod_{j=1}^J f_j(T^{jn}x) \prod_{k=1}^K g_k(S^{kn}y) - \frac{1}{N} \sum_{n=1}^N \phi_m (T^nx) \prod_{j=2}^J f_j \circ T^{jn} \prod_{k=1}^K g_k(S^{kn}y) \right| = 0 \, .
\end{equation}

For any $m \in \N$, let $X_m \subset X$ be the set of full measure from Theorem \ref{t:returnTimesMain} that is associated to $\phi_m$ as well as the functions $f_2, f_3, \ldots, f_J$. Set $X' = \tilde{X} \cap \left( \bigcap_{m=1}^\infty X_m \right)$. Clearly, $\mu(X') = 1$, and we claim that this is the desired set of full measure. Suppose that $x \in X'$. For any other measure-preserving system $(Y, \mathcal{G}, \nu, S)$ and $g_1, g_2, \ldots g_k \in L^\infty(\mu)$, we may apply the triangle inequality to show that for any $m \in \N$, we have
\begin{align*}
    & \limsup_{N \to \infty} \left|\frac{1}{N} \sum_{n=1}^N \prod_{j=1}^J f_j(T^{jn}x) \prod_{k=1}^K g_k(S^{kn}y) \right| \\
    &\leq \limsup_{N \to \infty} \left| \frac{1}{N} \sum_{n=1}^N \prod_{j=1}^J f_j(T^{jn}x) \prod_{k=1}^K g_k(S^{kn}y) - \frac{1}{N} \sum_{n=1}^N \phi_m(T^nx) \prod_{j=2}^J f_j(T^{jn}x) \prod_{k=1}^K g_k(S^{kn}y) \right| + 0 \\
    &\leq \sup_{N \in \N} \left| \frac{1}{N} \sum_{n=1}^N \prod_{j=1}^J f_j(T^{jn}x) \prod_{k=1}^K g_k(S^{kn}y) - \frac{1}{N} \sum_{n=1}^N \phi_m(T^nx) \prod_{j=2}^J f_j(T^{jn}x) \prod_{k=1}^K g_k(S^{kn}y) \right| \, .
\end{align*}
We may conclude the proof by letting $m \to \infty$ and applying (\ref{m_to_infty}).
\end{proof}

\begin{remark}
In the statement of Corollary \ref{c:retutnTimesCF}, we can replace $\mathcal{Z}_{J+K-1}$ with any $\sigma$-subalgebra $\mathcal{B}$ for which the $(J+K-1)$-th order WW functions (of a certain power type) are dense in $L^2(\mathcal{B})^\perp$. For instance, Corollary \ref{c:PinskerWW} tells us that the Pinsker algebra $\mathcal{P}$ is such a factor. Furthermore, if $X := X_1 \times X_2$ where $(X_1, \mathcal{F}_1, \mu_1, T_1)$ is a K system and $(X_2, \mathcal{F}_2, \mu_2, T_2)$ is an ergodic system, then the $\sigma$-subalgebra $X \times \mathcal{F}_2$ is a characteristic factor of the multiple recurrence return times theorem by Theorem \ref{t:Kprod}. If $(X_2, \mathcal{F}_2, \mu_2, T_2)$ is a $J+K-1$-th order WW system, then the $\sigma$-subalgebra $X_1 \times \mathcal{Z}^Y_{J+K-1}$ is the characteristic factor for the multiple recurrence return times theorem by Theorem \ref{t:productWWS}.
\end{remark}

\appendix
\section{Proof of Theorem \ref{t:Ksys} (Pinsker algebra/K system example) for general case}\label{a:PA}
\begin{theorem}\label{t:generalPA}
The span of functions $f_k^A$, where $A\in T^{-l}\A$ and $k\in \Z$, are $J$-th order WW functions of power type $1/6$ in $L^2$ for all $J\geq 1$
\end{theorem}

\begin{proof}

For every $\eta \in V_{J-1}$, we let $k_\eta, l_\eta \in \Z$ and $A_\eta \in T^{-l_\eta}\mathcal{A}$. Again, we may assume that $l_\eta < k_\eta$. For a fixed $h\in \N^{J-1}$, we are interested in 
$$ \left \Vert \sup_{t} \left | \frac{1}{N} \sum_{n=1}^N e^{2\pi i n t} \left[ \prod_{\eta\in V_{J-1}} f^{A_\eta}_{k_\eta} \circ T^{\eta \cdot h}\right] \circ T^n  \right| \right\Vert_2^{2/3}$$

For notational convention, let $(0, \dots, 0) \in V_{J-1}$ be denoted $0$. Let $h = (h_1, \dots, h_{J-1})$ be such that for each component $h_i$ we have $h_i > \max_\eta\{l_0 - l_\eta, k_0 - l_\eta, k_0 - k_\eta, 1\}$. This ensures that for all $\eta \neq 0$, the following conditions hold:
$$\eta \cdot h + l_\eta > l_0 \quad \quad \eta\cdot h + l_\eta > k_0 \quad \quad \eta \cdot h + k_\eta > k_0$$
For all such $h$, we establish polynomial decay on the above term. Since only finitely many $h$ fail to satisfy these conditions, we bound their terms trivially and their contribution is lost in the average.

Define $F = \prod_{\eta\in V_{J-1}} f^{A_\eta}_{k_\eta} \circ T^{\eta \cdot h} $. As each $\Vert f_{k_\eta}^{A_\eta}\Vert_\infty \leq 2$, we have $\Vert F\Vert_\infty \leq 2^{2^{J-1}}$. As before, we compute pointwise by Van der Corput to see that
\begin{align*}
\sup_{t} \left|\sum_{n=1}^N e^{2\pi i n t} F(T^{n} x) \right|^2 &\leq 2^{2^{J}+1}N + 4\sum_{m = 1}^{N-1}\left| \sum_{n = 1}^{N-m} F(T^{n} x)F(T^{n + m} x)\right| \\
&\leq 2^{2^{J}+1}N + 2^{2^{J}+2}(k_0 - l_0)N + 4\sum_{m = k_0 - l_0+1}^{N-1}\left| \sum_{n = 1}^{N-m} F(T^{n} x)F(T^{n + m} x)\right|
\end{align*}
holds for almost all $x$. By integrating both sides and bounding the $L^1(\mu)$ norm by the $L^2(\mu)$ norm, we have
\begin{equation} \label{gen-k-est}
\begin{split}
&\int \sup_{t} \left| \sum_{n=1}^N e^{2\pi i n t} F(T^{n} x) \right|^2 \, d\mu(x) \\ &\leq 2^{2^{J}+1}N + 2^{2^{J}+2}(k_0 - l_0)N + 4\sum_{m = k_0 - l_0+1}^{N-1} \left(\int\left| \sum_{n = 1}^{N-m} F(T^{n} x)F(T^{n + m} x)\right|^2 \, d\mu(x)\right)^{1/2} \, .
\end{split}
\end{equation}
Consider the inner $n$ sum, which is squared. If we factor out this product, the diagonal terms can be bounded away by $2^{2^{J+1}}N$, and we are left with off-diagonal terms
\begin{align*}
&2\sum_{n < j}^{N-m} \int  F(T^{n} x)F(T^{n + m} x)F(T^{j} x)F(T^{j + m} x) \, d\mu(x) \\
&= 2\sum_{n < j}^{N-m} \int \prod_{\eta\in V_{J-1}} f^{A_\eta}_{k_\eta} \circ T^{n + \eta \cdot h} \cdot f^{A_\eta}_{k_\eta} \circ T^{n + m + \eta \cdot h} \cdot f^{A_\eta}_{k_\eta} \circ T^{j + \eta \cdot h} \cdot f^{A_\eta}_{k_\eta} \circ T^{j+m + \eta \cdot h}\, d\mu \, .
\end{align*}
Recall that $m > k_0 - l_0$. We claim that if $j > n + k_0 - l_0$, then 
\begin{align} \label{gen-int}
\int \prod_{\eta\in V_{J-1}} f^{A_\eta}_{k_\eta} \circ T^{n + \eta \cdot h} \cdot f^{A_\eta}_{k_\eta} \circ T^{n + m + \eta \cdot h} \cdot f^{A_\eta}_{k_\eta} \circ T^{j + \eta \cdot h} \cdot f^{A_\eta}_{k_\eta} \circ T^{j+m + \eta \cdot h}\, d\mu = 0 \, .
\end{align}
To this end, we consider the case $j > n + m$. Notice that in the product, the 4 above functions are measurable to the degrees $T^{-q}\A$ for the following values of $q$:
$$n + \eta\cdot h + l_\eta \quad n + m + \eta\cdot h + l_\eta \quad j + \eta\cdot h + l_\eta \quad j + m + \eta\cdot h + l_\eta $$
Hence, by the condition $j > n+m$, and the initially chosen conditions on $h$, we have

$$\prod_{\eta\in V_{J-1}} f^{A_\eta}_{k_\eta} \circ T^{n + m + \eta \cdot h} \cdot f^{A_\eta}_{k_\eta} \circ T^{j + \eta \cdot h} \cdot f^{A_\eta}_{k_\eta} \circ T^{j+m + \eta \cdot h} \quad \text{is } T^{-(n + m + l_0)}\A \text{ measurable.}$$
So for the integral (\ref{gen-int}), we would get the same value if we conditioned the integrand on $T^{-(n + m + l_0)}\A$, and the above product of functions can factor out of this conditional expectation. So the claim that (\ref{gen-int}) is zero reduces to showing 
$$\E\left( \prod_{\eta \in V_{J-1}} f^{A_\eta}_{k_\eta} \circ T^{n+\eta \cdot h} \Bigg| T^{-(n + m + l_1)}\A\right) = 0 \, .$$
Recall that $f_{k_\eta}^{A_\eta} = \1_{A_\eta} - \E(\1_{A_\eta} | T^{-k_\eta})$. Hence, the product of many such functions factors out. As a slight abuse of notation, we let $\theta \in V_{2^{J-1}-1}$ be the list of elements of $\{0, 1\}^{2^{J-1}-1}$, but indexed by nonzero $\eta \in V_{J-1}$, of which there are $2^{J-1}-1$. Define
$$G_\theta = \prod_{\overset{\eta \in V_{J-1} - \{0\}}{\theta_\eta = 0}} \1_{A_\eta}\circ T^{\eta \cdot h} \prod_{\overset{\eta \in V_{J-1} - \{0\}}{\theta_\eta = 1}} (-1)\E(\1_{A_\eta} | T^{-k_\eta}\A) \circ T^{\eta\cdot h}$$
Notice that all of the functions in the left product are $T^{-(\eta\cdot h+ l_\eta)}\A$ measurable and all of the functions in the right product are $T^{-(\eta\cdot h+ k_\eta)}\A$ measurable. We expand the product to see
$$\prod_{\eta \in V_{J-1}} f^{A_\eta}_{k_\eta} \circ T^{n+\eta \cdot h} = \sum_{\theta \in V_{2^{J-1}-1}} \1_{A_0} \circ T^n \cdot G_\theta \circ T^n - \E(\1_{A_0} | T^{-k_0} \A) \circ T^n \cdot G_\theta \circ T^n$$
We claim for each $\theta$ that the above two summands cancel under the $T^{-(n + m + l_0)}\A$ conditional. Looking at the left summand, we have

$$\E(\1_{A_0} \circ T^n \cdot G_\eta \circ T^n | T^{-(n + m + l_0)}\A) = \E(\1_{A_0} \cdot G_\theta | T^{-(m + l_0)}\A) \circ T^n \, ,$$
On the right, all of our conditions on $h$ allow $G_\theta$ to pass inside the $T^{-k_0}\A$ conditioning:
\begin{align*}
\E(\E(\1_{A_0}|T^{-k_0}\A) \circ T^n \cdot G_\theta \circ T^n | T^{-(n + m + l_0)}\A) &= \E(\E(\1_{A_0}|T^{-k_0}\A) \cdot G_\theta | T^{-( m + l_0)}\A) \circ T^n \\
&= \E(\E(\1_{A_0} \cdot G_\theta |T^{-k_0}\A)  | T^{-( m + l_0)}\A) \circ T^n \\
&= \E(\1_{A_0} \cdot G_\theta  | T^{-( m + l_0)}\A) \circ T^n
\end{align*}
as $m +l_0 > k_0$. So these terms cancel, and the integral (\ref{gen-int}) does indeed vanish under the conditioning.

For the case $n + m \geq j > n + k_0 - l_0$, we similarly observe that
$$\prod_{\eta\in V_{J-1}} f^{A_\eta}_{k_\eta} \circ T^{n + m + \eta \cdot h} \cdot f^{A_\eta}_{k_\eta} \circ T^{j + \eta \cdot h} \cdot f^{A_\eta}_{k_\eta} \circ T^{j+m + \eta \cdot h} \quad \text{is } T^{-(j + l_0)}\A \text{ measurable.}$$
and by the same argument, conditioning (\ref{gen-int}) under $T^{-(j + l_0)}\A$ shows that it is zero.

Returning to our bound (\ref{gen-k-est}), we have observed that the off-diagonal terms vanish when $j$ is larger than $n$ by $k_0 - l_0$. Hence, for each $n = 1$ to $N - m$, at most $k_0 - l_0$ terms are nonzero, and the total number of nonzero terms in the sum over $n$ and $J$ can be bounded by $N(k_0 -l_0)$. Plugging this into our estimate (\ref{gen-k-est}), we see
\begin{align*}
&\int \sup_{t} \left| \sum_{n=1}^N e^{2\pi i n t} F(T^{n} x) \right|^2 \, d\mu(x) \\
&\leq 2^{2^{J}+1}N + 2^{2^{J}+2}(k_0 - l_0)N + 4\sum_{m = k_0 - l_0+1}^{N-1} \left(2^{2^{J+1}} N + 2^{2^{J+1}+1}N(k_0-l_0) \right)^{1/2}\, .
\end{align*}
Hence, we have the bound
\begin{align*}
\int \sup_{t} \left| \frac{1}{N} \sum_{n=1}^N e^{2\pi i n t} \left[ \prod_{\eta\in V_{J-1}} f^{A_\eta}_{k_\eta} \circ T^{\eta \cdot h}\right](T^nx) \right|^2 \, d\mu(x) &\leq \frac{C}{N^{1/2}} \, ,
\end{align*}
which is uniform in sufficiently large $h$ with all components $h_i > L$, as selected before. The constant $C$ depends only on $A_\eta, k_\eta, l_\eta$. Raising both sides to the $1/3$ power, we have
\begin{align*}
\left\Vert \sup_{t} \left| \frac{1}{N} \sum_{n=1}^N e^{2\pi i n t} \left[ \prod_{\eta\in V_{J-1}} f^{A_\eta}_{k_\eta} \circ T^{\eta \cdot h}\right]\circ T^n\right|\right\Vert_2^{2/3} &\leq \frac{C}{N^{1/6}} \, .
\end{align*}
For small $1 \leq h \leq L$, we have the bound
\begin{align*}
\left\Vert \sup_{t} \left| \frac{1}{N} \sum_{n=1}^N e^{2\pi i n t} \left[ \prod_{\eta\in V_{J-1}} f^{A_\eta}_{k_\eta} \circ T^{\eta \cdot h}\right] \circ T^n \right|\right\Vert_2^{2/3} &\leq 2^{2^{J}/3}
\end{align*}
by the triangle inequality. As before they are lost in the average for $N > (L + 1)^2$. Therefore,
\begin{align*}
&\quad\frac{1}{\lfloor \sqrt{N} \rfloor^{J-1}} \sum_{h \in [\lfloor \sqrt{N} \rfloor]^{J-1}} \left\Vert \sup_{t} \left| \frac{1}{N} \sum_{n=1}^N e^{2\pi i n t} \left[ \prod_{\eta\in V_{J-1}} f^{A_\eta}_{k_\eta} \circ T^{\eta \cdot h}\right]\circ T^n \right|\right\Vert_2^{2/3} \\
& = \frac{1}{\lfloor \sqrt{N} \rfloor^{J-1}} \sum_{h \in [L]^{J-1}} \left\Vert \sup_{t} \left| \frac{1}{N} \sum_{n=1}^N e^{2\pi i n t} \left[ \prod_{\eta\in V_{J-1}} f^{A_\eta}_{k_\eta} \circ T^{\eta \cdot h}\right] \circ T^n \right|\right\Vert_2^{2/3} \\
&\quad +\frac{1}{\lfloor \sqrt{N} \rfloor^{J-1}} \sum_{h \in \{L+1, \dots, \lfloor \sqrt{N} \rfloor\}^{J-1}} \left\Vert \sup_{t} \left| \frac{1}{N} \sum_{n=1}^N e^{2\pi i n t} \left[ \prod_{\eta\in V_{J-1}} f^{A_\eta}_{k_\eta} \circ T^{\eta \cdot h}\right]\circ T^n \right|\right\Vert_2^{2/3} \\
& \leq \frac{2^{2^{J}/3}L^{J-1}}{\lfloor \sqrt{N} \rfloor^{J-1}}  + \frac{(\lfloor \sqrt{N} \rfloor - L - 1)^{J-1}}{\lfloor \sqrt{N} \rfloor^{J-1}} \frac{C}{N^{1/6}} \leq \frac{C'}{N^{1/6}}
\end{align*}
for a larger $C'$ that still only depends on $A_\eta, l_\eta, k_\eta$. Since the bound $N>(L+1)^2$ only depends on the same constants, we can increase $C'$ further to get the above bound for small $N$, without changing the dependence.

Hence, the multilinearity conditions have been satisfied.
\end{proof}

\section{Proof of Theorem \ref{t:csp} (skew product example) for general case}\label{a:sp}

This proof uses the following lemma:
\begin{lemma}
Let $Q$ be a nonzero polynomial of degree $\leq J$ and in $k$ variables. Then
$$\#\{h\in [N]^k: Q(h) = 0\} \leq JkN^{k-1}$$
\end{lemma}
\begin{proof}
Consider induction on $k$. If $Q$ is a polynomial in $k=1$ variables, then it has less than $J$ distinct zeroes. Hence, this holds for all $N$.

For a general polynomial $Q$ in $k$ variables, we write
$$Q(h) = Q(h', h_k) = a_J(h')h_k^J + a_{J-1}(h')h_k^{J-1} + \dots + a_1(h')h_k + a_0(h') $$
where $a_j$ are polynomials in $k-1$ variables, and $h' = (h_1, \dots, h_{k-1})$. As $Q$ is nonzero, we can find $j$ so that $a_j(h')$ is not the zero polynomial in the variables $h_1, \dots, h_{J-1}$.

For $h' \in [N]^{k-1}$, $Q(h', \cdot)$ is a polynomial in one variable. If $Q(h', \cdot)$ is the zero polynomial, then for any $h_k\in [N]$, we have $Q(h) = 0$. In order for $Q(h', \cdot)$ to be the zero polynomial, it must be the case that all $a$'s vanish at $h'$. The number of such $h'$ is bounded by those where at least $a_j$ vanishes. Since $a_{j}$ has $k-1$ variables and is a polynomial of degree $\leq J$, by induction this is less than $J(k-1)N^{k-2}$. So the total number of zeroes that may arise when $Q(h', \cdot)$ is the zero polynomial is bounded by $N \cdot J(k-1)N^{k-2} = J(k-1)N^{k-1}$.

If $Q(h', \cdot)$ is not the zero polynomial, then again it has at most $J$ zeroes. So the number of zeroes that can arise in this case is bounded by $J$ times the total number of $h'$, which is $N^{k-1}$. Between the two cases, the total number of zeroes is bounded by $J(k-1)N^{k-1} + JN^{k-1} = JkN^{k-1}$ as desired.
\end{proof}

\begin{theorem}\label{t:generalCSP}
Let $k\geq 4$. For Lebesgue a.e.  $\alpha\in \R$, the system $(\T^k, T_{\alpha}, m)$ is a $k-1$th order WW system of power type $1/24$ in $L^2$.
\end{theorem}

\begin{proof} 
The orthogonal complement of $L^2(\mathcal{Z}_{k-1})^\perp$ of $(\T^k, T_\alpha, m)$ is spanned by functions 
$$f_a(x_1, x_2, \dots, x_k) = \prod_{l=1}^k e^{2\pi i a_l x_l}= e^{2\pi i( a \cdot x)} $$
where $a = (a_1, a_2, \dots, a_k)\in \Z^k$ has $a_k \neq 0$. By the multilinearity concerns , we are interested in 
$$\left \Vert \sup_{t} \left|\frac{1}{N} \sum_{n=1}^N e^{2\pi i n t} \left[ \prod_{\eta \in V_{k-2}} c^{|\eta|}f_{a^\eta} \circ T_\alpha^{h \cdot \eta} \right]\circ T_\alpha^n \right|\right \Vert_2 $$
for a fixed collection of $a^\eta = (a^\eta_1, a^\eta_2, \dots, a^\eta_k)$ with all $\eta$ having $a^\eta_k \neq 0$ and a fixed $h = (h_1, \dots, h_{k-2})\in \mathbb{N}^{k-2}$. Suppose $\alpha \neq 0$. For $x \in \T^k$, we compute
\begin{align*}
f_{a^\eta}(T_\alpha^nx) = \text{Exp}\big[& 1(a^\eta_1x_1 + a^\eta_2x_2 + \dots + a^\eta_kx_k) 
\\ + &n(a^\eta_1\alpha + a^\eta_2 x_1 + \dots a^\eta_k x_{k-1}) 
\\ + &P_2(n)(a^\eta_2\alpha + a_3^\eta x_1 + \dots + a^\eta_k x_{k-2}) 
\\ + & \dots
\\ + &P_{k-1}(n)(a^{\eta}_{k-1}\alpha + a^\eta_k x_1)
\\ + &P_k(n)(a_k\alpha) \big] \, ,
\end{align*}
from which it follows
\begin{align*}
&\left[\prod_{\eta\in V_{k-2}} c^{|\eta|} f_{a^\eta} \circ T_\alpha^{h \cdot \eta} \right](T_\alpha^n x)
\\ &= \text{Exp} \left[ \begin{pmatrix} \sum_{\eta \in V_{k-2}} (-1)^{|\eta|}(n+ \eta\cdot h)(a^\eta_1)
\\ + \sum_{\eta \in V_{k-2}} (-1)^{|\eta|}P_2(n+ \eta\cdot h)(a^\eta_2)
\\ + \dots 
\\ + \sum_{\eta \in V_{k-2}} (-1)^{|\eta|}P_{k-1}(n+ \eta\cdot h)(a^\eta_{k-1})
\\ + \sum_{\eta \in V_{k-2}} (-1)^{|\eta|}P_k(n+ \eta\cdot h)(a^\eta_k)\end{pmatrix} \alpha + \begin{pmatrix} \sum_{\eta \in V_{k-2}} (-1)^{|\eta|} 1 (a^\eta_1x_1 + \dots + a^\eta_k x_k) \\ + \sum_{\eta \in V_{k-2}} (-1)^{|\eta|}(n+ \eta\cdot h)(a^\eta_2x_1 + \dots + a^\eta_k x_{k-1}) \\ + \dots \\ + \sum_{\eta \in V_{k-2}} (-1)^{|\eta|}P_{k-2}(n+ \eta\cdot h)(a^\eta_{k-1} x_1 + a^\eta_k x_2) \\ + \sum_{\eta \in V_{k-2}} (-1)^{|\eta|}P_{k-1}(n+ \eta\cdot h)(a^\eta_k x_1) \end{pmatrix} \right] \\
&=: \text{Exp}\left[ Q(n)\alpha + \dot Q_x(n)\right] \, .
\end{align*}
Consider the coefficients of the polynomial $Q(n)$ to depend on $h$ and the fixed $a^\eta$. Since all $a^\eta_j$ are integers, and each $P_j$ is integer-valued, it follows that $Q$ is integer valued for $n$. The remainder polynomial $\dot Q_x$ depends on $a^\eta, h$ and the point $x$.

As each $P_j$ has degree exactly $j$, we see that the degree of $Q(n)$ is at most $k$. Consider that the $n^2$ coefficient of $Q(n)$, which is a polynomial in $h = (h_1, \dots, h_{k-2})$, is given by
$$q_2(h) := \frac{Q''(0)}{2} = \frac{1}{2}\sum_{\eta \in V_{k-2}} (-1)^{|\eta|}\left( \sum_{j=1}^{k} a_j^\eta P_j''(\eta \cdot h)\right) \, ,$$
where the derivatives are taken in $n$. We observe that $q_2(h)$ is not the zero polynomial in $h$, as there is a nonzero $h_1h_2\dots h_{k-2}$ term that is contributed only by $P''_k(\eta\cdot h)$ for $\eta = (1, 1, \dots, 1)$. If we denote $H_0\subset \mathbb{N}^{k-2}$ to be $h$ where $q_2(h)= 0$, it follows by the lemma that
$$\#(H_0 \cap [N]^{k-2}) \leq (k-2)^2N^{k-3} $$
for all $N$. Hence, for $h$ not in $H_0$, the term $q_2(h)$ does not vanish and $Q(n)$ has degree at least 2.

For such $h$, we continue exactly as before. Applying the Van der Corput Lemma pointwise, we see
\begin{align*}
\sup_t\left|\frac{1}{N} \sum_{n=1}^N e^{2\pi i n t} \left[ \prod_{\eta \in V_{k-2}} c^{|\eta|}f_{a^\eta} \circ T_\alpha^{h \cdot \eta} \right](T_\alpha^n x) \right|^2 &= \sup_t\left|\sum_{n=1}^N e^{2\pi i n t} e^{2\pi i [Q(n)\alpha + \dot Q_x(n)]}\right|^2 \\
&\leq 2N + 4\sum_{m=1}^{N-1} \left|\sum_{n=1}^{N-m} e^{2\pi i [(Q(n) - Q(n+m))\alpha + (\dot Q_x(n) - \dot Q_x(n+m))]}\right| \, .
\end{align*}
Define $P(n)= Q(n) - Q(n+m)$ and $\dot P_x = \dot Q_x(n) - \dot Q_x(n+m)$. Notice that this finite difference operation decreases the degrees of $Q$ and $\dot Q$ by exactly one, as $m\neq 0$. Hence, $P(n)$ has degree between $1$ and $k-1$. Notably, $P(n)$ has $m$ dependence, but is still integer valued. By the Cauchy-Schwarz inequality, we arrive at
\begin{align*}
\sup_t\left|\frac{1}{N} \sum_{n=1}^N e^{2\pi i n t} \left[ \prod_{\eta \in V_{k-2}} c^{|\eta|}f_{a^\eta} \circ T_\alpha^{h \cdot \eta} \right](T_\alpha^n x) \right|^2 \leq \frac{2}{N} + 4\left( \frac{1}{N}\sum_{m=1}^{N-1} \left|\frac{1}{N}\sum_{n=1}^{N-m} e^{2\pi i [P(n)\alpha + \dot P_x(n)]}\right|^2\right)^{1/2} \, .
\end{align*}
Taking the integral of both sides in $x$, we have
\begin{align*}
\left\Vert \sup_t\left|\frac{1}{N} \sum_{n=1}^N e^{2\pi i n t} \left[ \prod_{\eta \in V_{k-2}} c^{|\eta|}f_{a^\eta} \circ T_\alpha^{h \cdot \eta} \right]\circ T_\alpha^n \right| \right\Vert_2^2 \leq \int \frac{2}{N} + 4\left( \frac{1}{N}\sum_{m=1}^{N-1} \left|\frac{1}{N}\sum_{n=1}^{N-m} e^{2\pi i [P(n)\alpha + \dot P_x(n)]}\right|^2\right)^{1/2} d\mu(x) \, .
\end{align*}
Let us denote $g_N(\alpha, h) := \left\Vert \sup_t \left|\frac{1}{N} \sum_{n=1}^N e^{2\pi i n t} \left[ \prod_{\eta \in V_{k-2}} c^{|\eta|}f_{a^\eta} \circ T^{h \cdot \eta} \right]\circ T_\alpha^n  \right|\right\Vert_2^2$. If we integrate both sides in $\alpha$, notice that since everything is uniformly bounded, we can use Fubini-Tonelli to switch the order of integration:
\begin{align*}
\int g_N(\alpha, h) \, d\alpha \leq \int \frac{2}{N} + 4\left( \frac{1}{N}\sum_{m=1}^{N-1}  \frac{1}{N^2} \int \left|\sum_{n=1}^{N-m} e^{2\pi i [P(n)\alpha + \dot P_x(n)] \alpha}\right|^2 \, d\alpha \right)^{1/2} d\mu(x) \, .
\end{align*}
On the inside integral, we expand as a double sum exactly as before:
\begin{align*}
\int \left|\sum_{n=1}^{N-m} e^{2\pi i [P(n)\alpha + \dot P_x(n)] \alpha}\right|^2 \, d\alpha &\leq \sum_{j=1}^N \left(\sum_{n=1}^{N} \left| \int e^{2\pi i (P(n) - P(j))\alpha  }\, d\alpha \right |  \right) \, .
\end{align*}
Recall that $P$ depends on $m, h$ and $x$. Since $h$ is fixed and we are looking at the above sum inside the integral in $x$ and the sum in $m$, the only nontrivial dependence of $P$ is in the variables $n$ and $j$. Inside the parenthesis, $j$ is fixed, and $P(n) - P(j)$ is an integer-valued polynomial of degree between $1$ and $k-1$. Hence, there are at most $k-1$ values of $n$ where $P(n) - P(j) = 0$. Since $\int e^{2\pi i n l \alpha}$ is 0 for $l \neq 0$ and 1 otherwise, it follows that the inner $n$ sum is bounded by $k-1$, and the whole term bounded by $(k-1)N$.

Substituting this into the original estimate, we see
\begin{align*}
\int g_N(\alpha, h) \, d\alpha \leq \int \frac{2}{N} + 4\left( \frac{1}{N}\sum_{m=1}^{N-1}  \frac{1}{N^2} ((k-1)N) \right)^{1/2} d\mu(x) \leq \frac{2}{N} + \frac{4\sqrt{k-1}}{N^{1/2}} \leq \frac{C}{N^{1/2}}
\end{align*}
for $C = 2 + 4\sqrt{k-1}$. Hence, if we take $N = M^4$, we have 
$$\int M^{1/2}g_{M^4}(\alpha, h) \, d\alpha \leq \frac{C}{M^{3/2}} \,.$$
Recall that this estimate holds for all $h\notin H_0$. If we bound $g_{M^4}(\alpha, h_0)$ trivially by 1 for all $h_0 \in H_0$, we get the estimate $\int M^{1/2}g_{M^4}(\alpha, h_0) \leq M^{1/2}$. Using our bound on the size of $H_0$, we see that

\begin{align*}
&\int \frac{M^{1/2}}{(M^2)^{k-2}}\sum_{h\in [M^2]^{k-2}} g_{M^4}(\alpha, h) \, d\alpha \\ &= \frac{1}{(M^2)^{k-2}}\sum_{h\in [M^2]^{k-2}}  \int M^{1/2}g_{M^4}(\alpha, h) \, d\alpha \\
&= \frac{1}{(M^2)^{k-2}}\sum_{h\in H_0 \cap [M^2]^{k-2}} \int M^{1/2}g_{M^4}(\alpha, h) \, d\alpha + \frac{1}{(M^2)^{k-2}}\sum_{h\in H_0^c \cap [M^2]^{k-2}} \int M^{1/2}g_{M^4}(\alpha, h) \, d\alpha\\
&\leq \frac{\#(H_0 \cap [M^2]^{k-2})}{(M^2)^{k-2}}M^{1/2} + \frac{\#(H_0^c \cap [M^2]^{k-2})}{(M^2)^{k-2}} \frac{C}{M^{3/2}} 
\\ &\leq \frac{(k-2)^2(M^2)^{k-3}}{(M^2)^{k-2}}M^{1/2} + \frac{\#([M^2]^{k-2})}{(M^2)^{k-2}} \frac{C}{M^{3/2}} \\ &\leq \frac{(k-2)^2 + C}{M^{3/2}} \, .
\end{align*}
The monotone convergence theorem then tells us that $$\int \sum_{M} \frac{M^{1/2}}{(M^2)^{k-2}}\sum_{h\in [M^2]^{k-2}} g_{M^4}(\alpha, h) \, d\alpha < \infty \, ,$$ so for almost all $\alpha$, the term $\frac{M^{1/2}}{(M^2)^{k-2}}\sum_{h\in [M^2]^{k-2}} g_{M^4}(\alpha, h)$ goes to zero, and is hence bounded by some constant $C(\alpha)$. So 
$$\frac{1}{(M^2)^{k-2}}\sum_{h\in [M^2]^{k-2}} g_{M^4}(\alpha, h) \leq \frac{C(\alpha)}{M^{1/2}}$$
holds for all $M$. Recalling what $g$ is, we have 
$$\frac{1}{(M^2)^{k-2}}\sum_{h\in [M^2]^{k-2}} \left\Vert \sup_t \left|\frac{1}{M^4} \sum_{n=1}^{M^4} e^{2\pi i n t} \left[ \prod_{\eta \in V_{k-2}} c^{|\eta|}f_{a^\eta} \circ T_\alpha^{h \cdot \eta} \right] \circ T_\alpha^n \right|\right\Vert_2^2\leq \frac{C(\alpha)}{M^{1/2}}\, .$$

To extend from $M^4$ to all $N$, consider that for any $N$ there exists $M$ with $M^4 \leq N < (M+1)^4$. Note that it is also the case that $M^2 \leq \lfloor \sqrt{N} \rfloor < (M+1)^2$. We first note that

\begin{align*}
&\frac{1}{\lfloor \sqrt{N} \rfloor^{k-2}}\sum_{h\in [\lfloor \sqrt{N} \rfloor]^{k-2}} \left\Vert \sup_t \left|\frac{1}{N} \sum_{n=1}^N e^{2\pi i n t} \left[ \prod_{\eta \in V_{k-2}} c^{|\eta|}f_{a^\eta} \circ T_\alpha^{h \cdot \eta} \right] \circ T_\alpha^n \right|\right\Vert_2^2 \\
&\leq \frac{1}{(M^2)^{k-2}}\sum_{h\in [M^2]^{k-2}} \left\Vert \sup_t \left|\frac{1}{N} \sum_{n=1}^N e^{2\pi i n t} \left[ \prod_{\eta \in V_{k-2}} c^{|\eta|}f_{a^\eta} \circ T_\alpha^{h \cdot \eta} \right]\circ T_\alpha^n \right|\right\Vert_2^2  \\
&+ \frac{1}{(M^2)^{k-2}}\sum_{h\in \{M^2 + 1, \dots, (M+1)^2\}^{k-2}} \left\Vert \sup_t \left|\frac{1}{N} \sum_{n=1}^N e^{2\pi i n t} \left[ \prod_{\eta \in V_{k-2}} c^{|\eta|}f_{a^\eta} \circ T_\alpha^{h \cdot \eta} \right]\circ T_\alpha^n  \right|\right\Vert_2^2 \, . 
\end{align*}
The total number of summands in the last term $(2M)^{k-2}$. Bounding each of them trivially, the whole term is bounded by $\frac{2^{k-2}}{M^{k-2}}$. For the remainder of the term, we apply the same line of reasoning inside the sum:

\begin{align*}
&\left\Vert \sup_t\left|\frac{1}{N}\sum_{n=1}^{N} e^{2\pi i n t}\left[ \prod_{\eta \in V_{k-2}} c^{|\eta|}f_{a^\eta} \circ T_\alpha^{h \cdot \eta} \right]\circ T_\alpha^n \right| \right\Vert_2^2 \\
&\left(\left\Vert \sup_t\left|\frac{1}{M^4}\sum_{n=1}^{M^4} e^{2\pi i n t} \left[ \prod_{\eta \in V_{k-2}} c^{|\eta|}f_{a^\eta} \circ T_\alpha^{h \cdot \eta} \right]\circ T_\alpha^n \right| \right\Vert_2 + \left\Vert \sup_t\left|\frac{1}{M^4}\sum_{n=M^4 + 1}^{(M+1)^4} e^{2\pi i n t} \left[ \prod_{\eta \in V_{k-2}} c^{|\eta|}f_{a^\eta} \circ T_\alpha^{h \cdot \eta} \right]\circ T_\alpha^n \right| \right\Vert_2 \right)^2 \\
&\leq \left(\left\Vert \sup_t\left|\frac{1}{M^4}\sum_{n=1}^{M^4} e^{2\pi i n t} \left[ \prod_{\eta \in V_{k-2}} c^{|\eta|}f_{a^\eta} \circ T_\alpha^{h \cdot \eta} \right]\circ T_\alpha^n \right| \right\Vert_2 + \frac{(M+1)^4-M^4-1}{M^4} \right)^2 \\ 
&\leq \left(\left\Vert \sup_t\left|\frac{1}{M^4}\sum_{n=1}^{M^4} e^{2\pi i n t} \left[ \prod_{\eta \in V_{k-2}} c^{|\eta|}f_{a^\eta} \circ T_\alpha^{h \cdot \eta} \right]\circ T_\alpha^n \right| \right\Vert_2 + \frac{14}{M} \right)^2 \\
&\leq 2\left\Vert \sup_t\left|\frac{1}{M^4}\sum_{n=1}^{M^4} e^{2\pi i n t} \left[ \prod_{\eta \in V_{k-2}} c^{|\eta|}f_{a^\eta} \circ T_\alpha^{h \cdot \eta} \right]\circ T_\alpha^n \right| \right\Vert_2^2 + 2\left(\frac{14}{M}\right)^2 \, .
\end{align*}
Notice that we could still bound $1/N$ by $1/M^4$ by factoring it in and out of the norm.
Applying these to our initial estimates, we see
\begin{align*}
&\frac{1}{\lfloor \sqrt{N} \rfloor^{k-2}}\sum_{h\in [\lfloor \sqrt{N} \rfloor]^{k-2}} \left\Vert \sup_t \left|\frac{1}{N} \sum_{n=1}^N e^{2\pi i n t} \left[ \prod_{\eta \in V_{k-2}} c^{|\eta|}f_{a^\eta} \circ T_\alpha^{h \cdot \eta} \right]\circ T_\alpha^n \right|\right\Vert_2^2 \\
&\leq \frac{1}{(M^2)^{k-2}}\sum_{h\in [M^2]^{k-2}} \left( 2 \left\Vert \sup_t \left|\frac{1}{M^4} \sum_{n=1}^{M^4} e^{2\pi i n t} \left[ \prod_{\eta \in V_{k-2}} c^{|\eta|}f_{a^\eta} \circ T_\alpha^{h \cdot \eta} \right]\circ T_\alpha^n \right|\right\Vert_2^2 + \frac{392}{M^2} \right) + \frac{2^{k-2}}{M^{k-2}} \\
&\leq \frac{2C(\alpha) + 392 + 2^{k-2}}{M^{1/2}} \leq \frac{4C(\alpha) + 784 + 2^{k-1}}{(M+1)^{1/2}} \leq \frac{4C(\alpha)+784 + 2^{k-1}}{N^{1/8}} \, ,
\end{align*}
and the $2/3$ exponent is accounted for by Hölder's inequality on averages, yielding an overall power of $(1/N^{1/8})^{1/3} = 1/N^{1/24}$. For each choice of a collection of $a^\eta$'s, we can find a full measure set where the above holds. Since there are countably many choices, we can find a set of full measure where all hold simultaneously. For each $\alpha$ in this set, the multilinearity conditions are satisfied and $(\T^k, T_{\alpha}, m)$ is a $k-1$th order WW system of order $1/24$ in $L^2$.
\end{proof}

\section{Bourgain's bound on double recurrence}\label{a:BB}

\begin{theorem}[Bourgain's bound on double recurrence]\label{BB2} Let $(X, \mathcal{F}, \mu, T)$ be an invertible dynamical system, and let $a_1, a_2 \in \Z$ be distinct and both nonzero. Then there exists $C > 0$ such that for every $N \in \N$ and $f_1, f_2\in L^\infty(\mu)$ for which $\max_{j=1, 2} \norm{f_j}_\infty \leq 1$, we have
$$\left \Vert \frac{1}{N} \sum_{n=1}^N f_1 \circ T^{a_1n} \cdot f_2 \circ T^{a_2n} \right\Vert_1 \leq C\left( \frac{1}{N} + \left\Vert \sup_t \left| \frac{1}{N} \sum_{n=1}^N e^{2\pi i n t} f_1 \circ T^n \right|\right\Vert_1^{2/3} \right) \, .$$
\end{theorem}

As mentioned before, the relevant ideas and techniques originate from the beginning of Bourgain's paper on double recurrence \cite{bourgain}. The structure and presentation of this proof is modeled after the argument by Assani \cite{AssaniWW}.

The proof also uses the following lemma:
\begin{lemma}\label{l:power}
Let $(X, \mathcal{F}, \mu, T)$ be an invertible dynamical system, $f\in L^\infty(\mu)$, $p \in [1, \infty]$, and $a\in \Z$ nonzero. Then
$$\left\Vert \sup_{t} \left|\frac{1}{N} \sum_{n=1}^N e^{2\pi i n t} f \circ T^{an}  \right| \right\Vert_p \leq |a| \left\Vert \sup_{t} \left|\frac{1}{N} \sum_{n=1}^N e^{2\pi i n t} f\circ T^{n} \right| \right\Vert_p  \, .$$
\end{lemma}

\begin{proof} 

Without loss of generality, take $a>0$, as we can apply this same lemma to $T^{-1}$. Pointwise, we may note that
\begin{align*}
\sup_{t} \left|\frac{1}{N} \sum_{n=1}^N e^{2\pi i n t} f(T^{an}x) \right| &= \sup_{t} \left|\frac{1}{N} \sum_{n=1}^{aN} e^{2\pi i n (t/a)} f(T^{n}x) \left(\frac{1}{a}\sum_{k=1}^a e^{2\pi i n k/a} \right)\right| \\
&\leq \frac{1}{a}\sum_{k=1}^a \sup_{t} \left|\frac{1}{N} \sum_{n=1}^{aN} e^{2\pi i n (t/a + k/a)} f(T^{n}x)\right| \\
&= \sup_{t} \left|\frac{1}{N} \sum_{n=1}^{aN} e^{2\pi i n t} f(T^{n}x) \right| \, ,
\end{align*}
as the $k$ dependence vanishes under the supremum. Taking the $p$-norm of both sides, it follows that 
\begin{align*}
\left\Vert \sup_{t} \left|\frac{1}{N} \sum_{n=1}^N e^{2\pi i n t} f\circ T^{an} \right|\right\Vert_p &\leq \left\Vert \sup_{t} \left|\frac{1}{N} \sum_{n=1}^{aN} e^{2\pi i n t} f\circ T^{n} \right| \right\Vert_p 
\\&= \left\Vert \sup_{t} \left|\sum_{l = 0}^{a-1}\frac{1}{N} \sum_{n=1}^{N} e^{2\pi i(lN + n) t} f\circ T^{lN + n} \right| \right\Vert_p \\
&\leq \sum_{l = 0}^{a-1}\left\Vert \sup_{t} \left|\frac{1}{N} \sum_{n=1}^{N} e^{2\pi i n t} f \circ T^{lN + n} \right| \right\Vert_p \\
&= a \left\Vert \sup_{t} \left|\frac{1}{N} \sum_{n=1}^{N} e^{2\pi i n t} f \circ T^{n} \right| \right\Vert_p \, ,
\end{align*}
as the $l$ dependence is lost under the norm.
\end{proof}

\begin{remark} If $a>0$, then invertibility is not necessary.
\end{remark}

\begin{proof}[Proof of theorem] The idea for this proof originates from the first two pages of Bourgain's paper on double recurrence.
Define the function
$$A_N(f_1, f_2)(x) = \frac{1}{N}\sum_{n=1}^N f_1(T^{a_1n} x) f_2(T^{a_2 n} x) \, .$$
Let $1 < N_1 < N$ be integers. Pointwise, we can verify that for any $1\leq k \leq N_1$ we have
\begin{align*}
|A_N(f_1, f_2) - A_N(f_1 \circ T^{a_1k}, f_2 \circ T^{a_2k})| &\leq \frac{2N_1}{N}\Vert f_1 \Vert_\infty \Vert f_2 \Vert_\infty \, .
\end{align*}
Hence, the same holds true if we take the integral (1-norm) of both sides. Using this, we see that
\begin{align*}
\Vert A_N(f_1, f_2) \Vert_1 &= \left\Vert \sum_{k=1}^{N_1} \frac{1}{N_1}A_N(f_1, f_2) - \sum_{k=1}^{N_1} \frac{1}{N_1}A_N(f_1 \circ T^{a_1k}, f_2 \circ T^{a_2k}) + \sum_{k=1}^{N_1} \frac{1}{N_1}A_N(f_1 \circ T^{a_1k}, f_2 \circ T^{a_2k})\right\Vert_1 \\
&\leq \frac{1}{N_1}\sum_{k=1}^{N_1} \Vert A_N(f_1, f_2) - A_N(f_1 \circ T^{a_1k}, f_2 \circ T^{a_2k})\Vert_1 + \frac{1}{N_1} \sum_{k=1}^{N_1} \Vert A_N(f_1 \circ T^{a_1k}, f_2 \circ T^{a_2k} )\Vert_1 \\
&\leq \frac{1}{N_1}\sum_{k=1}^{N_1} \frac{2N_1}{N} \Vert f_1 \Vert_\infty \Vert f_2 \Vert_\infty + \frac{1}{N_1} \sum_{k=1}^{N_1} \Vert A_N(f_1, f_2 \circ T^{(a_2-a_1)k} )\Vert_1 \\ 
&= \frac{2N_1}{N} \Vert f_1 \Vert_\infty \Vert f_2 \Vert_\infty + \left\Vert\frac{1}{N_1} \sum_{k=1}^{N_1}  \left|A_N(f_1, f_2 \circ T^{(a_2-a_1)k})\right|\right\Vert_1  \\ 
&\leq \frac{2N_1}{N} \Vert f_1 \Vert_\infty \Vert f_2 \Vert_\infty + \left\Vert \left(\frac{1}{N_1} \sum_{k=1}^{N_1}  \left|A_N(f_1, f_2 \circ T^{(a_2-a_1)k})\right|^2 \right)^{1/2}\right\Vert_1 
\end{align*}
by the triangle inequality and the measure preserving norm. The final line follows by Hölder's inequality on sums.

So we are interested in the functions
$$A_N(f_1, f_2\circ T^{(a_2-a_1)k})(x) = \frac{1}{N} \sum_{n=1}^N f_1(T^{a_1n}x)f_2(T^{a_2n+(a_2-a_1)k}x) \, .$$
Consider expanding the following product of sums and ($L^2$ in the $t$ variable) inner product:

$$\left \langle \left(\sum_{n=1}^{N} e^{2\pi i n a_2t} f_1(T^{a_1n} x)\right)\left( \sum_{m=-N(2|a_2| + |a_1|)}^{N(2|a_2| + |a_1|)} e^{-2\pi i m t} f_2(T^{m} x) \right), e^{2\pi i (a_1-a_2)k t} \right \rangle \, .$$
Since the $e^{2\pi i l t}$ terms for $l \in \mathbb{Z}$ are an orthonormal set, the only terms that can survive are where $e^{2\pi i (a_2n-m) t} = e^{2\pi i (a_1-a_2)k t}$, or $m = a_2n+(a_2-a_1)k$. Hence, for every summand in $n$, there is at most one summand in $m$ that can survive. 
Recall that $1 \leq n \leq N $ and $1 \leq k\leq N_1 < N$. Hence, all possible values of $a_1n + (a_2 - a_1)k$ lie between $-N(2|a_2| + |a_1|)$ and $N(2|a_2| + |a_1|)$.  Because of the bounds that appear on $m$, every such value that pairs with an $n$ to survive does appear.
For each $n$, the product of summands that does not disappear in the inner product has coefficient $f_1(T^{a_1n}x)f_2(T^{a_2n+(a_2 - a_1)k}x)$, so the above term expands to be exactly $\sum_{n=1}^N f_1(T^{a_1n}x) f_2(T^{a_2n+(a_2-a_1)k}x)$, which is the sum we are interested in. Since $N$ is fixed, the $1/N$ term on the outside is a constant. 

As we look at the terms $A_N(f_1, f_2 \circ T^{(a_2 - a_1)k})$ for $k = 1, \dots, N_1$, the above reasoning shows that each is a Fourier coefficient of the product of sums. By Bessel's inequality, when we square sum these terms over any $x$, we are bounded above by the squared $L^2$ norm in $t$, which is an integral from 0 to 1:

\begin{align*}
&\sum_{k=1}^{N_1} |A_N(f_1, f_2\circ T^{(a_2-a_1)k})(x)|^2\\ &= \frac{1}{N^2} \sum_{k=1}^{N_1} \left|\sum_{n=1}^N f_1(T^{a_1n}x) f_2(T^{a_2n+(a_2-a_1)k}x)\right|^2 \\
&\leq \frac{1}{N^2} \sum_{k=-\infty}^{\infty} \left|\sum_{n=1}^N f_1(T^{a_1n}x) f_2(T^{a_2n+(a_2 - a_1)k}x)\right|^2 \\
&\leq \frac{1}{N^2} \int_0^1 \left|\sum_{n=1}^{N} e^{2\pi i n a_2t} f_1(T^{a_1n} x) \right|^2 \left| \sum_{m=-N(2|a_2| + |a_1|)}^{N(2|a_2| + |a_1|)} e^{-2\pi i m t} f_2(T^{m} x)\right|^2dt \\
&\leq  \left(\sup_{t} \left|\frac{1}{N}\sum_{n=1}^{N} e^{2\pi i n a_2t} f_1(T^{a_1n} x) \right|^2 \right)\int_0^1\left| \sum_{m=-N(2|a_2| + |a_1|)}^{N(2|a_2| + |a_1|)} e^{-2\pi i m t} f_2(T^{m} x)\right|^2dt \\
&\leq  2N(2|a_2| + |a_1|+1) \Vert f_2 \Vert_\infty^2\left(\sup_{t} \left|\frac{1}{N}\sum_{n=1}^{N} e^{2\pi i n t} f_1(T^{a_1n} x) \right| \right)^2
\end{align*}
(note that at the end the $a_2$ is lost in the supremum in $t$). Call $C = 4|a_2| + 2|a_1| + 2$. If we divide by $N_1$, take the square root, and integrate, we have
\begin{align*}
\left\Vert \left( \frac{1}{N_1}\sum_{k=1}^{N_1} |A_N(f_1, f_2\circ T^{(a_2-a_1)k})|^2\right)^{1/2} \right\Vert_1 &\leq C^{1/2}\left(\frac{N}{N_1}\right)^{1/2}\Vert f_2 \Vert_\infty\left\Vert\sup_{t} \left|\frac{1}{N}\sum_{n=1}^{N} e^{2\pi i n t} f_1 \circ T^{a_1n} \right| \right\Vert_1 \\
&\leq C^{1/2}\left(\frac{N}{N_1}\right)^{1/2}\Vert f_2 \Vert_\infty |a_1|\left\Vert\sup_{t} \left|\frac{1}{N}\sum_{n=1}^{N} e^{2\pi i n t} f_1 \circ T^{n} \right| \right\Vert_1
\end{align*}
using Lemma \ref{l:power}.

Call $C' = C^{1/2}|a_1|$. Combining our two estimates, we get
\begin{align*}
\Vert A_N(f_1, f_2) \Vert_1 &\leq \Vert f_2 \Vert_\infty\max\{2\Vert f_1\Vert_\infty, C'\}\left( \frac{N_1}{N} + \left(\frac{N}{N_1}\right)^{1/2} \left\Vert\sup_{t} \left|\frac{1}{N}\sum_{n=1}^{N} e^{2\pi i n t} f_1 \circ T^n \right| \right\Vert_1\right)
\end{align*}
or
\begin{align*}
\Vert A_N(f_1, f_2) \Vert_1 &\leq C'\left( \frac{N_1}{N} + \left(\frac{N}{N_1}\right)^{1/2} \left\Vert\sup_{t} \left|\frac{1}{N}\sum_{n=1}^{N} e^{2\pi i n t} f_1\circ T^n  \right| \right\Vert_1\right) \, .
\end{align*}

Set $\delta = \left\Vert\sup_{t} \left|\frac{1}{N}\sum_{n=1}^{N} e^{2\pi i n t} f_1 \circ T^n  \right| \right\Vert_1$, which is less than 1, as $f_1$ is bounded by 1. If $N\delta^{2/3} \geq1$, choose $N_1 = \lfloor N\delta^{2/3} \rfloor$. Then the previous estimate becomes
\begin{align*}
\Vert A_N(f_1, f_2) \Vert_1 &\leq C'\left( \frac{\lfloor N \delta^{2/3}\rfloor}{N} + \left(\frac{N}{\lfloor N \delta^{2/3} \rfloor}\right)^{1/2} \delta\right) \, .
\end{align*}
Since $\frac{\lfloor a \rfloor}{a} \leq 1$ and $\frac{a}{\lfloor a \rfloor} \leq 2$ for all $a \geq 1$, we see that
\begin{align*}
\Vert A_N(f_1, f_2) \Vert_1 &\leq C'\left( \frac{\lfloor N \delta^{2/3}\rfloor}{N} + \left(\frac{N}{\lfloor N \delta^{2/3} \rfloor}\right)^{1/2} \delta\right) \\
&= C'\left( \delta^{2/3}\frac{\lfloor N \delta^{2/3}\rfloor}{N\delta^{2/3}} + \left(\frac{N\delta^{2/3}}{\lfloor N \delta^{2/3} \rfloor}\right)^{1/2} \delta^{2/3}\right) \\
&\leq C'(1 + \sqrt{2})\delta^{2/3} = C'(1 + \sqrt{2})\left\Vert\sup_{t} \left|\frac{1}{N}\sum_{n=1}^{N} e^{2\pi i n t} f_1\circ T^n \right| \right\Vert_1^{2/3} \, .
\end{align*}
If $N\delta^{2/3} < 1$, then $\delta < \frac{1}{N^{3/2}}$. Hence, if we pick $N_1 = 1$, we see that
\begin{align*}
\Vert A_N(f_1, f_2) \Vert_1 &\leq C'\left( \frac{1}{N} + N^{1/2}  \delta\right) \leq C' \left( \frac{1}{N} + N^{1/2}\frac{1}{N^{3/2}}\right) \leq \frac{2C'}{N} \, .
\end{align*}
So $\Vert A_N(f_1, f_2) \Vert_1$ is bounded by one of the above terms or the other, hence the sum, and we get the claimed bound:

$$\Vert A_N(f_1, f_2) \Vert_1 \leq C''\left(\frac{1}{N} +  \left\Vert \sup_{t} \left|\frac{1}{N} \sum_{n=1}^N f_1 \circ T^n e^{2\pi i n t} \right| \right\Vert_1^{2/3} \right) \, ,$$
where $C'' = (1 + \sqrt{2})C' = (1 + \sqrt{2})(4|a_2| + 2|a_1| + 2)^{1/2}|a_1|\, .$ 
\end{proof}

\begin{remark}
So long as $a_1, a_2$ are positive and $a_1 < a_2$, invertibility is not necessary.
\end{remark}
\begin{remark}
From the same ideas, it is possible to get a stronger bound, without the $1/N$ term. But this bound only holds true for sufficiently large $N$, depending on $f_1$.
\end{remark}

\printbibliography
\end{document}